\documentclass{article}

\usepackage{amsmath, amsthm, amssymb, dsfont, mathtools}
\usepackage{mathrsfs}

\usepackage{graphicx}
\usepackage{threeparttable}
\usepackage{verbatim}
\usepackage{natbib}
\usepackage{caption}
\usepackage{subcaption}
\usepackage{fancyvrb}
\usepackage{enumerate}
\usepackage[inline]{enumitem}
\usepackage{relsize}
\usepackage{xcolor}
\usepackage{microtype}
\usepackage{hyperref}
\usepackage[margin=1.4in]{geometry}
\hypersetup{
  colorlinks,citecolor=blue,urlcolor=blue,linkcolor=blue,
  pdftitle={Compound decisions and empirical Bayes via Bayesian nonparametrics},
  pdfauthor={Nikolaos Ignatiadis and Sid Kankanala},
  pdfkeywords={Bayesian nonparametrics, Dirichlet process, Gaussian process, Gaussian compound decisions}
}
\usepackage{diagbox}
\usepackage{nikos_tex}
\usepackage{float}

\usepackage{multirow}
\usepackage{makecell}
\usepackage{booktabs}
\usepackage[utf8]{inputenc}
\usepackage[english]{babel}

\graphicspath{{./figures/}{./}}

\theoremstyle{definition}
\newtheorem{prop}{Proposition}
\newtheorem{defi}[prop]{Definition}

\newtheorem{lemm}[prop]{Lemma}
\newtheorem{theo}[prop]{Theorem}
\newtheorem{rema}{Remark}

\newtheorem{assu}[prop]{Assumption}
\newtheorem{spec}[prop]{Prior Specification}

\date{September, 2026}

\title{Compound decisions and empirical Bayes\\ via Bayesian nonparametrics}

\author{
Nikolaos Ignatiadis\\
University of Chicago\\
\texttt{ignat@uchicago.edu}
\and
Sid Kankanala\\
University of Chicago\\
\texttt{sid.kankanala@chicagobooth.edu}
}

\allowdisplaybreaks[1]

\begin{document}

\maketitle

\begin{abstract}
We study compound decision theory from a nonparametric Bayesian  perspective, with particular emphasis on their relationship to empirical Bayes (EB) procedures. Motivated by the sharp risk guarantees available for EB procedures based on the nonparametric maximum likelihood estimator (NPMLE), we investigate whether analogous guarantees can be established for fully Bayesian decision rules. In a class of Gaussian compound decision problems, we show that the fully Bayesian posterior mean achieves near-optimal risk. Moreover, it is admissible as a genuine Bayes rule, whereas the corresponding NPMLE plug-in rule is inadmissible.  Simulations illustrate the  performance of nonparametric Bayes procedures relative to common alternatives. As an application, we apply our methodology to Census tract-level estimates of economic mobility from the \emph{Opportunity Atlas}.
\end{abstract}


\section{Introduction}
\label{sec:intro}

Empirical Bayes (EB) procedures~\citep{robbins1956empirical,efron2019bayes,walters2024empirical, koenker2026empirical} are widely used in applied economics. They provide a principled framework for compound decision problems by sharing information across a large collection of related problems. Common applications include estimating teacher value-added~\citep{gilraine2020new,rose2022effects}, treatment effects~\citep{azevedo2019empirical, azevedo2020testing}, and healthcare quality~\citep{gu2023invidious,chan2026productivity}, measuring discrimination~\citep{kline2024discrimination}, and panel forecasting~\citep{liu2020forecasting}.

The standard EB construction begins with an oracle Bayes rule whose unknown components include a latent distribution (the empirical Bayesian's ``prior'') and, in some settings, additional nuisance parameters. EB estimates these components, often fully nonparametrically, and substitutes the resulting estimates into the oracle rule. A central appeal of this framework is that the resulting decision rules can enjoy strong denoising guarantees in some settings such as Gaussian compound decision problems, where the goal is to estimate latent parameters under squared-error loss~\citep{jiang2009general, chen2026empirical, chen2026sharp}. At the same time, a classical objection to EB for decision making~\citep{deely1981bayes,morris1983parametric,berger1985statistical,zellner1985bayesian} is that it treats the estimated latent distribution and nuisance parameters as known, thereby ignoring uncertainty about the shrinkage structure that may be relevant for subsequent decisions. Fully nonparametric Bayesian modeling offers a natural alternative: the prior can provide useful regularization, particularly when the number of related problems is moderate, by expressing preferences over features such as smoothness and model complexity while retaining nonparametric flexibility, and the posterior propagates uncertainty about the shrinkage structure into the latent parameters and subsequent decisions based on them.

 In our view, the preceding considerations are  relevant in  modern EB applications. Indeed, many important economic settings involve only a moderate number of related problems, for example $n=108$ firms in~\citet{kline2022systemic}, $n=68$ policies in~\citet{moon2025optimal}, and $n=75$ nurse practitioners in~\citet{chan2026productivity}. At the same time, EB estimates are increasingly used as inputs into downstream decision problems that go beyond the denoising criterion for which the original procedure is justified. In such settings, uncertainty about the estimated shrinkage structure can remain important even when point estimation is accurate. For example, in interval inference, the coverage error of plug-in credible intervals may decay only at a logarithmic rate~\citep{han2026empirical}.\footnote{This reflects the slow rates associated with ill-posed deconvolution problems, in contrast to the much faster rates available for denoising.}

Motivated by these concerns, we study the natural fully Bayesian approach to compound decision problems, in which the unknown shrinkage structure is assigned a suitable nonparametric prior. The resulting hierarchical model induces a genuine Bayes rule that averages over posterior uncertainty about this structure rather than conditioning on a single estimate. More broadly, it yields a joint posterior for the latent parameters and the unknown components of the shrinkage rule, allowing this uncertainty to be carried into subsequent decision problems. As we show, the resulting Bayes rule also enjoys attractive decision-theoretic properties, including admissibility. The central question is whether these benefits can be obtained while retaining the tractability and strong frequentist risk performance that have made EB methods so attractive. We investigate this question in detail for a class of Gaussian compound decision problems.

To fix ideas, consider a setting in which a researcher observes an estimate $Z_i$ of a unit-specific latent parameter $\mu_i$, together with a known auxiliary quantity $\nu_i$ that characterizes its sampling distribution. For example, $\nu_i$ may represent a standard error or another measure of estimation precision. We assume that
\begin{equation}
\tag{CD}
Z_i \cond \mu_i,\nu_i \;\simindep\; p(\cdot \mid \mu_i,\nu_i),
\qquad i=1,\ldots,n,
\label{eq:hierarchy1}
\end{equation}
where $p(\cdot \mid \mu_i,\nu_i)$ is the known conditional distribution of $Z_i$ given $(\mu_i, \nu_i)$. Let $\boldZ=(Z_1,\ldots,Z_n)$, $\boldnu=(\nu_1,\ldots,\nu_n)$, and $\boldmu=(\mu_1,\ldots,\mu_n)$ denote the observed estimates, auxiliary quantities, and unknown unit-specific parameters, respectively. The compound decision problem is to estimate $\boldmu$ from the observed data $(\boldZ,\boldnu)$ using a decision rule $\widehat{\boldmu}=\boldt(\boldZ,\boldnu)$. To regularize the problem, empirical Bayes procedures typically supplement the sampling model in \eqref{eq:hierarchy1} with a second-layer ``working model'' for the latent parameters. A common specification is
\begin{equation} \tag{B} \mu_i \mid \nu_i,G,\eta \;\simindep\; \mathcal{T}_{\eta}(\nu_i,G), \qquad i=1,\ldots,n, \label{eq:hierarchy2} \end{equation}
where $G$ is an unknown latent distribution and $\eta$ is an unknown, possibly infinite-dimensional, nuisance component. Here, $\mathcal{T}_{\eta}(\nu_i,G)$ denotes the distribution induced by transforming $G$ as a function of $\nu_i$, according to a rule known up to $\eta$. This transformation is typically chosen to provide some flexibility in modeling the latent structure, while preserving the tractability of the conditional law of $\mu_i$. By a ``working model", we mean that \eqref{eq:hierarchy2} is introduced to regularize the compound decision problem and need not be part of the true data-generating process. If \eqref{eq:hierarchy1} and \eqref{eq:hierarchy2} were both true, then the corresponding oracle Bayes estimator under squared-error loss would be
$$
\widehat{\boldmu}^{\mathrm{B}}(G,\eta)
=
\mathbb{E}_{G,\eta}\!\left[\boldmu \mid \boldZ,\boldnu\right].
$$
Empirical Bayes methods proceed by estimating $G$ and $\eta$ from the data and plugging in the resulting estimates $(\widehat G,\widehat\eta)$ to obtain an estimated oracle rule $\widehat{\boldmu}^{\mathrm{B}}(\widehat{G},\widehat{\eta})$. By contrast, we consider a fully Bayesian treatment in which $(G,\eta)$ are assigned a  nonparametric prior,
$
(G,\eta)\sim\Pi.
$
Let $\Pi(\cdot\mid\boldZ,\boldnu)$ denote the resulting posterior distribution. Under squared-error loss, the associated Bayes rule for the compound decision problem is the posterior mean
\begin{equation}
\widehat{\boldmu}(\Pi) = \mathbb{E}_{\Pi}\!\left[\boldmu \mid \boldZ,\boldnu\right] = \int \widehat{\boldmu}^{\mathrm{B}}(G,\eta)\, \dd\Pi\!\left(G,\eta \mid \boldZ,\boldnu\right).
\label{eq:bayes_postmean}
\end{equation}
The posterior also provides uncertainty quantification, for example through credible intervals for the individual parameters $\boldmu$. In this setting, the prior specification also allows information about the distributional structure of the latent parameters and the form of the nuisance components to be incorporated directly into the analysis. By averaging over the posterior distribution of $(G,\eta)$, the fully Bayesian approach propagates uncertainty about the shrinkage structure into the decision rule for $\boldmu$. The resulting rule is a genuine Bayes rule and, under standard conditions, inherits a variety of attractive decision-theoretic properties such as admissibility. Moreover, for the broad class of compound decision problems commonly encountered in the literature, the hierarchy in \eqref{eq:hierarchy1} and \eqref{eq:hierarchy2} provides a tractable likelihood for $(G,\eta)$, and so computation of the posterior  $\Pi(\cdot\mid\boldZ,\boldnu)$ is also tractable.

The main theoretical contributions of this paper are as follows. First, we study the canonical homoscedastic Gaussian means setting under a  Dirichlet process Gaussian mixture (DPGM) prior for $G$. We establish contraction rates for the nonparametric Bayes posterior and frequentist risk guarantees for the resulting Bayes rule under the compound decision sampling model \eqref{eq:hierarchy1} alone, without requiring the latent working model \eqref{eq:hierarchy2} to be correctly specified. This result addresses an earlier question by \citet{vandervaart2019comment} concerning the potential viability of Bayesian nonparametric procedures for compound decision problems wherein the parameters $\boldmu$ are treated as fixed (that is, the working model~\eqref{eq:hierarchy2} is allowed to be  misspecified).
Second, we show that the conventional plug-in rule based on the nonparametric maximum likelihood estimator (NPMLE) is inadmissible in this setting, highlighting an important decision-theoretic distinction between plug-in empirical Bayes and fully Bayesian rules. Third, we consider a heteroscedastic Gaussian means model with a nonparametric conditional location-scale structure, in which $G$ is assigned a DPGM prior and the unknown location and scale functions collected in $\eta$ are assigned Gaussian process (GP) priors. In this setting, we provide a fully nonparametric Bayesian treatment of the model in \citet{chen2026empirical} and establish risk guarantees for the resulting Bayes rule.

As an application, we study Census tract-level estimates of economic mobility from the ``Opportunity Atlas''~\citep{chetty2026opportunity}, revisiting earlier analyses by~\citet{bergman2024creating} and~\citet{chen2026empirical}. We first analyze the Opportunity Atlas data by revisiting the comparison of Census tracts considered by~\citet{chen2026empirical}. Although the Bayesian and plug-in empirical Bayes procedures yield similar point estimates, the Bayesian posterior gives a substantially less decisive assessment of the relative ranking of the tracts once uncertainty about the estimated shrinkage structure is taken into account. We then complement the empirical analysis with the calibrated simulation designs of~\citet{chen2026empirical}, which draw on fifteen Opportunity Atlas outcomes. These experiments show that the fully Bayesian procedure is competitive with modern plug-in empirical Bayes methods for point estimation, while also providing useful uncertainty assessments across the calibrated designs. More broadly, the application illustrates how the Bayesian framework can combine competitive point estimation with a posterior assessment of uncertainty that can inform subsequent decisions involving the latent parameters.

The paper is organized as follows. Section~\ref{sec:homoscedastic_gaussian} develops the main ideas in the classical homoscedastic Gaussian compound decision problem. Section~\ref{sec-hetero} extends the analysis to heteroscedastic models in which the latent distribution may vary nonparametrically with estimation precision. Section~\ref{sec-sims} provides simulation evidence on the performance of  Bayes procedures relative to common empirical Bayes alternatives. Section~\ref{sec-atlas} applies the methodology to Census tract-level estimates of economic mobility from the Opportunity Atlas, considering both calibrated simulation designs based on its fifteen outcomes and the observed tract-level data. Section~\ref{conclude} concludes, while Section~\ref{sec:proof_ideas} describes some of the main proof ideas.

\subsection{Literature review: Compound decisions via hierarchical Bayes} \label{lit}

In this literature review we focus on previous proposals to use nonparametric hierarchical Bayes to solve compound decision problems that are usually solved using empirical Bayes. We make further connections with the EB literature throughout the manuscript.

The idea of using a fully Bayes approach to compound decision problems dates back to the early days of EB.
To wit, the idea already appears in the paper that introduced compound decision theory~\citep{robbins1951asymptotically}. Robbins studies model~\eqref{eq:hierarchy1} with $\boldmu \in \cb{\pm 1}^n$ and $Z_i \mid \mu_i \sim \mathrm{N}(\mu_i,1)$. He proposes to estimate $\boldmu$ by first using $\boldZ$ to estimate $p_n := \{\# i: \mu_i=1\}/n$, and then applying the Bayes rule for the model in which \smash{$(\mu_i + 1)/2 \simiid \mathrm{Bernoulli}(\hat{p}_n)$}. Robbins however acknowledges that his proposed estimator is not admissible, and moreover, he suggests a ``possible candidate for a rule superior to [his]'' by pursuing a fully Bayes approach in which one places a prior on \smash{$\boldmu \in \cb{\pm 1}^n$}. This suggestion is pursued by~\citet{gilliland1976asymptotica}.
An important precursor to our main result is~\citet{datta1991asymptotic}, who establishes asymptotic optimality of fully Bayes procedures for compound decision problems in compact exponential families.

An early argument for using nonparametric Bayes via the Dirichlet process~\citep{ferguson1973bayesian} for EB is given by~\citet{antoniak1974mixtures} for the Gaussian decision problem.  The same idea is pursued by~\citet{berry1979empirical} for the binomial EB problem;~\citet{gu2017empirical} implement such an approach and find prediction performance comparable to that of the NPMLE.
~\citet{weinstein2025nonparametric} carefully explain why fully Bayesian procedures are well-suited for sequence problems (when only~\eqref{eq:hierarchy1} holds, without requiring the working model~\eqref{eq:hierarchy2} to be correctly specified) and show promising empirical performance with a nonparametric prior given by Pólya Trees.\footnote{These empirical results refer to an early version of this paper, circulated under the title ``Hierarchical Bayes modeling for large-scale inference.''
}
Assuming both~\eqref{eq:hierarchy1} and~\eqref{eq:hierarchy2} govern data generation,~\citet{cannella2026universal} provide regret guarantees for the Gaussian and Poisson EB problems, and~\citet{favaro2026merging} for the Poisson EB problem.
\citet{ghosal2022discussion} suggests using BNP to construct confidence intervals for empirical Bayes estimands such as the posterior mean as an alternative to the confidence intervals of~\citet{ignatiadis2022confidence}. \citet{ignatiadis2025partially} develop a fully Bayesian alternative to the empirical partially Bayes testing approach of~\citet{ignatiadis2025empirical}.

\citet*{liu2020forecasting} develop an empirical Bayes approach for panel forecasting using Tweedie's formula and show asymptotically vanishing regret. Subsequent works by these authors on panel forecasting~\citep*{liu2021panel, liu2023forecasting,   liu2023density} instead adopt a fully Bayesian approach stating the latter is preferable over EB ``for interval and density forecasts, because it captures all sources of uncertainty.''

\section{Compound decisions}
\label{sec:homoscedastic_gaussian}

\subsection{Setup}
In this section, we begin with the classical homoscedastic Gaussian compound decision problem~\citep{johnstone2019gaussian}. For simplicity, we normalize the sampling variance to one, so that the sampling model \eqref{eq:hierarchy1} takes the form
\begin{equation}
\tag{CD-$\mathrm{N}$}
Z_i \cond \mu_i \;\simindep\; \mathrm{N}(\mu_i,1),
\qquad i=1,\ldots,n,
\label{eq:hierarchy1N}
\end{equation}
Since the sampling distribution is identical across units, there is no unit-specific auxiliary parameter $\boldnu$. The compound decision problem is to estimate $\boldmu=(\mu_1,\ldots,\mu_n)$ from the observations $\boldZ=(Z_1,\ldots,Z_n)$ using a decision rule $\widehat{\boldmu}=\boldt(\boldZ)$. To regularize the problem, empirical Bayes procedures commonly introduce the working model
\begin{equation}
\tag{B-$\mathrm{N}$}
\mu_i \cond G_{\star} \;\simiid\; G_{\star},
\qquad i=1,\ldots,n,
\label{eq:hierarchy2N}
\end{equation}
where $G_{\star}$ is an unknown latent distribution. We emphasize that \eqref{eq:hierarchy2N} is only a working model and need not describe the true data-generating process. If it were correctly specified, then under squared-error loss the corresponding oracle Bayes rule would be
\begin{equation}
\widehat{\boldmu}^{\mathrm{B}}(G_{\star})
=
\big(
\delta_{G_\star}(Z_1),\ldots,\delta_{G_\star}(Z_n)
\big),
\qquad
\delta_G(z)
:=
\frac{\textstyle\int \mu \varphi(z-\mu)\,\dd G(\mu)}
{\textstyle\int \varphi(z-\mu)\,\dd G(\mu)},
\label{eq:cd_b_post_mean}
\end{equation}
where $\varphi$ denotes the standard Gaussian density and $\delta_G$ denotes the corresponding Bayes rule (in the univariate component problem) for a generic probability distribution $G$ on $\RR$.

In the classical $G$-modeling approach to empirical Bayes~\citep{efron2014two}, one constructs an estimator \smash{$\widehat G $} and then estimates each $\mu_i$ using the plug-in rule \smash{$\widehat{\mu}_i^{\mathrm{EB}}=\delta_{\widehat G}(Z_i)$}. In particular, when \smash{$\widehat G$} is the nonparametric maximum likelihood estimator (NPMLE)~\citep{robbins1950generalization,kiefer1956consistency}, \citet{jiang2009general} and~\citet{chen2026sharp} establish remarkably strong regret guarantees for the resulting EB rule. The main message of this section is that classical Bayesian nonparametric procedures can achieve comparable regret guarantees under the compound decision model \eqref{eq:hierarchy1N} alone, while also retaining a variety of attractive decision-theoretic properties that the plug-in rule lacks.

Throughout this section, we write $\PP[\boldmu]{\cdot}$ and $\EE[\boldmu]{\cdot}$ for probability and expectation under the sampling model \eqref{eq:hierarchy1N}, with $\boldmu$ treated as fixed.

\subsection{Bayesian nonparametrics}
We consider a nonparametric Bayesian who places a prior $\Pi$ on $G$ and forms the posterior $\Pi(\cdot \mid \boldZ)$ under the potentially misspecified hierarchy \eqref{eq:hierarchy1N}--\eqref{eq:hierarchy2N}. As noted by \citet{vandervaart2019comment}, a natural question is whether such Bayesian nonparametric procedures admit meaningful frequentist guarantees under the compound decision model \eqref{eq:hierarchy1N}, with $\boldmu$ treated as fixed and without requiring the working model \eqref{eq:hierarchy2N} to be correctly specified. This raises two questions. First, does the posterior $\Pi(\cdot \mid \boldZ)$ contract under the fixed-$\boldmu$ sampling model, and, if so, toward what target and in what metric? Second, what frequentist risk guarantees can be established for the Bayes rule induced by this posterior, and how do they compare with those of classical empirical Bayes procedures?

To address the preceding questions in further detail, we first specify the nonparametric prior $\Pi$ on $G$. We focus on a flexible class of Dirichlet process Gaussian mixture (DPGM) priors, defined as follows.

\begin{spec}[DPGM Prior]
\label{dpgm}
We place a prior on $G$ through the representation
\[
G = Q * \mathrm{N}(0,A), \;
Q \sim \mathrm{DP}(\alpha,H), \;
A \sim \Gamma, \;
Q \indep A,
\]
where $\alpha>0$, $*$ denotes convolution,
and $\mathrm{DP}$ is the Dirichlet process~\citep{ferguson1973bayesian}. The prior $\Gamma$ on the Gaussian smoothing variance $A$ is supported on $[0,\bar A]$ for some $\bar A>0$. Moreover, there exist constants $c_1,c_2>0$ and $\kappa,r\geq 0$ such that, for all sufficiently small $\delta>0$,
\begin{equation}
\Gamma(A\leq \delta)
\geq
c_1 \delta^r \exp\{-c_2\delta^{-\kappa}\}.
\label{eq:prior_var_tail}
\end{equation}
For the base measure $H$, we consider either of the following specifications:
\begin{enumerate}[label=(\roman*), nosep, leftmargin=*]
\item \textbf{Base-Compact.} $H$ has a continuous, strictly positive density supported on $[-R,R]$.\label{assu:DP:compact}
\item \textbf{Base-Gauss.} $H=\mathrm{N}(0,A_H)$ for some $A_H>0$.\label{assu:DP:gauss}
\end{enumerate}
\label{assu:DPGM}
\end{spec}

The DPGM specification provides considerable flexibility for the shape of $G$, with the mixing variance $A$ governing its degree of smoothness. As $A$ approaches zero, $G$ can approach a discrete Dirichlet process draw, while larger values of $A$ induce greater smoothing. As we show below, the amount of mass that $\Gamma$ assigns near zero, summarized by $\kappa$, plays an important role in the frequentist behavior of both the posterior and the resulting Bayes rule. A useful special case is the degenerate choice $\Gamma=\delta_0$, for which the Gaussian convolution disappears and the specification reduces to the standard Dirichlet process prior $G\sim\mathrm{DP}(\alpha,H)$. In this case, the induced posterior mean coincides with estimators previously studied for the Gaussian sequence model by, among others, \citet{kuo1986note,escobar1994estimating,maceachern1994estimating}.

Our first result establishes posterior contraction for the marginal density induced by $G$. For a fixed mean vector $\boldmu \in \RR^n$ and a probability distribution $G$ on $\RR$, respectively, define
\begin{equation}
f_{\boldmu}(z)
:=
\frac{1}{n}\sum_{i=1}^n \varphi(z-\mu_i),
\; \;
f_G(z)
:=
\int \varphi(z-\mu)\,\dd G(\mu).
\label{eq:average_marginal}
\end{equation}
The density $f_{\boldmu}$ is the average marginal density of the observations under the sampling model~\eqref{eq:hierarchy1N} with mean vector $\boldmu$, while $f_G$ is the Gaussian mixture density induced by $G$. The following theorem shows that, under the fixed-$\boldmu$ compound decision sampling model, the posterior induced by the DPGM prior concentrates around $f_{\boldmu}$ in Hellinger distance. For two densities $f$ and $h$ on $\RR$, the squared Hellinger distance is defined by
\begin{equation}
\Dhel^2(f,h)
:=
\frac{1}{2}
\int
\left(\sqrt{f(z)}-\sqrt{h(z)}\right)^2
\,\dd z.
\label{eq:hellinger}
\end{equation}
For a prior $\Pi$ satisfying Specification~\ref{assu:DPGM}, define
\begin{equation}
\varepsilon_n(\Pi)
:=
\left(n^{-1/2}\log n\right)
\lor
\left(
n^{-\frac{1}{2+\kappa}}
\log^{\frac{\kappa}{2+\kappa}} n
\right).
\label{eq:epsilon_n}
\end{equation}

\begin{theo}[Posterior contraction]
\label{thm:compound-contraction}
Suppose \eqref{eq:hierarchy1N} holds with $\max_{1\le i\le n}|\mu_i|\le M$, and let $\Pi$ satisfy Specification~\ref{assu:DPGM} under either \ref{assu:DP:compact} with $R\geq M$ or \ref{assu:DP:gauss}. Then there exist constants $C,c>0$ such that, for all sufficiently large $n$,
\[
\PP[\boldmu]{
\Pi\!\left(
G:\Dhel(f_{\boldmu},f_G)\ge C\varepsilon_n(\Pi)
\;\Big|\; \boldZ
\right)
\le \exp\!\big(-c\,n\varepsilon_n^2(\Pi)\big)}
\ge 1-\frac{4}{n^2}.
\]
\end{theo}
The density $f_{\boldmu}$ is the natural marginal-density target in the fixed-$\boldmu$ compound decision problem; we explain the underlying intuition in Section~\ref{sec:proof_ideas}. It also naturally ties with our goal of proving frequentist risk guarantees for the induced Bayes rule: by the fundamental theorem of compound decisions~\citep{robbins1951asymptotically,zhang2003compound} (which we review in Section~\ref{sec:proof_ideas}), the compound risk of any separable rule of the form $\widehat{\mu}_i=\delta(Z_i)$ under the fixed mean vector $\boldmu$ coincides with the Bayes risk of the same rule under the empirical distribution
\begin{equation}
\compoundG := \frac{1}{n} \sum_{i=1}^n \delta_{\mu_i}.
\label{eq:compound_prior}
\end{equation}
Consequently, the component Bayes rule $\delta_{\compoundG}$ (defined in~\eqref{eq:cd_b_post_mean}) associated with $\compoundG$ is optimal among separable rules and therefore provides the natural separable oracle benchmark for the fixed-vector problem. The density $f_{\boldmu}$ is precisely the marginal density associated with this oracle. To summarize, posterior contraction toward $f_{\boldmu}$ in $\mathbb{P}_{\boldmu}$-probability is the natural fixed-vector analogue of contraction toward $f_{G_\star}$ in $\mathbb{P}_\star$-probability under a correctly specified hierarchical model, where $\mathbb{P}_\star$ denotes the sampling law under which \smash{$Z_i \simiid f_{G_\star}$}.

A key feature of Theorem~\ref{thm:compound-contraction} is that it is established entirely under the compound decision sampling model \eqref{eq:hierarchy1N}, with $\boldmu$ treated as fixed. In particular, there need not exist any latent distribution $G_\star$ generating the coordinates of $\boldmu$, even though the posterior itself is constructed from the working model \eqref{eq:hierarchy2N}. The theorem gives an explicit contraction rate for the class of DPGM priors in Specification~\ref{assu:DPGM}, including the standard Dirichlet process prior as a special case. In the latter case, it strengthens the compound posterior consistency result of \citet{datta1991consistency} by providing an explicit rate. The result can also be viewed as a compound-decision counterpart of \citet[Theorem~5.1]{ghosal2001entropies}, with the important distinction that the means $\mu_i$ are fixed rather than generated from a common mixing distribution $G_\star$.

Under squared-error loss, the Bayes rule induced by the posterior $\Pi(\cdot \mid \boldZ)$ is the posterior mean
\begin{equation}
\widehat{\boldmu}(\Pi)
:=
\mathbb{E}_{\Pi}\!\left[\boldmu \mid \boldZ\right]
=
\int \widehat{\boldmu}^{\mathrm{B}}(G)\,
\dd \Pi(G \mid \boldZ).
\label{eq:bnp-post-mean}
\end{equation}
Unlike the empirical Bayes plug-in rule, which evaluates $\widehat{\boldmu}^{\mathrm{B}}(G)$ at a single estimate of $G$, the nonparametric Bayes rule averages over the full posterior distribution of $G$. A central question is whether this fully Bayesian treatment can retain the strong frequentist risk guarantees of classical empirical Bayes procedures, even when the working model \eqref{eq:hierarchy2N} is potentially misspecified. As is standard in the compound decision literature, we measure the frequentist performance of an estimator $\widehat{\boldmu}$ under the fixed-$\boldmu$ sampling model \eqref{eq:hierarchy1N} by its root mean squared error (RMSE)
\begin{equation}
\label{eq:RMSE}
R(\widehat{\boldmu},\boldmu)
:=
\sqrt{\frac{1}{n}\EE[\boldmu]{\Norm{\widehat{\boldmu}-\boldmu}_2^2}}.
\end{equation}
Our next result shows that the nonparametric Bayes rule $\widehat{\boldmu}(\Pi)$ enjoys strong regret guarantees under this criterion. We define regret relative to the best risk attainable within the class of permutation equivariant decision rules~\citep{weinstein2021permutation},
\begin{equation}
\mathcal{T}^{\mathrm{PE}}
:=
\left\{
\boldt:\RR^n\to\RR^n
\,:\,
\boldt(\boldz_\pi)=\boldt(\boldz)_\pi
\text{ for all } \boldz\in\RR^n,\ \pi\in\mathcal{S}_n
\right\},
\label{eq:PE}
\end{equation}
where $\boldz_\pi=(z_{\pi(1)},\ldots,z_{\pi(n)})$ and $\mathcal{S}_n$ denotes the set of permutations of $\{1,\ldots,n\}$. In the absence of additional information distinguishing the coordinates of $\boldmu$, this class provides a natural benchmark, requiring the decision rule to respect the symmetry of the compound decision problem.

\begin{theo}
\label{theo:main_risk_bound_eq}
Suppose the conditions of Theorem~\ref{thm:compound-contraction} hold. Then there exists a constant $C>0$, depending only on $M$ and the prior specification, such that, for all $n\geq 2$,
$$
\sup_{\boldmu\in[-M,M]^n}
\left\{
R\big(\widehat{\boldmu}(\Pi),\boldmu\big)
-
\inf_{\boldt\in\mathcal{T}^{\mathrm{PE}}}
R\big(\boldt(\boldZ),\boldmu\big)
\right\}
\leq
C\,\varepsilon_n(\Pi)\log^{3/2}n.
$$
Here, $\varepsilon_n(\Pi)$ is defined in~\eqref{eq:epsilon_n}, and $R(\cdot,\boldmu)$ is the frequentist risk in~\eqref{eq:RMSE} under the sampling model~\eqref{eq:hierarchy1N} with $\boldmu$ treated as fixed.
\end{theo}
When $\kappa=0$, the rate in Theorem~\ref{theo:main_risk_bound_eq} simplifies to
$
\varepsilon_n(\Pi)\log^{3/2}n
=
n^{-1/2}\log^{5/2}n.
$
This matches the benchmark rate of \citet{jiang2009general}. The oracle benchmark in Theorem~\ref{theo:main_risk_bound_eq} is strong: the infimum is taken separately for each fixed $\boldmu$ over the entire class of permutation equivariant decision rules, so the comparison is with the best equivariant rule for the particular mean vector being estimated. It contains all simple separable decision rules considered in the classical compound decision literature~\citep{jiang2009general}, including the oracle Bayes rules $\widehat{\boldmu}^{\mathrm{B}}(G)$ in \eqref{eq:cd_b_post_mean} for any distribution $G$ on $\RR$. It also contains data dependent empirical Bayes procedures, such as the NPMLE plug-in rule, the James-Stein estimator~\citep{james1961estimation}, as well as Bayes rules induced by arbitrary exchangeable priors on $\boldmu$. The latter include, in particular, the fully Bayesian rules $\widehat{\boldmu}(\widetilde{\Pi})$ obtained from the hierarchy \eqref{eq:hierarchy1N}--\eqref{eq:hierarchy2N} under any prior $\widetilde{\Pi}$ on $G$. Thus, in the absence of additional information distinguishing the units, $\mathcal{T}^{\mathrm{PE}}$ provides a natural benchmark: it is not tied to any particular empirical Bayes or Bayesian construction, but instead ranges over a broad class of procedures that respect the permutation symmetry of the compound decision problem.

The role of $\kappa$ in Specification \ref{dpgm} and the rate $\varepsilon_n(\Pi)$ in (\ref{eq:epsilon_n}) can be understood through the amount of prior mass available near small smoothing variances. When $\kappa=0$, the prior assigns sufficiently large mass to arbitrarily small values of $A$. Thus, although $G=Q*\mathrm{N}(0,A)$ has a smooth density whenever $A>0$, the prior retains the flexibility to approximate the marginal distributions generated by essentially discrete mixing distributions as $A$ approaches zero. This allows the Bayesian procedure to attain the same benchmark regret rate as the NPMLE-based empirical Bayes rule. When $\kappa>0$, by contrast, the prior increasingly penalizes the small values of $A$ needed for such approximations, leading to the slower regret rate in Theorem~\ref{theo:main_risk_bound_eq}. The following remark provides some further insight into this trade-off.

\begin{rema}[Smoothing bias]
For a given smoothing variance $A > 0$, define the Gaussian mixture approximation bias:
\[
\lambda(A)
:=
\inf_{Q\in\mathcal{P}(\mathbb{R})}
\Dhel\!\left(f_{\boldmu},f_{Q*\mathrm{N}(0,A)}\right).
\]
The regret rate reflects a trade-off between making $A$ sufficiently small to control $\lambda(A)$ and assigning enough prior mass to such values of $A$. Without additional structure on $\boldmu$, one always has
$\lambda(A)\lesssim A$, and this bound is sharp in general. For example, when $\mu_i=0$ for all $i$, one has $\lambda(A)\asymp A$. Thus, when $\kappa>0$, the low prior mass near $A=0$ cannot in general be offset by a faster approximation rate. In  models with a ``true'' latent distribution $G_\star$, faster approximation rates may be available under additional smoothness conditions on $G_\star$~\citep{rousseau2024wasserstein,kankanala2023quasi}. In the current fixed-$\boldmu$ compound decision problem, however, there is no ``true'' $G_\star$, making it less clear how an analogous smoothness condition should be formulated.
\end{rema}

These observations also shed  light on the longstanding issue of smoothness in empirical Bayes. Indeed, a common criticism of the NPMLE is that the estimated  mixing distribution $\widehat{G}$ is discrete. This has motivated a variety of approaches for ``betting on smoothness''~\citep{efron2022discussion},\footnote{Some introduce smoothness through an additional tuning choice made separately from the original estimation problem, including post-hoc kernel smoothing of the NPMLE~\citep{pan1999two,koenker2017rebayes,gu2023invidious}, early stopping of EM algorithms initialized at a smooth distribution~\citep{laird1991smoothing,shen1999empirical,jiang2019comment}, and predictive-recursion methods~\citep{tokdar2009consistencya,martin2021survey}. Other approaches instead obtain smoothness through stronger assumptions on a correctly specified working model, for example by assuming that the true mixing distribution satisfies $G_\star=Q_\star*\mathrm{N}(0,A_\star)$ for a fixed, known $A_\star>0$~\citep{magder1996smooth,kim2026empirical}, or by restricting $G_\star$ to a parametric exponential family~\citep{efron2016empirical,kline2024discrimination}.} which typically obtain smoothness either through an additional smoothing choice external to the original estimation problem or through a substantive restriction on $G_\star$. Moreover, these approaches have not generally been accompanied by the sharp regret guarantees available for the classical unsmoothed NPMLE.

The Bayesian framework, based on the DPGM prior, makes the trade-off between smoothness and approximation flexibility explicit. The prior encodes preferences over the amount of smoothing through $\Gamma$, while the posterior adapts the degree of smoothing to the data. If $\Gamma(\{0\})=0$, then $A>0$ almost surely and the induced distribution $G=Q*\mathrm{N}(0,A)$ has a smooth density with prior probability one. Importantly, smoothness itself need not entail a loss in regret. For example, a continuous prior $\Gamma$ with polynomial mass near zero (e.g., a Beta prior), corresponding to $\kappa=0$, places probability one on smooth mixing distributions while retaining sufficient mass near $A=0$ to attain the benchmark regret rate in Theorem~\ref{theo:main_risk_bound_eq}. Thus, a broad class of fully Bayesian procedures can ``bet on smoothness,'' learn the degree of smoothing from the data, and still attain the same benchmark regret rate as the NPMLE-based empirical Bayes rule.

We now turn to admissibility in the compound decision problem. This complements the preceding regret analysis by examining whether a decision rule can be uniformly improved in frequentist risk.
\begin{defi}[Admissibility]
\label{def:admissibility_cd}
An estimator $\widehat{\boldmu}$ is inadmissible over the parameter space $\RR^n$ if there exists another estimator $\widetilde{\boldmu}$ such that
\begin{enumerate}
\item $R(\widetilde{\boldmu}, \boldmu) \leq R(\widehat{\boldmu}, \boldmu)$ for all $\boldmu \in \RR^n$, and
\item $R(\widetilde{\boldmu}, \boldmu_0) < R(\widehat{\boldmu}, \boldmu_0)$ for some $\boldmu_0 \in \RR^n$.
\end{enumerate}
Otherwise, $\widehat{\boldmu}$ is admissible over $\RR^n$.
\end{defi}
Our first result establishes the admissibility of the full Bayes decision rule.
\begin{prop}
\label{prop:bnp-admissible-cd}
The Bayes estimator $\widehat{\boldmu}(\Pi) = \mathbb{E}_{\Pi}\!\left[\boldmu \mid \boldZ\right]$ is admissible over $\RR^n$  in the compound decision setting with data generating process given by~\eqref{eq:hierarchy1N}.
\end{prop}
We next contrast this result with the NPMLE-based empirical Bayes estimator. To state this result precisely, we first define the NPMLE of the mixing distribution $G$~\citep{robbins1950generalization, kiefer1956consistency}. Let
\begin{equation} \widehat{G} \equiv \widehat{G}(\boldZ)\in \argmax_{G \in \mathcal{P}(\RR)} \cb{ \prod_{i=1}^n f_G(Z_i)},
\label{eq:npmle_and_marginal} \end{equation}
where $\mathcal{P}(\RR)$ denotes the set of all probability distributions supported on $\RR$. The EB estimator for the $i$-th coordinate is the plug-in Bayes rule $\widehat{\mu}_i^{\mathrm{EB}}(\boldZ) := \delta_{\widehat{G}}(Z_i)$.

\begin{prop}
\label{prop:npmle-inadmissible-cd}
The NPMLE-based EB estimator $\widehat{\boldmu}^{\mathrm{EB}}$ is inadmissible over $\mathbb R^n$ for all $n \geq 2$.\footnote{
For $n=1$, the NPMLE plug-in is admissible over $\RR$. The reason is that the NPMLE \smash{$\widehat{G}$} is a point mass at $Z_1$, and so \smash{$\widehat{\mu}_1^{\mathrm{EB}} \equiv Z_1$} which is admissible for $n=1$~\citep{hodges1951some}.
}
\end{prop}

Taken together, these results highlight an important decision-theoretic distinction between the two procedures: the Bayesian rule is admissible, whereas the NPMLE-based empirical Bayes rule is not. Beyond this distinction, the nonparametric Bayes rule also admits a useful structural interpretation. The following result, due to~\citet{datta1991asymptotic}, gives a leave-one-out (LOO) representation of $\widehat{\mu}_i(\Pi)$. In particular, it shows that the mixing distribution used to estimate $\mu_i$ is learned from the remaining observations $\boldZ_{-i}$, while $Z_i$ enters only through the final shrinkage rule.

\begin{prop}[LOO Representation]
\label{prop:datta_loo}
For each $i=1,\ldots,n$,
\[
\widehat{\mu}_i(\Pi)
=
\delta_{\widebar{G}_{-i}}(Z_i), \; \;
\widebar{G}_{-i}
:=
\EE[\Pi]{G\mid\boldZ_{-i}},
\]
where $\boldZ_{-i}=(Z_1,\ldots,Z_{i-1},Z_{i+1},\ldots,Z_n)$ and $\widebar{G}_{-i}$ is the posterior mean of $G$ given $\boldZ_{-i}$.
\end{prop}

This representation contrasts with the NPMLE-based EB estimator $\widehat{\mu}_i^{\mathrm{EB}}$, which estimates $G$ using all observations, including $Z_i$. The LOO form is appealing because it separates the two sources of information entering the estimate: $\boldZ_{-i}$ determines the estimated prior distribution, while $Z_i$ enters only through the final shrinkage rule. In this sense, the rule cleanly distinguishes indirect information learned from the ensemble of other observations from the direct information about $\mu_i$ contained in $Z_i$. For the NPMLE, there is empirical evidence that leave-one-out estimation can improve performance in some settings~\citep{ho2025largescale}. Implementing such a procedure, however, requires solving the NPMLE optimization problem $n$ times and can therefore be computationally burdensome. By contrast, Proposition~\ref{prop:datta_loo} shows that the Bayesian rule acquires this leave-one-out structure automatically.

\section{Heteroscedastic models} \label{sec-hetero}
In many empirical Bayes applications, the precision of the unit-level estimates varies across units and may itself be informative about the latent parameters. We therefore consider the heteroscedastic Gaussian sampling model
\begin{equation}
\tag{CD-H}
Z_i \cond \mu_i,\sigma_i^2 \;
\simindep \;
\mathrm{N}(\mu_i,\sigma_i^2),
\qquad i=1,\ldots,n,
\label{eq:hierarchy1NHet}
\end{equation}
where the variances $\boldsigma^2=(\sigma_1^2,\ldots,\sigma_n^2)$ are observed and play the role of the auxiliary quantities $\boldnu$ in \eqref{eq:hierarchy1}. We assume throughout that
$0<\ubar\sigma\leq\sigma_i\leq\bar\sigma<\infty$ for all $i$.  Our analysis proceeds conditionally on the realized values of $\boldsigma^2$.

To accommodate the possibility that estimation precision is informative about the latent parameters, we follow~\citet{chen2026empirical} and consider a conditional location-scale working model,
\begin{equation}
\tag{B-H}
\mu_i
=
m_{\star}(\sigma_i)
+
s_{\star}(\sigma_i)\zeta_i,
\qquad
\zeta_i \simiid G_{\star},
\qquad i=1,\ldots,n,
\label{eq:hierarchy2NClose}
\end{equation}
where $G_{\star}$ is an unknown latent distribution and $m_{\star}$ and $s_{\star}$ are unknown location and scale functions, respectively. Thus, the location and scale of the conditional distribution of $\mu_i$ are allowed to vary nonparametrically with estimation precision, while $G_{\star}$ captures a common shape across units. This specification provides a useful instance of the  transformation $\mathcal{T}_{\eta}$ in \eqref{eq:hierarchy2}: it offers some flexibility in modeling the latent distribution while preserving tractability of the conditional law through the location-scale structure.

\begin{rema}[Heteroscedastic compound oracle]
In our analysis below, we suppose that data generation is governed by both~\eqref{eq:hierarchy1NHet} and~\eqref{eq:hierarchy2NClose}. Thus, unlike in Section~\ref{sec:homoscedastic_gaussian}, we do not consider the fixed-$\boldmu$ compound-decision setting. The reason that we (and related work, e.g.,~\citet{jiang2020general,Soloff,chen2026empirical}) do not consider the compound setting under heteroscedasticity is that it is unclear how to define a suitable oracle. In particular, consider a permutation-equivariant oracle as in~\eqref{eq:PE}, now consisting of decision rules $\boldt(\boldZ,\boldsigma^2)$ satisfying
$
\boldt(\boldZ,\boldsigma^2)_\pi
=
\boldt(\boldZ_\pi,\boldsigma_\pi^2),
$
so that the pairs $(Z_i,\sigma_i^2)$ are kept together under permutation. Allowing the oracle knowledge of $\boldmu$ is then too strong: if all $\sigma_i^2$ are distinct, it can simply use
$
t_i(\boldZ,\boldsigma^2)
=
\sum_{j=1}^n
\mu_j\ind\{\sigma_i^2=\sigma_j^2\}
=
\mu_i
$
and achieve RMSE equal to zero. Hence, a regret guarantee analogous to Theorem~\ref{theo:main_risk_bound_eq} is not possible against this oracle. Instead, under~\eqref{eq:hierarchy1NHet} and~\eqref{eq:hierarchy2NClose}, we compare against the usual Bayesian oracle,
$
\EE[\star]{\boldmu\mid\boldZ,\boldsigma^2}.
$
\label{rema:heteroscedastic_oracle}
\end{rema}

Throughout this section, let $(\Sigma^\alpha,\|\cdot\|_{\Sigma^\alpha})$ and $(H^\alpha,\|\cdot\|_{H^\alpha})$ denote the H\"older and Sobolev spaces of order $\alpha$ on the interval $ \mathcal{X} = [\ubar\sigma,\bar\sigma]$, respectively. (See Appendix~\ref{sec:funct_space} for definitions.)

\subsection{Analysis}
\label{subsec:CLOSE}
We consider a nonparametric Bayesian treatment of the unknown components in the working model~\eqref{eq:hierarchy2NClose}. We parameterize the location and scale functions by $\eta=(\eta_1,\eta_2)$ according to
\[
m(\sigma)=\eta_1(\sigma), \;
s_{\eta_2}(\sigma)=\Lambda^{1/2}\!\left(\eta_2(\sigma)\right),
\]
where $\Lambda:\RR\to[0,\infty)$ is a prespecified link function, ensuring that the scale function is nonnegative. We model the unknown functions $\eta_1$ and $\eta_2$ nonparametrically using Gaussian process (GP) priors. Specifically, for $\alpha>0$, let $\Pi_{\eta}$ denote the law of a centered Mat\'ern  Gaussian process on $[\ubar\sigma,\bar\sigma]$. This  process has covariance kernel
\begin{equation}
\label{GP}
K_{\alpha}(s,t)
=
\int_{\RR}
e^{-\mathbf{i}\langle s-t,\zeta\rangle}
(1+|\zeta|^2)^{-(\alpha+1/2)}
\,\dd\zeta.
\end{equation}
The parameter $\alpha$ controls the regularity of the sample paths: draws from $\Pi_{\eta}$ belong almost surely to the H\"older spaces $\Sigma^\beta$ for every $\beta<\alpha$~\citep[Proposition I.4]{ghosal2017fundamentals}. Thus, the prior may be viewed as concentrating on functions that are ``almost $\alpha$-smooth.''\footnote{Associated with $\Pi_\eta$ is a reproducing kernel Hilbert space (RKHS), which in this case is, up to norm equivalence, $H^{\alpha+1/2}$. Approximation under the Gaussian process prior is determined by how closely a target function can be approximated by elements of this space and by the norm of the approximating element.}

We write $\theta=(G,\eta_1,\eta_2)$ for a generic parameter and $\Pi(\cdot\mid\boldZ,\boldsigma^2)$ for the posterior induced by the hierarchy \eqref{eq:hierarchy1NHet}--\eqref{eq:hierarchy2NClose} under a prior $\Pi$ on $\theta$. We specify $\Pi$ as follows.
\begin{spec}[DPGM--GP prior]
\label{loc-scale-prior}
The prior $\Pi$ draws $G$, $\eta_1$, and $\eta_2$ independently, as follows.
\begin{enumerate}[label=(\roman*), nosep, leftmargin=*]
\item $G$ is drawn according to Specification~\ref{assu:DPGM} under \ref{assu:DP:compact}, with $R>0$, $r\geq 0$ and $\kappa=0$.
\item $\eta_1$ and $\eta_2$ are drawn independently from $\Pi_{\eta}$, with regularity $\alpha>1/2$.
\item We use the softplus link
$
\Lambda(t)=\log(1+e^t),
$
so that
$
s_{\eta_2}^2(\sigma)
= \log(1+ e^{\eta_2(\sigma)})
$.
\end{enumerate}
\end{spec}
An important difference from the setting in Section~\ref{sec:homoscedastic_gaussian} is that the working model now involves the unknown functions $m_\star$ and $s_\star$. Even if $G_\star$ were known, these functions must still be learned nonparametrically from variation in $\sigma_i$. For $\alpha$-smooth functions of a one-dimensional argument, the usual minimax rate is $n^{-\alpha/(2\alpha+1)}$. As our results below illustrate, this nonparametric difficulty determines the contraction rate in the present setting. Accordingly, we write
\begin{equation}
\varepsilon_n(\Pi)
:=
n^{-\alpha/(2\alpha+1)}.
\end{equation}
We impose the following regularity conditions on the data-generating process.
\begin{assu}[Regularity of  $(G_{\star}, m_{\star}, s_{\star}$)]
\label{assu:LS}
Fix $B,M>0$ and $\alpha>1/2$. Suppose that
\begin{enumerate}[label=(\roman*), nosep, leftmargin=*]
\item $G_{\star} \in \mathcal{P}([-M,M])$;
\item $m_\star \in \Sigma^{\alpha} \cap H^{\alpha}$ and $
\max\{\Norm{m_{\star}}_{\Sigma^{\alpha}},\, \Norm{m_{\star}}_{H^{\alpha}}, \, \Norm{m_{\star}}_{\infty}\} \leq B$;
\item $s_\star \in \Sigma^{\alpha} \cap H^{\alpha}$ and $
\max\{\Norm{s_{\star}}_{\Sigma^{\alpha}},\, \Norm{s_{\star}}_{H^{\alpha}}, \, \Norm{s_{\star}}_{\infty}\} \leq B$. For every sufficiently large $n$, there exists $\eta_{2,n} \in H^{\alpha+1/2}$ such that
$$
\Norm{s_\star - s_{\eta_{2,n}}}_\infty \le B \varepsilon_n(\Pi),
\quad \text{and} \quad
\Norm{\eta_{2,n}}_{H^{\alpha+1/2}} \le B \sqrt{n}\varepsilon_n(\Pi).
$$
\end{enumerate}
\end{assu}
The first two conditions impose standard support and smoothness restrictions on $G_\star$ and $m_\star$. Condition~(iii) similarly imposes smoothness on $s_\star$, together with an approximation requirement reflecting that the prior is placed on $\eta_2$ rather than directly on $s_\star$. Specifically, $s_\star$ need only be approximated at the target rate by functions of the form $s_{\eta_2}=\Lambda^{1/2}\circ\eta_2$, with
$\Norm{\eta_2}_{H^{\alpha+1/2}}$ of order at most $\sqrt{n}\varepsilon_n(\Pi)$. Since $\sqrt{n}\varepsilon_n(\Pi)\rightarrow\infty$, the complexity of the approximating functions may increase with $n$.

For the softplus link, when $s_\star$ is bounded away from zero, the approximation requirement in Assumption~\ref{assu:LS}(iii) is implied by the smoothness conditions imposed on $s_\star$. In that case,
\[
\eta_{2,\star}(\sigma)
=
\Lambda^{-1}\!\left(s_\star^2(\sigma)\right)
=
\log( e^{s_{\star}^2(\sigma)} - 1  ).
\]
Since the map $t\mapsto\Lambda^{-1}(t^2)$ has sufficiently many bounded derivatives over the range of $s_\star$,  $\eta_{2,\star}$ inherits the H\"older and Sobolev regularity of $s_\star$. Hence, if
$s_\star\in\Sigma^\alpha\cap H^\alpha$, standard approximation results imply Assumption~\ref{assu:LS}(iii)~\citep[Lemma 11.37]{ghosal2017fundamentals}. The more general approximation formulation in Assumption~\ref{assu:LS}(iii) avoids imposing any uniform positive lower bound on $s_\star$. This flexibility is useful in applications where the dispersion of the latent distribution varies with estimation precision and may become arbitrarily small at some values of $\sigma$. By formulating Assumption~\ref{assu:LS}(iii) as an approximation condition, we can accommodate such behavior even when the chosen link generates only strictly positive scale functions.
\begin{rema}[Scale functions at the boundary]
The procedure in~\citet{chen2026empirical} imposes a uniform positive lower bound on the scale function $s_\star(\sigma)$. This condition arises  from their two-step plug-in construction, which first centers and rescales the data using estimates of the location and scale functions to remove precision dependence, and then applies a conventional NPMLE procedure to the normalized observations. As $s_\star(\sigma)$ approaches zero, this normalization becomes increasingly unstable and is undefined at the boundary $s_\star(\sigma)=0$. Our framework does not require such a restriction. To illustrate, consider the fully degenerate case $
s_\star(\sigma)=0 $ for all $\sigma$
and take the softplus link $\Lambda(t)=\log(1+e^t)$. Define
$
\eta_{2,n}(\sigma)
=
-2\log\!\left(\varepsilon_n(\Pi)^{-1}\right).
$
Since $\log(1+x)\leq x$, we have $
s_{\eta_{2,n}}(\sigma)
=
\log^{1/2}\!(1+e^{\eta_{2,n}(\sigma)})
\leq
\varepsilon_n(\Pi)$
uniformly in $\sigma$. Moreover,
$
\Norm{\eta_{2,n}}_{H^{\alpha+1/2}}
\lesssim
\log\!\left(\varepsilon_n(\Pi)^{-1}\right)
=
O(\log n)
=
o\!\left(\sqrt{n}\varepsilon_n(\Pi)\right),$
so that Assumption~\ref{assu:LS}(iii) is satisfied even in the degenerate boundary case $s_\star = 0$.
\end{rema}

We next impose a regularity condition on the realized standard deviations $\sigma_1,\ldots,\sigma_n$. Because $m_\star$ and $s_\star$ are learned nonparametrically from variation in $\sigma_i$, the fixed design points must provide sufficient local information throughout their range. Conditions of this type are standard in nonparametric estimation and require, roughly, that the design not become too sparse around any $\sigma_i$. We formalize this as follows.
\begin{assu}[Regular fixed design]
\label{assu:close-fixed-design} (i) The standard deviations $\sigma_1,\ldots,\sigma_n$ satisfy
$\sigma_i\in[\ubar\sigma,\bar\sigma]$ for all $i$, where
$0<\ubar\sigma<\bar\sigma<\infty$. (ii) There exists a constant $c>0$ such that, for all sufficiently large $n$,
$
\min_{1\le i\le n}
\#\big\{j:|\sigma_j-\sigma_i|\le h_n\big\}
\ge cnh_n,
$ where $h_n=n^{-1/(2\alpha+1)}$.
\end{assu}
Assumption~\ref{assu:close-fixed-design} requires each $\sigma_i$ to have at least order $nh_n$ neighboring design points within distance $h_n$, ensuring that no part of the realized design becomes too sparse at the scale relevant for nonparametric estimation. The choice
$h_n=n^{-1/(2\alpha+1)}$ is the usual bandwidth for estimating an $\alpha$-smooth function of one variable; for example, $h_n=n^{-1/5}$ when $\alpha=2$. At this bandwidth, the condition is implied by Assumptions (LP1) and (LP3) of~\citet{tsybakov2008introduction} and by Assumption SM8.2 of~\citet{chen2026empirical}.

For a generic parameter $\theta=(G,\eta_1,\eta_2)$, integrating out the latent variable in the working model gives the conditional marginal density
$$
f_{\theta,\sigma^2}(z)
:=
\int
\varphi\!\left(
z-\eta_1(\sigma)-s_{\eta_2}(\sigma)u;\sigma^2
\right)
\,\dd G(u)
$$
for an observation with standard deviation $\sigma$. Since the marginal distribution generally varies with $\sigma$, we collect the densities corresponding to the realized design points in the vector
$
\boldf_{\theta,\boldsigma^2}
:=
(
f_{\theta,\sigma_1^2},
\ldots,
f_{\theta,\sigma_n^2}
).$
We use analogous notation for the data-generating distribution:
\begin{equation}
\boldf_{\star,\boldsigma^2}
:=
(
f_{\star,\sigma_1^2},
\ldots,
f_{\star,\sigma_n^2}
), \; \;
f_{\star,\sigma^2}(z)
:=
\int
\varphi\!\left(
z-m_\star(\sigma)-s_\star(\sigma)u;\sigma^2
\right)
\,\dd G_\star(u).
\label{eq:close_truth}
\end{equation}
Conditional on $\boldsigma^2$, the observations are independent but generally not identically distributed. We therefore measure the discrepancy between two vectors of marginal densities
$\boldf=(f_1,\ldots,f_n)$ and $\boldh=(h_1,\ldots,h_n)$ by their average Hellinger distance,
\begin{equation}
\Dhel^2(\boldf,\boldh)
:=
\frac{1}{n}
\sum_{i=1}^n
\Dhel^2(f_i,h_i).
\label{eq:avg_hellinger_het}
\end{equation}
With this notation, we can state our posterior contraction result. The theorem shows that the posterior learns the vector of conditional marginal distributions generated by the working model at the usual nonparametric rate $\varepsilon_n(\Pi)=n^{-\alpha/(2\alpha+1)}$. Thus, although both the latent distribution  and the location and scale functions are unknown, the additional nonparametric structure does not lead to a slower rate than the benchmark associated with estimating an $\alpha$-smooth function of one variable.

\begin{theo}[Posterior contraction]
\label{thm:close_contraction}
Suppose that data generation is governed by~\eqref{eq:hierarchy1NHet} and~\eqref{eq:hierarchy2NClose}, that Specification~\ref{loc-scale-prior} holds for $\Pi$, and that Assumption~\ref{assu:LS} holds. Then, for every $q\geq1$, there exist constants $D,c,C_q>0$ such that, for all sufficiently large $n$,
\[
\PP[\star]{\,
\Pi\!\left(
\Dhel\!\left(
\boldf_{\star,\boldsigma^2},
\boldf_{\theta,\boldsigma^2}
\right)
\geq D\varepsilon_n(\Pi)
\,\Big|\,
\boldZ,\boldsigma^2
\right)
\leq
\exp\!\left\{-cn\varepsilon_n^2(\Pi)\right\}
\,\Big|\,
\boldsigma^2
}
\geq
1-\frac{C_q}{n^q}.
\]
The constants may depend on the prior specification, Assumption~\ref{assu:LS}, $\ubar\sigma$, and $\bar\sigma$, but not on $n$.
\end{theo}
We now turn from posterior contraction to the decision-theoretic performance of the resulting estimator. As discussed in Remark~\ref{rema:heteroscedastic_oracle}, in the present setting we evaluate performance relative to the oracle Bayes rule under the working model~\eqref{eq:hierarchy2NClose}. Accordingly, rather than conditioning on a fixed vector $\boldmu$ as in~\eqref{eq:RMSE}, we integrate over the latent means generated by the working model. As throughout this section, all risks are conditional on the realized values of $\boldsigma^2$, with this conditioning suppressed from the notation.

Our goal is to estimate $\boldmu$, so we consider decision rules
$\boldt:\RR^n\times[\ubar\sigma^2,\bar\sigma^2]^n\to\RR^n$
and estimators of the form
$\widehat{\boldmu}=\boldt(\boldZ,\boldsigma^2)$. We measure their performance by the integrated RMSE
\begin{equation}
R(\widehat{\boldmu},\star)
:=
\left\{
\frac{1}{n}
\EE[\star]{\Norm{\widehat{\boldmu}-\boldmu}^2}
\right\}^{1/2},
\label{eq:RMSE_bayes}
\end{equation}
where $\star=(G_\star,m_\star,s_\star)$ and the expectation is taken over both the latent means in~\eqref{eq:hierarchy2NClose} and the observations in~\eqref{eq:hierarchy1NHet}, conditional on $\boldsigma^2$. Under squared-error loss, the optimal rule is the oracle Bayes rule
\begin{equation}
\begin{aligned}
\widehat{\boldmu}^{\mathrm{B}}(\star)
&:=
\left(
\delta_\star(Z_1,\sigma_1^2),
\ldots,
\delta_\star(Z_n,\sigma_n^2)
\right),\\
\delta_\star(z,\sigma^2)
&:=
\frac{1}{
f_{\star,\sigma^2}(z)
}
\int
\{m_\star(\sigma)+s_\star(\sigma)u\}
\varphi\!\left(
z-m_\star(\sigma)-s_\star(\sigma)u;\sigma^2
\right)
\,\dd G_\star(u).
\label{eq:close_bayes_rule}
\end{aligned}
\end{equation}
Equivalently,
$\delta_\star(z,\sigma^2)$ is the conditional mean of $\mu_i$ given
$(Z_i,\sigma_i^2)=(z,\sigma^2)$. Its risk is
\[
R\!\left(\widehat{\boldmu}^{\mathrm{B}}(\star),\star\right)
=
\left\{
\frac{1}{n}
\sum_{i=1}^n
\EE[\star]{
\Var[\star]{\mu_i\mid Z_i,\sigma_i^2}
}
\right\}^{1/2}.
\]
This is an oracle benchmark because the rule depends on the unknown components $(G_\star,m_\star,s_\star)$ of the working model. The BNP posterior mean estimator $\widehat{\boldmu}(\Pi) = \EE[\Pi]{\boldmu \mid \boldZ, \boldsigma^2}$ defined in~\eqref{eq:bayes_postmean} instead learns these components from the data through the posterior. The following result shows that its excess risk relative to the oracle is governed by the same nonparametric rate as the posterior contraction result, up to logarithmic factors.

\begin{theo}[Hierarchical regret]
\label{theo:close-regret}
Suppose that data generation is governed by~\eqref{eq:hierarchy1NHet} and~\eqref{eq:hierarchy2NClose}, that Specification~\ref{loc-scale-prior} holds for $\Pi$, and that Assumptions~\ref{assu:LS} and~\ref{assu:close-fixed-design} hold. Then there exists a constant $C>0$, depending only on the constants in the prior specification and the assumptions, such that, for all sufficiently large $n$,
$$
R\!\left(\widehat{\boldmu}(\Pi),\star\right)
-
R\!\big(\widehat{\boldmu}^{\mathrm{B}}(\star),\star\big)
\leq
C\,\varepsilon_n(\Pi)\log^2 n.
$$
\end{theo}
Thus, the BNP estimator asymptotically attains the oracle Bayes risk at the usual one-dimensional nonparametric rate, up to logarithmic factors, despite jointly learning the latent distribution and the unknown location and scale functions. The resulting regret rate matches the minimax-optimal rate established by~\citet{chen2026empirical} for their two-step plug-in empirical Bayes procedure.

\begin{rema}[Normalization and identification]
\label{normalize}
The location-scale representation in~\eqref{eq:hierarchy2NClose} is not uniquely identified without a normalization of $G_\star$, since location and scale transformations of the latent variable can be offset by corresponding transformations of $(m_\star, s_\star)$. The analysis in~\citet{chen2026empirical} proceeds   by imposing that $G_\star$ has mean zero and variance one. Under this normalization,
$
m_\star(\sigma)
=
\mathbb{E}[Z_i\mid\sigma_i=\sigma] \; , \;
s_\star^2(\sigma)
=
\Var{Z_i\mid\sigma_i=\sigma}-\sigma^2,
$
so that the location and scale functions are uniquely identified as conditional moments from the data. Our procedure does not require imposing this normalization. Instead, we work directly with the original likelihood, and our theoretical analysis depends only on the induced conditional distributions and Bayes rules, which are invariant across observationally equivalent location-scale representations. If desired, the conventional mean-zero, unit-variance normalization can be imposed ex post on posterior draws for interpretation.\footnote{For a posterior draw $(G,m,s)$, let $a_G$ and $b_G>0$ denote the mean and standard deviation under $G$. Replacing $G$ by the distribution of $(U-a_G)/b_G$, $U\sim G$, and setting
$\bar m(\sigma)=m(\sigma)+a_Gs(\sigma)$ and
$\bar s(\sigma)=b_Gs(\sigma)$
yields an observationally equivalent representation in which the latent distribution has mean zero and variance one.}
\end{rema}

\subsection{Special case: Mean-precision independence}
\label{subsec:CDNHet}

A useful and simple special case of the location-scale model~\eqref{eq:hierarchy2NClose} is obtained by setting
$m_\star(\sigma)=0$ and $s_\star(\sigma)=1$ for all $\sigma$. The working model then reduces to
\begin{equation}
\tag{B-$\mathrm{N}^*$}
\mu_i\cond\sigma_i^2,G_\star
\;\simiid\;
G_\star,
\qquad i=1,\ldots,n,
\label{eq:hierarchy2NInd}
\end{equation}
so that the distribution of the latent means no longer varies with estimation precision. From a Bayesian perspective, this amounts to placing point-mass priors on the location and scale functions at $0$ and $1$, respectively, so that only $G_\star$ remains to be learned. This is the heteroscedastic empirical Bayes model studied, for example, by~\citet{jiang2020general} and~\citet{Soloff}.

The objects from the preceding subsection simplify accordingly. The marginal density is
\begin{equation}
f_{G,\sigma^2}(z)
:=
\int \varphi(z-\mu;\sigma^2)\,\dd G(\mu),
\qquad
\boldf_{G,\boldsigma^2}
:=
\left(
f_{G,\sigma_1^2},\ldots,f_{G,\sigma_n^2}
\right),
\label{eq:vectorized_marginal_density}
\end{equation}
while the corresponding oracle and Bayes rule are
\begin{equation}
\widehat{\boldmu}^{\mathrm{B}}(G_\star)
=
\left(
\delta_{G_\star}(Z_1,\sigma_1^2),\ldots,
\delta_{G_\star}(Z_n,\sigma_n^2)
\right) \;,\;  \;\delta_G(z,\sigma^2)
:=
\frac{
\textstyle\int \mu\,\varphi(z-\mu;\sigma^2)\,\dd G(\mu)
}{
\textstyle\int \varphi(z-\mu;\sigma^2)\,\dd G(\mu)
}.
\label{eq:cd_b_post_mean_het_ind}
\end{equation}
We retain the DPGM prior in Specification~\ref{assu:DPGM}. Since there are no longer unknown location and scale functions to estimate nonparametrically, the contraction rate returns to that of the homoscedastic compound decision problem.

\begin{theo}[Posterior contraction]
\label{theo:het-ind-contraction}
Suppose that the data are generated according to~\eqref{eq:hierarchy1NHet} and~\eqref{eq:hierarchy2NInd}, with
$G_\star\in\mathcal P([-M,M])$.
Let $\Pi$ satisfy Specification~\ref{assu:DPGM} under either \ref{assu:DP:compact} or \ref{assu:DP:gauss}. Under \ref{assu:DP:compact}, assume additionally that the support of the base measure $H$ satisfies $R\ge M$.
Then there exist constants $C,c>0$, depending only on the specification, $M$, $\ubar\sigma^2$, and $\bar\sigma^2$, such that, for all sufficiently large $n$,
\[
\PP[G_\star]{
\Pi\!\left(
G:
\Dhel\!\left(
\boldf_{G_\star,\boldsigma^2},
\boldf_{G,\boldsigma^2}
\right)
\geq C\varepsilon_n(\Pi)
\,\Big|\,
\boldZ,\boldsigma^2
\right)
\leq
\exp\!\left\{-cn\varepsilon_n^2(\Pi)\right\}
}
\geq
1-\frac{4}{n^2},
\]
where $\varepsilon_n(\Pi)$ is defined in~\eqref{eq:epsilon_n}.
\end{theo}

Thus, allowing the sampling variances to differ across units does not by itself slow posterior contraction relative to Theorem~\ref{thm:compound-contraction}. The same conclusion carries over to regret.

\begin{theo}[Hierarchical regret]
\label{theo:regret_bayes_ind_gauss}
Suppose the conditions of Theorem~\ref{theo:het-ind-contraction} hold.
Then there exists a constant $C>0$, depending only on the specification, $M$, $\ubar\sigma^2$, and $\bar\sigma^2$, such that, for all $n\geq2$,
\[
\sup_{G_\star\in\mathcal P([-M,M])}
\left\{
R\!\left(\widehat{\boldmu}(\Pi),G_\star\right)
-
R\!\left(\widehat{\boldmu}^{\mathrm{B}}(G_\star),G_\star\right)
\right\}
\leq
C\,\varepsilon_n(\Pi)\log^{3/2}n.
\]
\end{theo}

Analogous regret bounds for NPMLE plug-in estimators were established by~\citet{jiang2020general} and~\citet{Soloff}. Theorem~\ref{theo:regret_bayes_ind_gauss} shows that the fully Bayesian estimator attains the corresponding benchmark rate.

\section{Simulations} \label{sec-sims}

We first conduct a numerical study in the homoscedastic setting of
Section~\ref{sec:homoscedastic_gaussian}. Simulated datasets are generated
according to both~\eqref{eq:hierarchy1N}
and~\eqref{eq:hierarchy2N}, so that
there is a true latent distribution $G_{\star}$. More precisely, we let \smash{$\mu_i \simiid G_{\star}$} and \smash{$Z_i \simindep \mathrm{N}(\mu_i, 1)$} for $i=1,\ldots,n$. We consider
four choices of $G_{\star}$:
$$
\begin{alignedat}{4}
&\text{Gaussian:}\;\;  && G_{\star}=\mathrm{N}(0,V), \qquad
&&\text{Gamma:}\;\;    && G_{\star}=\mathrm{Gamma}(6\sqrt{V},3),\\
&\text{Two-point:}\;\; && G_{\star}=0.9\,\delta_0+0.1\,\delta_{\sqrt{10V}}, \qquad
&&\text{Bimodal:}\;\;  && G_{\star}=0.5\,\mathrm{N}(-2,V)+0.5\,\mathrm{N}(2,V).
\end{alignedat}
$$

Across settings, we vary $V\in\{0.1,1,2\}$ and the sample size
$n\in\{200,1000\}$.  Each setting is evaluated using $1000$ Monte Carlo
replications. The first three specifications comprise a subset of the
data generating processes in the study of~\citet{koenker2020empirical}.

We compare the following procedures:
\begin{itemize}[leftmargin=*,nosep]
\item \textbf{Oracle.} The Bayes rule computed under the true latent
distribution $G_{\star}$.
\item \textbf{NPMLE.} The empirical Bayes procedure based on the
Kiefer--Wolfowitz NPMLE $\widehat G$.  We use the by-now-standard approach of~\citet{koenker2014convex} and~\citet{koenker2017rebayes} using discretization and the MOSEK~\citep{mosek} interior point solver.
\item \textbf{AKP.} The robust empirical Bayes procedure of
\citet{armstrong2022robust}. Its point estimate is the empirical Bayes
posterior mean obtained under a Gaussian working model for $G_{\star}$.
\item \textbf{DPGM.} Our fully Bayesian procedure under the DPGM prior in
Specification~\ref{assu:DPGM}.
We make the following empirical hyperprior choices. We first let $A \sim \mathrm{Beta}(2,5)$.\footnote{This prior satisfies condition~\eqref{eq:prior_var_tail} with $r=2$ and $\kappa=0$, and so leads to fast regret rates in Theorem~\ref{theo:main_risk_bound_eq}.
}
We next choose $H=\mathrm{N}\{\bar{\boldZ},\, 9 \widehat{\mathrm{Var}}(\boldZ)\}$, where $\bar{\boldZ}$ and \smash{$\widehat{\mathrm{Var}}(\boldZ)$} denote the sample average and sample variance of $Z_1,\ldots,Z_n$.
Regarding the concentration parameter of the Dirichlet process, we follow
the approach of~\citet{escobar1995bayesian} with the specific hyperparameter choices of~\citet{hirano2002semiparametric}, setting the hyperprior as $\alpha \sim \mathrm{Gamma}(2,0.5)$ (with a shape-rate parameterization). For computation, we use the Gibbs sampler ($5,000$ burn-in, $10,000$ post-burn-in iterations) of Algorithm 2 in~\citet{neal2000markov} along with updates for $\alpha$ described in~\citet{escobar1995bayesian}.
\\
\end{itemize}
We evaluate the methods first in terms of estimation of the $\mu_i$ in squared error, reporting RMSE. We also consider interval estimation for the individual parameters $\mu_i$: following~\citet{armstrong2022robust}, for each method of constructing nominal $95\%$ intervals $\mathcal{I}_i$, we report (Monte Carlo estimates of) average coverage ($n^{-1}\sum_{i=1}^n \PP{\mu_i \in \mathcal{I}_i}$) and average interval length ($n^{-1}\sum_{i=1}^n \EE{|\mathcal{I}_i|}$).

\begin{table}
\centering
\caption{Point estimation and interval performance across latent distributions $G_{\star}$.}
\label{tab:sim-summary-main}
\renewcommand{\arraystretch}{0.85}
\setlength{\tabcolsep}{5pt}
\begin{threeparttable}
\begin{tabular}{lccccccccc}
\toprule
 & \multicolumn{3}{c}{$V = 0.1$} & \multicolumn{3}{c}{$V = 1.0$} & \multicolumn{3}{c}{$V = 2.0$} \\
\cmidrule(lr){2-4} \cmidrule(lr){5-7} \cmidrule(lr){8-10}
Method & RMSE & Cov. & Len. & RMSE & Cov. & Len. & RMSE & Cov. & Len. \\
\midrule
\multicolumn{10}{l}{\textit{Panel A.} Gaussian, $n = 200$} \\
Oracle & 0.30 & 0.95 & 1.18 & 0.71 & 0.95 & 2.77 & 0.81 & 0.95 & 3.20 \\
NPMLE & 0.33 & 0.45 & 0.66 & 0.73 & 0.71 & 1.96 & 0.84 & 0.73 & 2.32 \\
AKP & 0.32 & 0.95 & 1.75 & 0.71 & 0.95 & 2.78 & 0.82 & 0.95 & 3.20 \\
DPGM & 0.32 & 0.98 & 1.65 & 0.72 & 0.94 & 2.76 & 0.83 & 0.95 & 3.19 \\
\addlinespace[0.1em]
\multicolumn{10}{l}{\textit{Panel B.} Gaussian, $n = 1000$} \\
Oracle & 0.30 & 0.95 & 1.18 & 0.71 & 0.95 & 2.77 & 0.82 & 0.95 & 3.20 \\
NPMLE & 0.31 & 0.56 & 0.74 & 0.72 & 0.79 & 2.18 & 0.83 & 0.80 & 2.56 \\
AKP & 0.31 & 0.97 & 1.63 & 0.71 & 0.95 & 2.80 & 0.82 & 0.95 & 3.21 \\
DPGM & 0.31 & 0.96 & 1.32 & 0.71 & 0.95 & 2.75 & 0.82 & 0.95 & 3.18 \\
\addlinespace[0.1em]
\multicolumn{10}{l}{\textit{Panel C.} Discrete two-point, $n = 200$} \\
Oracle & 0.28 & 1.00 & 0.85 & 0.46 & 1.00 & 0.56 & 0.30 & 1.00 & 0.19 \\
NPMLE & 0.32 & 0.44 & 0.57 & 0.49 & 0.38 & 0.87 & 0.35 & 0.34 & 0.51 \\
AKP & 0.30 & 0.96 & 1.64 & 0.69 & 0.95 & 3.07 & 0.80 & 0.95 & 3.29 \\
DPGM & 0.31 & 0.96 & 1.61 & 0.51 & 0.99 & 2.14 & 0.36 & 1.00 & 1.76 \\
\addlinespace[0.1em]
\multicolumn{10}{l}{\textit{Panel D.} Discrete two-point, $n = 1000$} \\
Oracle & 0.28 & 1.00 & 0.85 & 0.47 & 1.00 & 0.56 & 0.32 & 1.00 & 0.19 \\
NPMLE & 0.30 & 0.56 & 0.69 & 0.48 & 0.46 & 0.80 & 0.33 & 0.48 & 0.45 \\
AKP & 0.29 & 0.95 & 1.66 & 0.69 & 0.95 & 3.09 & 0.80 & 0.95 & 3.29 \\
DPGM & 0.29 & 0.93 & 1.29 & 0.48 & 0.99 & 1.61 & 0.33 & 1.00 & 1.23 \\
\addlinespace[0.1em]
\multicolumn{10}{l}{\textit{Panel E.} Gamma, $n = 200$} \\
Oracle & 0.41 & 0.95 & 1.51 & 0.62 & 0.95 & 2.39 & 0.69 & 0.95 & 2.66 \\
NPMLE & 0.44 & 0.51 & 0.92 & 0.65 & 0.67 & 1.64 & 0.72 & 0.69 & 1.87 \\
AKP & 0.43 & 0.96 & 2.07 & 0.64 & 0.95 & 2.58 & 0.70 & 0.95 & 2.77 \\
DPGM & 0.43 & 0.97 & 1.84 & 0.63 & 0.95 & 2.46 & 0.70 & 0.94 & 2.69 \\
\addlinespace[0.1em]
\multicolumn{10}{l}{\textit{Panel F.} Gamma, $n = 1000$} \\
Oracle & 0.41 & 0.95 & 1.51 & 0.62 & 0.95 & 2.39 & 0.69 & 0.95 & 2.66 \\
NPMLE & 0.42 & 0.62 & 1.08 & 0.63 & 0.75 & 1.82 & 0.70 & 0.77 & 2.07 \\
AKP & 0.42 & 0.97 & 2.02 & 0.63 & 0.96 & 2.62 & 0.70 & 0.96 & 2.82 \\
DPGM & 0.42 & 0.96 & 1.67 & 0.63 & 0.95 & 2.43 & 0.69 & 0.95 & 2.67 \\
\addlinespace[0.1em]
\multicolumn{10}{l}{\textit{Panel G.} Bimodal Gaussian mixture, $n = 200$} \\
Oracle & 0.60 & 0.95 & 1.73 & 0.86 & 0.95 & 3.29 & 0.91 & 0.95 & 3.55 \\
NPMLE & 0.63 & 0.43 & 1.18 & 0.89 & 0.70 & 2.34 & 0.94 & 0.72 & 2.57 \\
AKP & 0.90 & 0.95 & 3.50 & 0.92 & 0.95 & 3.57 & 0.93 & 0.95 & 3.62 \\
DPGM & 0.62 & 0.98 & 2.30 & 0.87 & 0.94 & 3.27 & 0.92 & 0.94 & 3.53 \\
\addlinespace[0.1em]
\multicolumn{10}{l}{\textit{Panel H.} Bimodal Gaussian mixture, $n = 1000$} \\
Oracle & 0.61 & 0.95 & 1.74 & 0.85 & 0.95 & 3.28 & 0.91 & 0.95 & 3.55 \\
NPMLE & 0.62 & 0.51 & 1.27 & 0.86 & 0.79 & 2.60 & 0.92 & 0.80 & 2.84 \\
AKP & 0.90 & 0.95 & 3.50 & 0.91 & 0.95 & 3.57 & 0.93 & 0.95 & 3.63 \\
DPGM & 0.61 & 0.96 & 1.93 & 0.86 & 0.95 & 3.26 & 0.91 & 0.95 & 3.53 \\
\bottomrule
\end{tabular}
\begin{tablenotes}[flushleft]
\small
\item \textit{Notes:} Entries report averages over units ($n$) and Monte Carlo replications ($1,000$). For the two-point design, equal-tailed posterior intervals may have posterior probability above $0.95$ (due to discreteness), explaining Oracle coverage of $1.00$.
\end{tablenotes}
\end{threeparttable}
\end{table}

Next, we describe the interval constructions. If $G_{\star}$ were
known, let
$
q_{G_{\star}}(\tau\mid z)
:=
\inf\left\{
t\in\RR:
\PP[G_{\star}]{\mu \leq t \mid Z=z}
\geq \tau
\right\}
$
denote the $\tau$-quantile of the posterior distribution of $\mu_i$ given
$Z_i=z$. A natural $95\%$ posterior interval is then
$
\mathcal{I}_{G_{\star}}(Z_i) := [
q_{G_{\star}}(0.025\mid Z_i),\,
q_{G_{\star}}(0.975\mid Z_i)].$
The plug-in empirical Bayes construction (e.g., using the NPMLE) replaces $G$ by an estimate
\smash{$\widehat G$}, yielding the $i$-th interval $\mathcal{I}_{\widehat{G}}(Z_i)$.
By contrast, the fully Bayesian interval is given by the $0.025$ and $0.975$
quantiles of the posterior distribution of $\mu_i\mid\boldZ$, which integrates
over posterior uncertainty about $G$.
This distinction is most transparent by considering the posterior variance of $\mu_i$ under the fully Bayesian model,
\begin{equation}
\Var[\Pi]{\mu_i\mid\boldZ}
=
\EE[\Pi]{\Var[G]{\mu_i\mid Z_i}\mid\boldZ}
+
\Var[\Pi]{\EE[G]{\mu_i\mid Z_i}\mid\boldZ}.
\label{eq:total_variance}
\end{equation}
The plug-in EB construction fixes $G$ at $\widehat G$ and
therefore does not account for the second source of uncertainty, which
reflects uncertainty about the mixing distribution itself.\footnote{The issue is classical in empirical Bayes. For parametric models,
\citet{morris1983parametric} notes that the plug-in intervals of
\citet{cox1975prediction} can under cover in small samples and advocates for
hierarchical Bayes corrections; see also the bootstrap approach of
\citet{laird1987empirical}, which is used e.g., by
\citet{efron2019bayes} and \citet{jiang2019comment}.}

Turning to the AKP intervals, they effectively start from a Gaussian working
model for $G_\star$, but widen the corresponding plug-in empirical Bayes
interval so as to guarantee asymptotic average coverage even when the Gaussian
specification is incorrect. Unlike the fully Bayesian construction, however,
they do not obtain the interval by integrating over uncertainty in the
estimated prior parameters.\footnote{AKP nevertheless account for estimation
error in the prior moments entering their procedure, including finite-sample
adjustments when estimating the prior variance. These adjustments mimic a Bayesian approach, and the authors refer to them as flat prior limited information Bayes.}

\subsection{Results}

Table~\ref{tab:sim-summary-main} reports the results. We first consider point estimation of the $\mu_i$. When $G_\star$ is Gaussian, all three empirical procedures have similar RMSE, with performance approaching that of the oracle as $n$ increases. The same is broadly true under the Gamma specification. Under the two-point and bimodal specifications, the NPMLE and DPGM continue to perform similarly, while the Gaussian working model underlying AKP can lead to noticeably larger RMSE. Practically this means that for these two specifications, it is important to allow for nonlinear shrinkage rules as provided by the NPMLE and DPGM (but not by AKP).

The picture is quite different for interval estimation. The plug-in NPMLE intervals substantially undercover in every setting considered here. While coverage generally improves with $n$, the undercoverage remains pronounced even at $n=1000$. The undercoverage phenomenon for the NPMLE is well known and has been noted by other authors. \citet{koenker2020empirical} attributes it to the discrete nature of the NPMLE: since $\widehat{G}$ in~\eqref{eq:npmle_and_marginal} is discrete (with  $O(\log n)$ support points~\citep{polyanskiy2020selfregularizinga}), the plug-in posterior distributions are discrete as well, and any intervals derived from them are therefore sensitive to the location of its mass points. A further explanation is based on the law of total variance in~\eqref{eq:total_variance} and the fact that the NPMLE ignores the second source of uncertainty, as discussed above.
Meanwhile, the DPGM intervals instead have coverage close to the nominal level throughout, ranging from $0.93$ to $1.00$ across the specifications in Table~\ref{tab:sim-summary-main}.

The AKP intervals also attain coverage close to the nominal level, but the relative interval lengths depend strongly on $G_\star$. Under the Gaussian specification, the DPGM and AKP intervals have similar lengths, particularly for the larger values of $V$, while DPGM is shorter for $V=0.1$. The difference is much larger under the two-point specification. For example, when $n=1000$ and $V=2$, the average interval length is $1.23$ for DPGM compared with $3.29$ for AKP. A similar pattern occurs for the bimodal specification (e.g., for $V=0.1$, DPGM has an average length of $1.93$ vs.\ $3.5$ for AKP).

\section{Empirical Application} \label{sec-atlas}
In this section, we apply our methodology to Census tract-level estimates of economic mobility. The ``Opportunity Atlas''~\citep{chetty2026opportunity} reports estimates $Z_i$ of children's adult outcomes by Census tract $i$, along with corresponding standard errors $\sigma_i$. In this setting, $Z_i$ is interpreted as a noisy estimate of an unobserved economic mobility parameter $\mu_i$. Our analysis revisits the empirical exercise in \citet{chen2026empirical}, where the decision rule $\widehat{\boldmu}(\boldZ, \boldsigma^2)$ is estimated using plug-in empirical Bayes procedures based on nonparametric regression and nonparametric maximum likelihood (NPMLE).

For all the outcomes considered, we model the latent mobility parameter $\mu_i$ using the conditional nonparametric location-scale specification from Section~\ref{subsec:CLOSE} with
\smash{$Z_i \mid \mu_i, \sigma_i \simindep \mathrm{N}(\mu_i,\sigma_i^2)$},
$\mu_i = m_{\star}(\sigma_i) + s_{\star}(\sigma_i)\zeta_i$ and $
\zeta_i \mid \sigma_i \simindep G_{\star}$.
Following the methodology in Section~\ref{subsec:CLOSE}, we place Gaussian process priors on the unknown functions $m_{\star}$ and $s_{\star}$, and a Dirichlet process Gaussian mixture prior on $G_{\star}$. Further implementation details are provided in Appendix~\ref{implement}.

The section is organized as follows. Section~\ref{results} presents the empirical results for the primary Opportunity Atlas outcome considered in~\citet{chen2026empirical}, while Section~\ref{calibrated} complements the analysis with simulations calibrated to all fifteen outcomes in their study.

Throughout, we use ``CLOSE'' to denote the plug-in empirical Bayes procedure of~\citet{chen2026empirical}. We emphasize that~\citet{chen2026empirical} develops and recommends CLOSE for denoising, for which their method performs well; see below. We additionally use the fitted CLOSE model in the natural plug-in fashion for downstream interval inference and pairwise ranking problems, treating the estimated shrinkage structure as fixed. These latter tasks are not considered by~\citet{chen2026empirical}, but illustrate how the fitted empirical Bayes model would be used in subsequent decision problems.

\subsection{Results}
\label{results}
We focus on the primary outcome analyzed in~\citet{chen2026empirical}: the probability that a Black child whose parents were at the 25th percentile of the national income distribution attains family income in the top 20 percent of the national distribution in adulthood. We use their sample of $n=10{,}056$ Census tracts, which covers the 20 largest Commuting Zones (measured by the number of tracts).

Figure~\ref{fig:mean_variance} compares the GP posterior estimates of the conditional mean $m(\sigma)$ and conditional standard deviation $s(\sigma)$ with the corresponding plug-in estimates from CLOSE. In the left panel, the two methods produce nearly identical estimates of $m(\sigma)$, and the GP credible band is relatively tight over most of the observed range of $\sigma$, suggesting that $m(\sigma)$ is well identified by the data.

The conditional standard deviation $s(\sigma)$ in the right panel of Figure~\ref{fig:mean_variance} is less precisely estimated, and the differences between the two approaches are more visible. CLOSE obtains $s^2(\sigma)$ by subtracting the known sampling variance from a nonparametric estimate of the residual variance, with a lower truncation that keeps the estimate bounded away from zero.\footnote{Specifically, CLOSE applies a local linear regression to the squared residuals $(Z_i-\widehat m(\sigma_i))^2$, subtracts $\sigma_i^2$, and truncates the resulting estimate below at a small positive floor.} When these two variance components are similar, the resulting estimate can be sensitive to estimation error, which is reflected near the edges of the $\sigma$ range where the truncation becomes relevant. The Bayesian procedure instead estimates $s(\sigma)$ jointly with the other model components through the likelihood. The posterior mean remains smooth, while the credible band widens  in regions where the data are less informative about the latent scale.

\begin{figure}
\centering
\includegraphics[width=\linewidth]{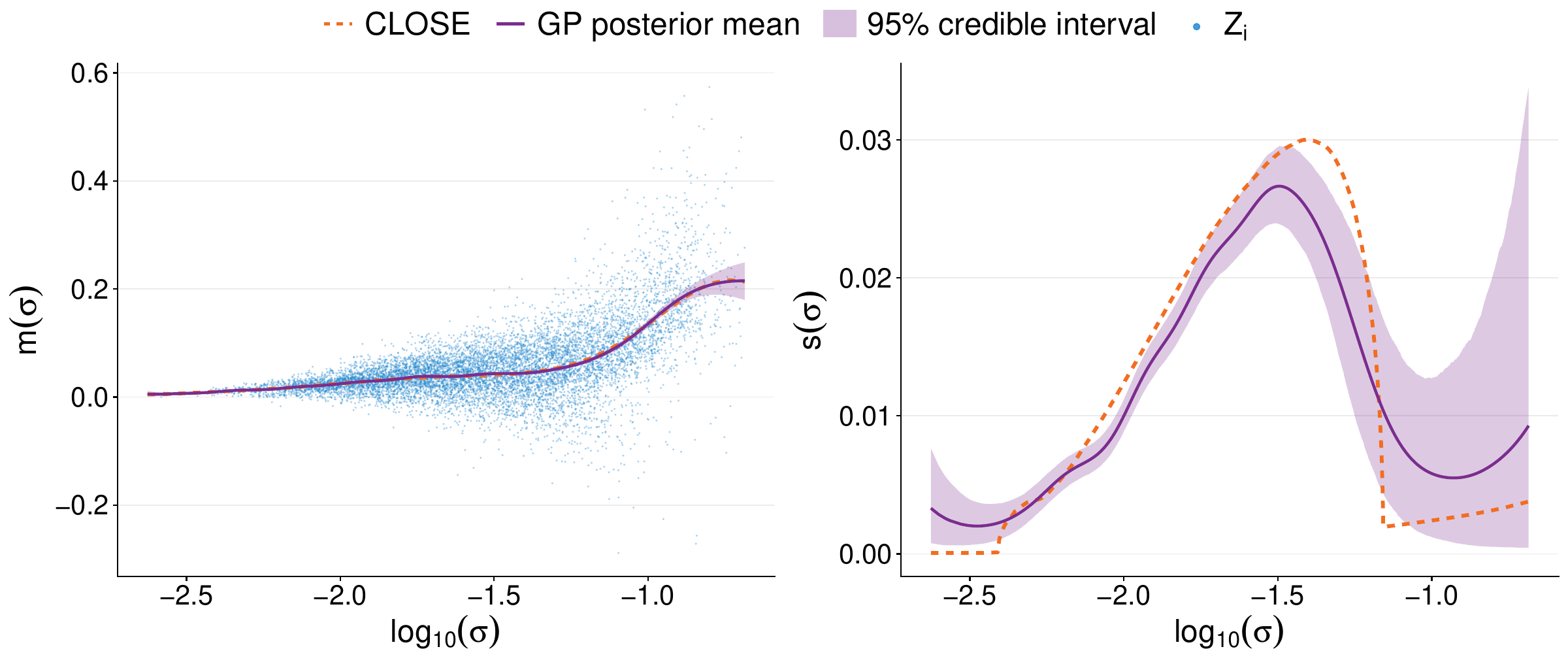}
\caption{Posterior estimates of the conditional mean $m(\sigma)$ (left) and conditional standard deviation $s(\sigma)$ (right), compared with the CLOSE plug-in estimates. Blue points in the left panel are the raw estimates $Z_i$.}
\label{fig:mean_variance}
\end{figure}

\begin{figure}
\centering
\includegraphics[width=\linewidth]{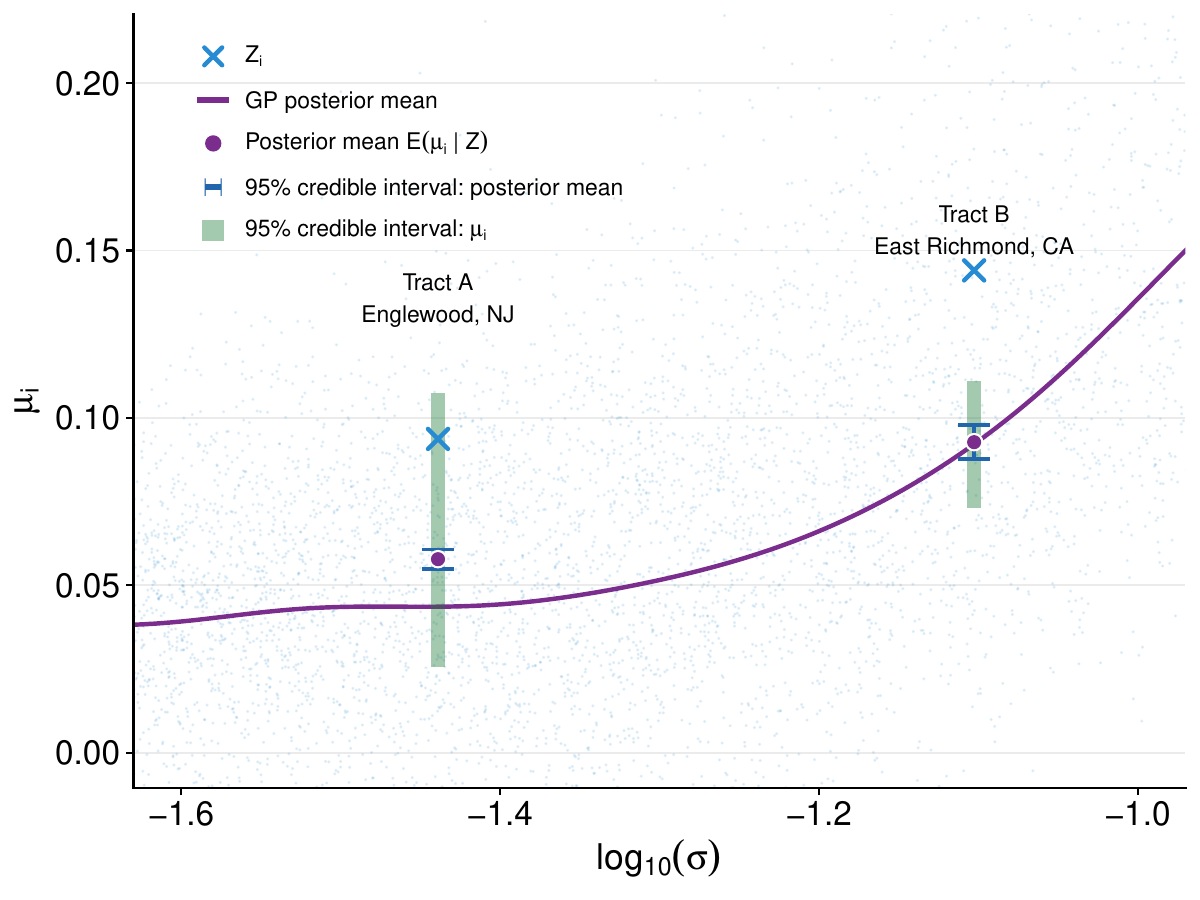}
\caption{Posterior inference for Census tracts.}
\label{fig:tract_comparison}
\end{figure}

\begin{table}
\centering
\caption{CLOSE and Bayesian estimates for two Census tracts.}
\label{tab:tracts}
\setlength{\tabcolsep}{5.7pt}
\begin{tabular}{lcccc}
\toprule
 & \multicolumn{2}{c}{Tract A: Englewood, NJ} & \multicolumn{2}{c}{Tract B: East Richmond, CA} \\
\cmidrule(lr){2-3} \cmidrule(lr){4-5}
 & CLOSE & Bayesian & CLOSE & Bayesian \\
\midrule
$Z_i$ & \multicolumn{2}{c}{0.094} & \multicolumn{2}{c}{0.144} \\
$\sigma_i$ & \multicolumn{2}{c}{0.036} & \multicolumn{2}{c}{0.079} \\
\midrule
Prior mean $\mathbb{E}[\mu_i \mid \sigma_i]$ & 0.045 & 0.044 & 0.098 & 0.092 \\
\quad 95\% credible interval & --- & [0.041, 0.046] & --- & [0.087, 0.097] \\
\midrule
Posterior mean $\mathbb{E}[\mu_i \mid Z_i, \sigma_i]$ & 0.058 & 0.058 & 0.098 & 0.093 \\
\quad 95\% credible interval & --- & [0.055, 0.061] & --- & [0.088, 0.098] \\
\midrule
95\% credible interval for $\mu_i$ & [0.049, 0.092] & [0.026, 0.107] & [0.095, 0.102] & [0.073, 0.111] \\
\midrule
Posterior odds of $\mu_B > \mu_A$ & \multicolumn{4}{c}{CLOSE: 128:1 \qquad\qquad Bayesian: 12:1} \\
\bottomrule
\end{tabular}
\vspace{0.5em}
\parbox{\textwidth}{\small \textit{Notes:} $Z_i$ is the raw estimate for tract $i$ with standard error $\sigma_i$; $\mu_i$ is the tract's latent mobility parameter. CLOSE is the plug-in empirical Bayes method of~\citet{chen2026empirical}. For CLOSE, the ``$95\%$ credible interval for $\mu_i$'' and the posterior odds are computed from the plug-in empirical Bayes posterior, treating the estimated shrinkage structure as known and the two tracts as independent. For the Bayesian procedure, the posterior odds are computed from the joint posterior of $(\mu_A,\mu_B)$, which accounts for their dependence through the shared shrinkage structure.}
\end{table}

Next, we turn to posterior inference for the unobserved economic mobility parameter $\mu_i$. Our interest in $\mu_i$ is motivated in part by applied settings in which empirical Bayes estimates are used as inputs for downstream decision problems. In the Opportunity Atlas context, for example, \citet{bergman2024creating} use empirical Bayes posterior means to identify Census tracts with high upward-mobility potential. To illustrate the main idea, we focus on the two Census tracts considered in \citet{chen2026empirical}: one in Englewood, NJ, and one in Richmond, CA, which we refer to as tracts A and B, respectively.

Figure~\ref{fig:tract_comparison} illustrates the resulting shrinkage. Both raw estimates lie above the fitted conditional mean and are therefore pulled downward toward it. The adjustment is larger for tract B: its raw estimate, $Z_B=0.144$, is shrunk to a posterior mean of $0.093$, whereas tract A is adjusted more modestly, from $0.094$ to $0.058$. As Table~\ref{tab:tracts} shows, CLOSE and the Bayesian procedure produce very similar point estimates for both tracts. This is consistent with the calibrated simulations in Section~\ref{calibrated}, where the two procedures also exhibit similar point-estimation risk. Moreover, the Bayesian credible intervals for the posterior means\footnote{To elaborate, recall that the oracle rule $\widehat{\mu}_i^{\mathrm{B}}(\star)$ in~\eqref{eq:close_bayes_rule} depends on the unknown $\star=(G_\star,m_\star,s_\star)$. The posterior for $\star$ therefore induces a posterior distribution for $\widehat{\mu}_i^{\mathrm{B}}(\star)$. Its posterior mean is precisely the Bayes estimator $\widehat{\mu}_i(\Pi)$, while its posterior quantiles provide the credible intervals reported here.
} are relatively tight.

A more noticeable difference between the two approaches appears in their assessment of uncertainty about the latent mobility parameters. In applied work, this uncertainty is often summarized using plug-in empirical Bayes intervals, which estimate the prior or shrinkage rule and then construct posterior intervals while treating the estimated objects as known; see, for example, \citet{kline2024discrimination}. For the two tracts considered here, the CLOSE plug-in intervals are narrow and disjoint, suggesting that their latent mobility outcomes can be ranked with relatively little uncertainty. By contrast, the Bayesian credible intervals for $\mu_A$ and $\mu_B$ are wider and overlap substantially. To assess the ranking directly, Table~\ref{tab:tracts} also reports the posterior odds that tract B has the higher latent mobility outcome, defined as
$$\gamma_{\text{}}(\mu_A,\mu_B) =
\Pr(\mu_B>\mu_A\mid\boldZ,\boldsigma^2)\big/
\Pr(\mu_B\leq\mu_A\mid\boldZ,\boldsigma^2).
$$
These odds are $128:1$ under the CLOSE plug-in posterior, compared with $12:1$ under the  full Bayesian posterior. Thus, both procedures favor tract B over tract A, but accounting for uncertainty in the shared shrinkage structure leads to a substantially more cautious assessment of the ranking. It is worth noting that this difference remains sizeable even with $n=10{,}056$ tracts. In applications with only a small or moderate number of related units, uncertainty about the shrinkage structure is likely to play an even larger role in downstream decisions~\citep{ignatiadis2022confidence, gu2025reasonable}.

\begin{figure}
\centering
\includegraphics[width=\linewidth]{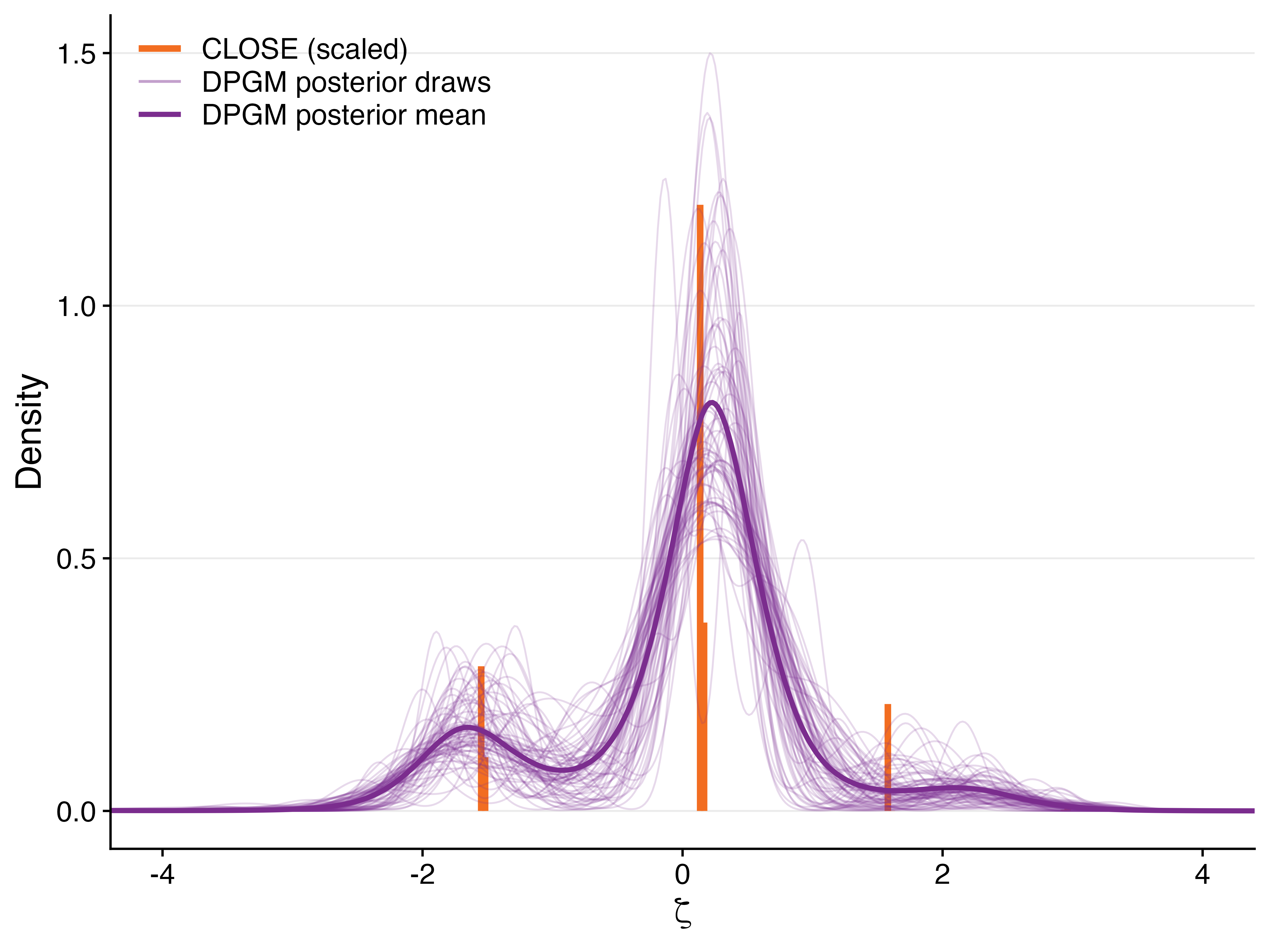}
\caption{Posterior distribution of the standardized shape distribution $G_{\star}$. Vertical orange lines denote the atoms of the discrete CLOSE-NPMLE estimate, with heights proportional to their estimated masses and scaled for visibility on the density axis.}
\label{fig:prior_G}
\end{figure}

Finally, Figure~\ref{fig:prior_G} illustrates the posterior uncertainty about the latent distribution $G_{\star}$. The figure plots posterior draws from the DPGM, together with the posterior mean and the corresponding CLOSE estimate, on a common standardized scale.\footnote{For comparability, each posterior draw is transformed ex post to have mean zero and variance one, as described in Remark~\ref{normalize}.} The two approaches point to a similar structure, with a dominant central mode and additional mass on either side. The posterior draws nevertheless reveal considerable uncertainty about the precise shape of the distribution. The location of the central mode is fairly stable across draws, but its height and sharpness vary substantially, while the smaller features away from the center also differ in their location and prominence. At the same time, the DPGM posterior draws remain smooth and continuous, reflecting the smoothing built into the Gaussian mixture prior specification.

\subsection{Calibrated simulations}
\label{calibrated}

We complement the empirical analysis using the calibrated simulation designs in~\citet{chen2026empirical}, covering all fifteen tract-level outcomes in their study. These comprise three measures of economic outcomes: mean adult income rank, incarceration, and top-quintile attainment, each reported for five subpopulations defined by race and gender.\footnote{Outcome names have two components. The first identifies the measure: \texttt{kfr} denotes mean household income rank in adulthood, \texttt{jail} denotes incarceration status  and \texttt{kfr\_top20} denotes attainment of the top quintile of household income. The remainder identifies the subpopulation over which the tract-level average is taken: all children (\texttt{pooled}), Black or white children (\texttt{black\_pooled}, \texttt{white\_pooled}), or Black or white men (\texttt{black\_male}, \texttt{white\_male}). All outcomes refer to children raised in the tract with parents at the 25th percentile of the national parental income distribution.} We use the same prior specification as in the empirical application for all fifteen outcomes, without any outcome-specific tuning. As in Section~\ref{results}, we do not use covariate information: both the data-generating process and the estimators condition only on the pairs $(Z_i,\sigma_i^2)$.\footnote{This corresponds to the ``no residualization'' benchmark in~\citet{chen2026empirical}.}

For each outcome, the data-generating process is calibrated as follows. First, the conditional location and scale functions $\widehat m(\sigma)$ and $\widehat s(\sigma)$ are estimated from the observed pairs $(Z_i,\sigma_i)$ by the local linear regression in \citet{chen2026empirical}.\footnote{Following \citet{chen2026empirical}, the calibration uses CLOSE's
estimated conditional moments, including the lower truncation of
$\widehat s^2(\sigma)$ described in Section~\ref{results}. Because the
simulated means are generated as
$\mu_i = \widehat m(\sigma_i) + \widehat s(\sigma_i)\,\zeta_i$ with $\zeta_i$
standardized, the dispersion of the true means at noise level $\sigma$ is
exactly $\widehat s(\sigma)$: wherever the truncation binds, the lower bound
holds exactly in the simulated data. Thus the simulations do not speak to the role
of the truncation in the  empirical application of Section \ref{results}. }  Second, tracts are partitioned into twenty bins according to $\sigma_i$, and within each bin $b$ a distribution $\widehat G_b$ is estimated by NPMLE. Each simulated dataset is then generated according to
\begin{align*}
\zeta_i \simindep \widehat G_{b(i)}, \; \; \;
\mu_i = \widehat m(\sigma_i) + \widehat s(\sigma_i)\,\zeta_i, \; \; \;
Z_i \mid \mu_i, \sigma_i \simindep \mathrm{N}(\mu_i, \sigma_i^2),
\end{align*}
where $b(i)$ denotes the bin containing tract $i$, and the standard errors $\sigma_i$ are held fixed at their observed values. The resulting designs preserve the empirically estimated dependence of the latent location and scale on $\sigma_i$ while allowing the shape of the latent distribution to vary across outcomes and across regions of the $\sigma_i$ distribution. Moreover, the simulation design does not impose the working model of Section~\ref{subsec:CLOSE}: $\widehat G_b$ varies across the twenty bins rather than being common across tracts. The exercise therefore also assesses the procedures under this form of misspecification.

For each simulated dataset, we compare four procedures: (i)~Independent-NPMLE, which estimates a single distribution for $\mu_i$ by NPMLE and ignores its dependence on $\sigma_i$; (ii)~CLOSE; (iii)~the Bayesian procedure; and (iv)~the oracle posterior mean, which knows the conditional location and scale functions and the bin-specific shape distributions used to generate the data. For point estimation, we report the fraction of the naive-to-oracle MSE gap captured, averaged across replications. For a given method, it is defined as
\[ \text{naive-to-oracle gap} =
(\mathrm{MSE}_{\mathrm{naive}}-\mathrm{MSE}_{\mathrm{method}})/
(\mathrm{MSE}_{\mathrm{naive}}-\mathrm{MSE}_{\mathrm{oracle}}),
\]
so that the naive estimator $\mu_i = Z_i$ receives $0\%$ and the oracle receives $100\%$.

We also evaluate uncertainty quantification using the coverage of nominal $95\%$ intervals for $\mu_i$, averaged across tracts and replications. For the Bayesian procedure, we use equal-tailed posterior credible intervals, which incorporate posterior uncertainty in the location and scale functions and the latent distribution. For CLOSE and Independent-NPMLE, we use the plug-in empirical Bayes intervals described in Section~\ref{results}, which condition on the estimated shrinkage structure. Table~\ref{tab:calibrated} reports the results.

\begin{table}[t!]
\centering
\caption{Calibrated simulations: point estimation, coverage, and interval width}
\label{tab:calibrated}
\setlength{\tabcolsep}{3pt}
\begin{tabular}{lc@{\hspace{5pt}}c@{\hspace{5pt}}ccccc}
\toprule
 & \multicolumn{3}{c}{\% of naive-to-oracle gap} & \multicolumn{3}{c}{Coverage in \% (average width)} & Width \\
\cmidrule(lr){2-4} \cmidrule(lr){5-7}
Outcome & Indep. & CLOSE & Bayes & Indep. & CLOSE & Bayes & ratio \\
\midrule
\multicolumn{8}{l}{\emph{Mean income rank}} \\
\texttt{kfr\_pooled} & 41.6 & 81.2 & 81.9 & 89.2 (0.064) & 88.3 (0.062) & 94.8 (0.067) & 1.09 \\
\texttt{kfr\_black\_pooled} & 67.6 & 94.5 & 94.7 & 89.8 (0.097) & 88.6 (0.092) & 95.0 (0.100) & 1.08 \\
\texttt{kfr\_white\_pooled} & 87.8 & 92.3 & 92.5 & 88.6 (0.115) & 86.5 (0.113) & 94.8 (0.128) & 1.13 \\
\texttt{kfr\_black\_male} & 60.5 & 94.6 & 94.9 & 86.6 (0.102) & 87.5 (0.092) & 94.6 (0.103) & 1.11 \\
\texttt{kfr\_white\_male} & 87.9 & 93.9 & 94.0 & 87.6 (0.133) & 84.0 (0.129) & 94.5 (0.148) & 1.15 \\
\midrule
\multicolumn{8}{l}{\emph{Incarceration}} \\
\texttt{jail\_pooled} & 38.7 & 97.9 & 98.0 & 88.0 (0.033) & 64.8 (0.019) & 97.0 (0.023) & 1.25 \\
\texttt{jail\_black\_pooled} & 52.8 & 98.2 & 98.8 & 78.7 (0.053) & 87.1 (0.053) & 96.0 (0.058) & 1.11 \\
\texttt{jail\_white\_pooled} & 77.3 & 98.0 & 97.6 & 67.6 (0.027) & 61.6 (0.026) & 97.8 (0.036) & 1.42 \\
\texttt{jail\_black\_male} & 47.6 & 98.0 & 98.6 & 78.8 (0.103) & 86.4 (0.079) & 95.9 (0.086) & 1.10 \\
\texttt{jail\_white\_male} & 54.5 & 98.7 & 98.3 & 71.6 (0.045) & 50.2 (0.031) & 98.4 (0.048) & 1.54 \\
\midrule
\multicolumn{8}{l}{\emph{Top-quintile attainment}} \\
\texttt{kfr\_top20\_black\_pooled} & 27.8 & 98.4 & 99.1 & 76.1 (0.070) & 83.6 (0.065) & 96.1 (0.066) & 1.01 \\
\texttt{kfr\_top20\_pooled} & 8.0 & 91.1 & 91.4 & 87.8 (0.079) & 84.6 (0.065) & 94.9 (0.076) & 1.17 \\
\texttt{kfr\_top20\_white\_pooled} & 73.4 & 96.9 & 97.2 & 88.2 (0.150) & 84.7 (0.125) & 95.0 (0.147) & 1.17 \\
\texttt{kfr\_top20\_black\_male} & 22.7 & 98.6 & 99.2 & 73.7 (0.065) & 69.9 (0.044) & 97.0 (0.052) & 1.19 \\
\texttt{kfr\_top20\_white\_male} & 59.3 & 98.0 & 98.1 & 86.2 (0.169) & 83.3 (0.121) & 95.3 (0.143) & 1.18 \\
\bottomrule
\end{tabular}

\vspace{0.5em}
\parbox{\textwidth}{\small \textit{Notes:} Results are averages over 500 simulation replications per outcome. The middle three columns report the coverage of nominal 95\% intervals for $\mu_i$, with the average interval width on the scale of the outcome in parentheses: plug-in empirical Bayes intervals for Independent-NPMLE and CLOSE, credible intervals for the Bayesian procedure. ``Width ratio'' is the ratio of the average Bayesian width to the average CLOSE width.}
\end{table}
For point estimation, CLOSE and the Bayesian procedure perform nearly identically, with a slight edge for the Bayesian procedure: its performance is at least as high as CLOSE's for thirteen of the fifteen outcomes, although the differences never exceed $0.7$ percentage points in either direction. Both procedures substantially outperform Independent-NPMLE due to the dependence between $\mu_i$ and $\sigma_i$ in these designs.

The two procedures differ more markedly in uncertainty quantification. The Bayesian credible intervals attain close to nominal coverage for every outcome, ranging from $94.5\%$ to $98.4\%$, while the plug-in intervals fall below nominal coverage for every outcome, ranging from $50.2\%$ to $88.6\%$ for CLOSE. The difference is largest for outcomes in which the latent variation is small relative to the sampling noise, such as \texttt{jail\_white\_male} ($50.2\%$) and \texttt{jail\_white\_pooled} ($61.6\%$): for these outcomes, the posterior for $\mu_i$ depends heavily on the estimated shrinkage structure, so estimation error in that structure accounts for a large share of the uncertainty about $\mu_i$, and intervals that condition on the estimated structure reflect only the remainder. The width ratios in the last column also show that the Bayesian intervals are typically only $10$--$20\%$ wider than the CLOSE intervals, and the two outcomes with the largest ratios ($1.54$ and $1.42$) are those with the lowest plug-in coverage.

As is common in the literature, Table~\ref{tab:calibrated} reports a single coverage rate for each design, obtained by averaging the coverage indicator across tracts and Monte Carlo replications. For each $\sigma_i$, we can also compute a tract-specific coverage rate, given by the fraction of replications in which the interval for tract $i$ contains its simulated $\mu_i$. This allows us to assess whether the aggregate coverage rates mask systematic differences across levels of estimation precision. Figure~\ref{fig:coverage-by-sigma} plots these tract-specific coverage rates as a function of $\sigma_i$ for six of the fifteen designs (two from each outcome family). The disaggregated results reveal substantial heterogeneity in the plug-in
intervals. Within a given design, coverage can vary considerably with
$\sigma$ and tends to deteriorate near the edges of the $\sigma$ range. This
pattern is consistent with greater uncertainty in the estimated conditional
moments in regions with fewer nearby tracts, which is not reflected in
intervals that condition on the estimated shrinkage structure. By comparison,
the Bayesian coverage is more stable across $\sigma$ and remains much closer
to the nominal level, including near the edges of the $\sigma$ range. Note that
the tract-level coverage exhibits a clustered appearance, which reflects the
calibrated data-generating process itself: all tracts within one of the $20$
calibration bins share the same latent distribution $\widehat{G}_b$, while
both procedures fit a single shape distribution across all tracts, so the
resulting misspecification is common to all tracts in a bin and changes
discretely at bin boundaries. This misspecification is largely absorbed by the
Bayesian credible intervals, whereas it is clearly visible in the coverage of
the plug-in intervals.

\begin{figure}
\centering
\includegraphics[width=\textwidth]{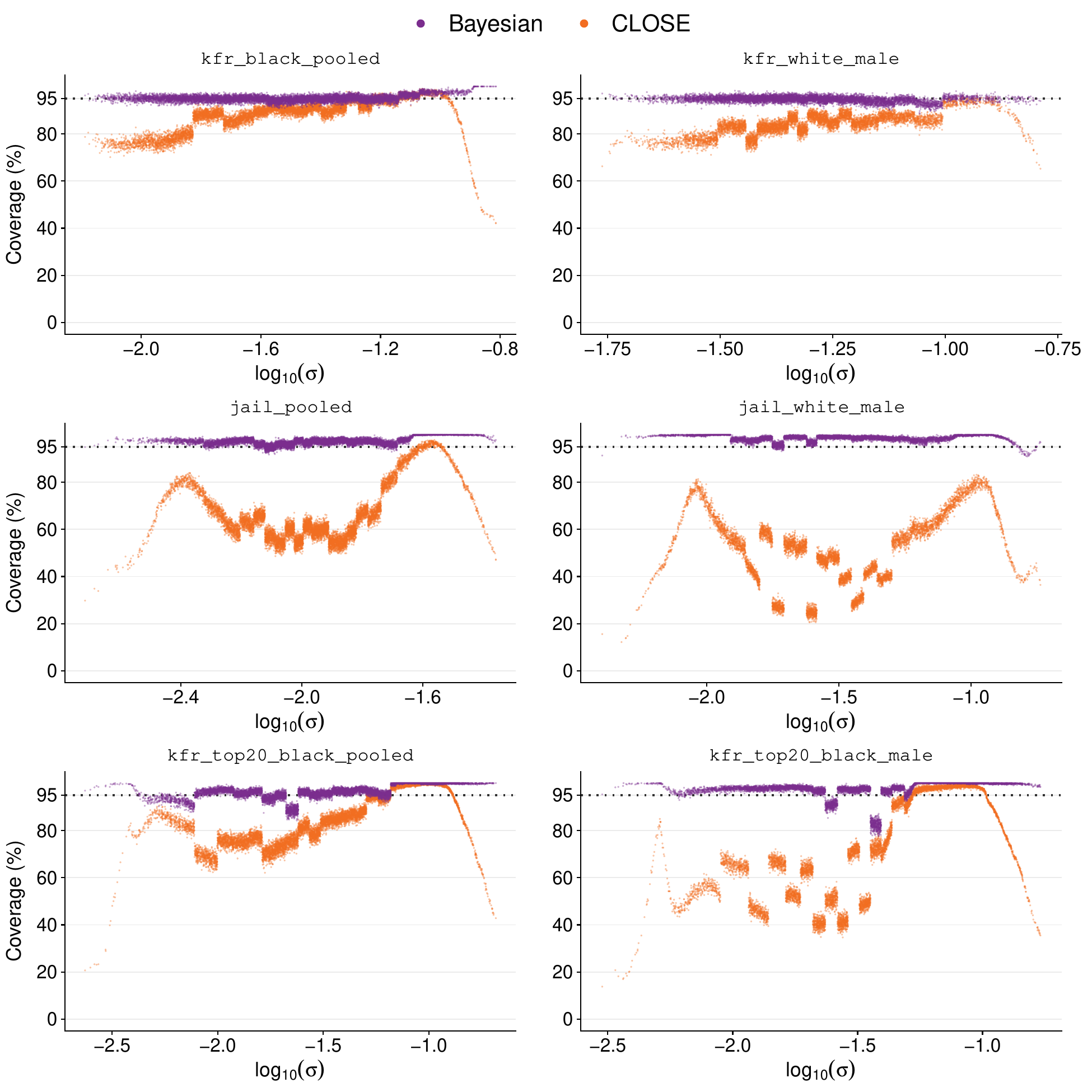}
\caption{Coverage of nominal 95\% intervals as a function of the noise level
$\sigma$ in six of the fifteen calibrated simulation designs (two from each
outcome family). Each point is one census tract, plotted at its own $\sigma_i$,
with coverage computed across the 500 simulation replications. The dotted
horizontal line marks the nominal 95\% level.}
\label{fig:coverage-by-sigma}
\end{figure}

Finally, we note that the calibrated simulations are built entirely from objects fitted by CLOSE. The latent scale is given by CLOSE's own truncated estimate $\widehat s(\sigma)$, while the latent distributions $\widehat G_b$ are discrete NPMLE solutions, in contrast to the continuous Gaussian mixtures supported by our DPGM prior. Moreover, $\widehat G_b$ varies across bins, so the common-$G_\star$ working model used by both procedures only approximates the calibrated data-generating process. Despite these differences, the single prior specification  performs well across all fifteen designs. The Bayesian point estimates are comparable to CLOSE, and slightly more accurate for most outcomes, while the credible intervals achieve coverage close to nominal throughout. The intervals are also only moderately wider than their plug-in counterparts. These results suggest that the Bayesian posterior is a competitive denoiser and can provide useful uncertainty assessments even when the working model is only approximate, at least in calibrated designs of the type considered here.

\section{Conclusion} \label{conclude}
This paper studies a nonparametric Bayes approach to a class of Gaussian compound decision problems typically addressed using empirical Bayes methods. We show that the resulting procedures can retain the strong frequentist denoising performance of empirical Bayes, while providing a joint posterior over the latent parameters and the unknown shrinkage structure. Beyond point estimation, this posterior can also be used to inform a variety of downstream decisions involving the latent parameters. This feature may be particularly useful in applications with only a small or moderate number of related problems, where the shrinkage structure is learned less precisely from the data.

There are several useful directions for extending this analysis. One is to allow richer forms of latent heterogeneity and covariate dependence. In model~\eqref{eq:hierarchy2NClose}, the latent distribution has a common shape $G_\star$, with variation across $\sigma_i$ entering only through its location and scale. The fully Bayesian framework lends itself naturally to more flexible conditional models. For example, the conditional mixture models of~\citet{norets2010approximation,norets2014posterior} could be used to allow the entire latent distribution, including its shape, to vary with the precision $\sigma_i$.

\section{Proof ideas for compound decision results}
\label{sec:proof_ideas}

In this section, we provide intuition for the main theoretical arguments for the theorems of Section~\ref{sec:homoscedastic_gaussian}.
Complete proofs are provided in the Appendix.

We start by providing the high-level intuition underlying Theorem~\ref{thm:compound-contraction}. The technical challenge lies in the fact that the model is misspecified, and so it is not a-priori clear that the posterior contracts around the average marginal density $f_{\boldmu}$ in~\eqref{eq:average_marginal}.
\begin{proof}[Proof sketch for Theorem~\ref{thm:compound-contraction}.]

Fix $C>0$ and write $\varepsilon_n \equiv \varepsilon_n(\Pi)$. We seek to show that the posterior assigns little probability to the set
$
\mathcal{U} := \cb{ G \,:\, \Dhel(f_{\boldmu}, f_G) \geq C\varepsilon_n }.
$
By the simplest Schwartz posterior contraction argument,
\begin{equation}
\begin{aligned}
\Pi( \mathcal{U} \mid \boldZ) \;= \; \frac{ \int_{\mathcal{U} } \prod_{i=1}^n f_G(Z_i) \dd\Pi(G)}{\int \prod_{i=1}^n f_G(Z_i)\dd\Pi(G)} \;=\; \frac{ \int_{\mathcal{U} } \prod_{i=1}^n \{f_G(Z_i)/f_{\boldmu}(Z_i)\}\dd\Pi(G)}{\int \prod_{i=1}^n \{f_G(Z_i)/f_{\boldmu}(Z_i)\}\dd\Pi(G)}.
\end{aligned}
\label{eq:schwartz}
\end{equation}
We have to show two facts: that with high probability under $\PP[\boldmu]{\cdot}$, both the numerator is ``small'' and the denominator is ``large.'' Note that
$\prod_{i=1}^n f_{\boldmu}(Z_i) \neq \prod_{i=1}^n \varphi(Z_i-\mu_i),$
where the latter is the likelihood under the data generating process in~\eqref{eq:hierarchy1N}, and so the ratio introduced in~\eqref{eq:schwartz} is not a likelihood ratio under $\PP[\boldmu]{\cdot}$. We can show that the posterior contracts despite this misspecification.\footnote{Similar arguments also appear in the analysis of nonparametric Bayesian procedures that utilize quasi-likelihoods \citep{kato2013quasi,kankanala2025generalized}.}
To see why, consider the following argument. Fix $f_G$ for any distribution $G$ in the support of $\Pi$. To understand both the numerator and denominator in~\eqref{eq:schwartz}, we must understand the behaviour of $\sum_{i=1}^n \log\cb{f_G(Z_i)/f_{\boldmu}(Z_i)}$. Taking expectations under $\PP[\boldmu]{\cdot}$, we get
$$
\begin{aligned}
\EE[\boldmu]{ \sum_{i=1}^n \log\p{\frac{f_G(Z_i)}{f_{\boldmu}(Z_i)}} } &= \sum_{i=1}^n \int \log\p{\frac{f_G(z)}{f_{\boldmu}(z)}} \varphi(z-\mu_i)\, \dd z \\
&= n \int \log\p{\frac{f_G(z)}{f_{\boldmu}(z)}} \underbrace{\frac{1}{n}\sum_{i=1}^n \varphi(z-\mu_i)}_{=f_{\boldmu}(z)}\, \dd z
\,=\, - n \DKL{f_{\boldmu}}{f_G},
\end{aligned}
$$
where $\DKL{\cdot}{\cdot}$ is the Kullback-Leibler divergence, defined for densities $f,h$ as
\begin{equation}
\DKL{f}{h} := \int f(z) \log\frac{f(z)}{h(z)} \dd z.
\label{eq:DKL}
\end{equation}
By comparison, if data generation were given by both~\eqref{eq:hierarchy1N} and~\eqref{eq:hierarchy2N} with true prior $G_{\star}$, then  we would have $\EE[G_{\star}]{ \sum_{i=1}^n \log\cb{f_G(Z_i)/f_{G_{\star}}(Z_i)}} = - n \DKL{f_{G_\star}}{f_{G}}$, i.e., effectively $f_{\boldmu}$ plays the role of $f_{G_\star}$.
Given suitable concentration uniformly over $f_G$, we can now surmise the following.
In the numerator of~\eqref{eq:schwartz}, fix $G \in \mathcal{U}$ (within the integral). Note that $-n\DKL{f_{\boldmu}}{f_G} \leq -2n\Dhel^2(f_{\boldmu}, f_G) \leq -2 C^2 n \varepsilon_n^2$ and so the numerator should be small. The proof in the appendix makes this argument precise, and also shows that the denominator is sufficiently large (using the fact that the prior $\Pi$ puts enough mass around $f_{\boldmu}$ in KL neighborhoods).
\end{proof}

We now turn to providing intuition for the proof of Theorem~\ref{theo:main_risk_bound_eq}.
The overall proof strategy has four main elements:
\begin{enumerate*}[noitemsep,leftmargin=*]
\item reduction from permutation equivariant to separable decision rules;
\item the fundamental theorem of compound decisions;
\item posterior contraction under~\eqref{eq:hierarchy1N} as shown in Theorem~\ref{thm:compound-contraction};
\item control of posterior means following~\citet{jiang2009general}.
\end{enumerate*}
We explain the overall strategy in more detail below.

To begin, we define the class of simple separable decision rules:
\begin{equation} \mathcal{T}^{\mathrm{S}} := \left\{ \boldt: \RR^n \to \RR^n \;\,\text{s.t.}\;\, \boldt(\boldZ) = (t(Z_1),\ldots,t(Z_n)) \,\text{ for }\, t: \RR \to \RR \right\}. \label{eq:Ssym} \end{equation}
As a first step, we note the following quantitative bound, due to \citet[Theorem~5.1]{greenshtein2009asymptotic}; see also \citet[Theorem~4.1]{han2025besting} and \citet{liang2025sharp}.
\begin{theo} There exists a constant $C_M >0$ that depends only on $M$ such that
$$\sup_{\boldmu \in [-M,M]^n} \p{ \inf_{\boldt \in \mathcal{T}^{\mathrm{S}} } \cb{R( \boldt(\boldZ),\boldmu) } - \inf_{\boldt \in \mathcal{T}^{\mathrm{PE}} } \cb{R(\boldt(\boldZ),\boldmu) }} \leq C_M \frac{1}{\sqrt{n}}.$$
\label{theo:equiv_and_simple}
\end{theo}
The upshot of Theorem~\ref{theo:equiv_and_simple} is that it suffices to compete with the best simple separable estimator, rather than with the best permutation equivariant estimator; moreover, the former admits an explicit characterization via the fundamental theorem of compound decision theory~\citep{robbins1951asymptotically, jiang2009general}. The key idea is that the compound model with fixed $\boldmu$ in~\eqref{eq:hierarchy1N} behaves, in certain respects, similarly to the Bayes model in~\eqref{eq:hierarchy1N}--\eqref{eq:hierarchy2N} with \smash{$\mu_i \simiid \compoundG$}, where $\compoundG$ is the empirical distribution of $\boldmu$ defined in~\eqref{eq:compound_prior}. We give a self-contained proof of the theorem, which also makes this connection transparent.

\begin{theo}[Fundamental theorem of compound decisions]
\label{thm:fundamental-compound}
Let $t:\RR\to\RR$ be measurable and consider the separable estimator with $\hat\mu_i=t(Z_i)$ for all $i$.
Then under~\eqref{eq:hierarchy1N},
\begin{equation}
\label{eq:fundamental-compound}
R^2( \widehat{\boldmu},\,\boldmu)
= \int (t(z) - \mu)^2\varphi(z-\mu)\, \dd z \, \dd \compoundG(\mu),
\end{equation}
where  the left-hand side risk is defined as in~\eqref{eq:RMSE}.
Consequently,
$t^\star_{\boldmu}(z)=\delta_{\compoundG}(z)$ (defined in~\eqref{eq:cd_b_post_mean}) yields the separable estimator that minimizes $R( \boldt(\boldZ),\,\boldmu)$ over all $\boldt \in \mathcal{T}^{\mathrm{S}}$.
\end{theo}

\begin{proof}
Let $\widehat{\boldmu}$ be a separable estimator with coordinate-wise function $t(\cdot)$. By definition, $R^2(\widehat{\boldmu},\,\boldmu) = \frac{1}{n}\sum_{i=1}^n \EE[\boldmu]{(t(Z_i)-\mu_i)^2} = \frac{1}{n}\sum_{i=1}^n \int (t(z)-\mu_i)^2\,\varphi(z-\mu_i)\,\dd z$, and since $\compoundG$ is the empirical distribution of $\boldmu$, the last expression equals the right-hand side of~\eqref{eq:fundamental-compound}.
Meanwhile, the same right-hand side is the risk when integrating over both~\eqref{eq:hierarchy1N} and~\eqref{eq:hierarchy2N} (with $G_{\star} = \compoundG$), and so is minimized over all univariate $t(\cdot)$ by $\delta_{\compoundG}(\cdot)$.
\end{proof}
With the explicit expression for the best separable denoiser, we can reduce the problem to controlling
$$
\frac{1}{n}\sum_{i=1}^n\EE[\boldmu]{ \cb{\widehat{\mu}_i(\Pi) - \delta_{\compoundG}(Z_i)}^2} = \frac{1}{n}\sum_{i=1}^n\EE[\boldmu]{ \cb{ \int \frac{f_G'(Z_i)}{f_G(Z_i)} \dd\Pi(G \mid \boldZ) -  \frac{f'_{\boldmu}(Z_i)}{f_{\boldmu}(Z_i)}}^2},
$$
where the equality uses~\eqref{eq:bnp-post-mean} alongside the Eddington/Tweedie formula~\citep{dyson1926method, efron2011tweedie} according to which $\delta_G(z) = z + f'_G(z)/f_G(z)$ (for $\delta_G$ defined in~\eqref{eq:cd_b_post_mean}). This connects the regret to the compound posterior contraction result of Theorem~\ref{thm:compound-contraction}; the formal argument then builds on techniques of~\citet{jiang2009general}.

\paragraph{Acknowledgements.} We thank Jiafeng Chen and Asaf Weinstein for helpful discussions. This work was completed in part with resources provided by the University of Chicago's Research Computing Center and Chicago Booth Computing Cluster. N.I. gratefully acknowledges support from the U.S. National Science Foundation (DMS-2443410).

\bibliographystyle{abbrvnat}
\bibliography{ebb}

\clearpage
\appendix
\setcounter{equation}{0}
\setcounter{figure}{0}
\setcounter{table}{0}
\setcounter{prop}{0}
\renewcommand{\theequation}{S\arabic{equation}}
\renewcommand{\thefigure}{S\arabic{figure}}
\renewcommand{\thetable}{S\arabic{table}}
\renewcommand{\theprop}{S\arabic{prop}}
\renewcommand{\theHequation}{appendix.\arabic{equation}}
\renewcommand{\theHfigure}{appendix.\arabic{figure}}
\renewcommand{\theHtable}{appendix.\arabic{table}}
\renewcommand{\theHprop}{appendix.\arabic{prop}}

\section{Implementation details}
\label{implement}

\subsection{CLOSE Implementation}
Let $Z_i$ denote the estimate for tract $i$, $\sigma_i$ its standard error,
and $x_i = \log_{10}\sigma_i$. All computations are carried out on
standardized scales: $\tilde Z_i = (Z_i - \bar Z)/s_Z$ and
$\tilde\sigma_i = \sigma_i / s_Z$, where $\bar Z$ and $s_Z$ are the sample
mean and standard deviation of the estimates, and
$\tilde x_i = (x_i - \bar x)/s_x$ analogously. The model that we use for computation is then:
\begin{align*}
  \tilde Z_i \mid \tilde\mu_i &\sim \mathrm{N}(\tilde\mu_i, \tilde\sigma_i^2), \qquad
  \tilde\mu_i = m(\tilde x_i) + s(\tilde x_i)\,\zeta_i, \qquad
  \zeta_i \mid G \overset{\text{iid}}{\sim} G,
\end{align*}
where $m(\cdot)$ and $s(\cdot)$ are unknown functions and $G$ is an unknown
distribution on $\mathbb{R}$. Results are transformed back to
the original scale for reporting.\\

\noindent\textbf{Conditional moment functions.}
We place Gaussian process priors on $m(\cdot)$ and on
$\eta(\cdot) = \operatorname{softplus}^{-1}(s^2(\cdot))$, so that
$s^2(x) = \log(1 + e^{\eta(x)})$ is positive without further constraints.
Each of the two functions, $f \in \{m, \eta\}$, has a Mat\'ern-3/2 covariance
function with its own amplitude and lengthscale parameters
$(\alpha_f, \rho_f)$, and all four parameters have independent
$\text{LogNormal}(0, 1)$ priors. For computation, each $f$ is represented
through its eigenfunction expansion \citep{riutort2023practical} on an
interval $[-L, L]$:
\begin{align*}
  f(x) = \sum_{j=1}^{M} \sqrt{S_f(\omega_j)}\,\beta_{f,j}\,\phi_j(x), \; \;
  \phi_j(x) = L^{-1/2}\sin\!\big(\omega_j (x + L)\big), \; \;
  \omega_j = \frac{j\pi}{2L},
\end{align*}
where $\beta_{f,j} \overset{\text{iid}}{\sim} \mathrm{N}(0,1)$ and
$S_f(\omega) = 4\alpha_f^2 \kappa_f^3 / (\kappa_f^2 + \omega^2)^2$, with
$\kappa_f = \sqrt{3}/\rho_f$, is the spectral density of the Mat\'ern-3/2
kernel. The basis functions $\phi_j$ are common to $m$ and $\eta$; the two
functions differ only through their parameters and coefficients.

Two choices remain. First, $L = c\,S$, where
$S = (\max_i \tilde x_i - \min_i \tilde x_i)/2$ is the half-range of the
standardized input and $c = 1.5$ pads the domain beyond the observed inputs,
as recommended by \citet{riutort2023practical}. This also leaves room for the
unknown population support $[\underline\sigma, \overline\sigma]$ of the noise
levels, which contains the observed range of $\sigma_i$. In the application
and in every simulation design, all observed inputs lie within $78\%$ of the
distance from the center of $[-L, L]$ to its boundary. Second, we use
$M = 30$ basis functions for each of $m$ and $\eta$. In the application the
posterior lengthscales are comparable to the half-range $S$
($\rho_m \approx 3.4$, $\rho_s \approx 2.7$, $S = 2.7$), so high-frequency
basis functions receive little weight: for each of $m$ and $\eta$, basis
functions $26$--$30$ contribute less than $1\%$ of the total posterior
expected absolute coefficient weight $|\sqrt{S_f(\omega_j)}\,\beta_{f,j}|$,
and increasing $M$ to $50$ left all results unchanged.\\

\noindent\textbf{Shape distribution.}
The distribution $G$ is a  Dirichlet process mixture of Gaussians
with a common component variance $A$ (the Gaussian smoothing variance of
Specification~\ref{assu:DPGM}):
\begin{align*}
  G = \sum_{k=1}^{K} w_k\, \mathrm{N}(\mu_k, A), \qquad
  w_k = v_k \prod_{l < k}(1 - v_l), \quad v_K = 1,
\end{align*}
with stick-breaking fractions $v_k \sim \text{Beta}(1, \alpha)$ for $k < K$
and component means $\mu_k \overset{\text{iid}}{\sim} H$. The hyperpriors,
applied on the standardized scale, are $A \sim \text{Beta}(2, 5)$,
$H = \mathrm{N}(0, 3^2)$, and $\alpha \sim \text{Gamma}(2, 0.5)$ in the shape--rate
parameterization, following \citet{hirano2002semiparametric}. Because the location and scale of $G$ are not separately identified from
$m(\cdot)$ and $s(\cdot)$ (they are normalized out in $G_\star$ below), the
base measure $H$ should be interpreted through how diffuse it is relative to a shape
distribution of unit variance. The choice $H = \mathrm{N}(0, 3^2)$ is diffuse in this
sense, and replacing it by a $\text{Uniform}(-10, 10)$ base measure leaves all
identified quantities in the application essentially unchanged. We use the truncation $K = 30$. In the application the truncation is not binding: the posterior expected
mass on components beyond the $20$th is $0.6\%$, and beyond the $27$th it is
$0.2\%$ (increasing $K$ to $50$ left all results unchanged). Given
$(m, s, G)$, the latent $\zeta_i$ are integrated out analytically, so the
likelihood contribution of tract $i$ is the mixture density
$\sum_k w_k\, \mathrm{N}\!\big(\tilde Z_i;\, m(\tilde x_i) + s(\tilde x_i)\mu_k,\;
s^2(\tilde x_i)A + \tilde\sigma_i^2\big)$.

Because $(m, s, G)$ are identified only up to the location and scale of $G$,
we report the normalized quantities defined in Remark~\ref{normalize}: at each
posterior draw, with $\bar\mu = \sum_k w_k \mu_k$ and
$\sigma_G^2 = \sum_k w_k \mu_k^2 - \bar\mu^2 + A$, the standardized shape
distribution $G_\star$ has components $(\mu_k - \bar\mu)/\sigma_G$ and
variance $A/\sigma_G^2$, and the identified conditional moments are
$m_\star = m + s\bar\mu$ and $s_\star = s\,\sigma_G$. These are then mapped
back to the original scale via $\bar Z + s_Z\, m_\star$ and $s_Z\, s_\star$.
The posterior of any functional of the observables, such as $\mu_i$, is
unaffected by this normalization.\\

\subsection{Posterior computation}
\label{app:computation}

The model is fit with Hamiltonian Monte Carlo (the no-U-turn sampler) in
Stan~\citep{JSSv076i01}, with the likelihood parallelized across tracts
using within-chain threading. In the empirical application we run two chains
of $2{,}000$ warmup and $2{,}000$ sampling iterations each, for $4{,}000$
posterior draws, with target acceptance rate $0.9$ and maximum tree depth
$12$; all split-$\widehat R$ statistics for the hyperparameters are below
$1.002$ (and there are no divergent transitions). In the calibrated
simulations, the Bayesian procedure is re-estimated from scratch on every
replication with exactly the same model and priors; the standardization
constants $(\bar Z, s_Z, \bar x, s_x)$ and the boundary $L$ are recomputed
from each simulated data set. Each simulation fit uses a single chain of
$2{,}000$ warmup and $1{,}000$ sampling iterations, with the same target
acceptance rate and tree depth as in the application.

Posterior summaries for $\mu_i$ are formed identically in both settings and
exploit conjugacy given a draw of $(m, s, G)$. Conditional on the draw, the
posterior of $\mu_i$ is a $K$-component mixture of normals with weights
proportional to
$w_k\,\mathrm{N}(\tilde Z_i;\, m_i + s_i\mu_k,\; s_i^2A+ \tilde\sigma_i^2)$,
component means $(s_i^2A\,\tilde Z_i + \tilde\sigma_i^2 (m_i +
s_i\mu_k))/(s_i^2A + \tilde\sigma_i^2)$, and common variance
$s_i^2A \tilde\sigma_i^2/(s_i^2A + \tilde\sigma_i^2)$, where
$m_i = m(\tilde x_i)$ and $s_i = s(\tilde x_i)$. Posterior means are computed
exactly at each draw and averaged across draws. Credible intervals are
equal-tailed and marginalize over the hyperparameter draws: at each draw we
sample $\mu_i$ from its conditional posterior (a component from the mixture
weights, then a normal draw), and the $95\%$ interval is formed from the
$2.5\%$ and $97.5\%$ quantiles of the pooled samples. Pairwise probabilities
such as $\PP{\mu_B > \mu_A \mid \boldZ}$ are computed exactly at each draw from the
two conditional mixtures, which are independent given the draw, and averaged
across draws. All quantities are mapped back to the original scale via
$\bar Z + s_Z(\cdot)$ before reporting.

\section{Definitions of function spaces}
\label{sec:funct_space}

\begin{defi}[Function spaces] Consider a smooth domain $\mathcal{X} \subseteq \RR$.  For any $\alpha > 0$, we have the following function spaces:
\begin{enumerate}
\item The Hölder space $\Sigma^{\alpha}(\mathcal{X})$ is the class of functions $f: \mathcal{X} \rightarrow \RR$ with $\| f \|_{\Sigma^{\alpha}} < \infty$, where  $$ \| f \|_{\Sigma^{\alpha}(\mathcal{X})} = \max_{k:k \leq \underline\alpha} \sup_{x \in \mathcal{X}} \left| D^k f(x)   \right| +  \sup_{x,y \in \mathcal{X} : x \neq y} \frac{\left|  D^{\underline\alpha} f(x)  - D^{\underline\alpha} f(y) \right|}{\left|  x-y \right|^{\alpha - \underline\alpha}}. $$
Here, $\underline\alpha$ denotes the biggest integer strictly smaller than $\alpha$.
\item The Sobolev space $H^{\alpha}(\RR)$ is the class of functions
$f:\mathcal \RR \rightarrow \RR$ with $\|f\|_{H^\alpha(\RR)}<\infty$, where
\[
\|f\|_{H^\alpha(\RR)}
:=
\|f\|_{L^2(\RR)}+v_\alpha(f) , \;
v_\alpha^2(f)
:=
\int_{\RR}|\lambda|^{2\alpha}|\hat f(\lambda)|^2 \dd \lambda . \]
Here, $\hat f$ denotes the Fourier transform of $f$. The Sobolev space over a smooth domain $\mathcal{X} \subset \RR$ is defined as the class of functions $f : \mathcal{X} \rightarrow \RR$ with quotient norm $ \|  f \|_{H^{\alpha}(\mathcal{X})} = \inf \{ \|g \|_{H^{\alpha}(\RR)} : g \in H^{\alpha}(\RR),\;  g|_{\mathcal X}=f$\}.

\end{enumerate}
\end{defi}

\section{Proofs for (in)admissibility}

\subsection{Proof of Proposition~\ref{prop:bnp-admissible-cd}}

\begin{proof}
Define the induced prior $G_\Pi$ on $\RR^n$ by
\[
G_\Pi(A_1,\dots,A_n) := \int \prod_{i=1}^n G(A_i)\, \dd\Pi(G), \; A_1,\dots,A_n \subset \RR.
\]
Below we will use the subscript $G_{\Pi}$ to denote expectations taken with respect to $\boldmu\sim G_\Pi$ and $\boldZ\mid \boldmu \sim \mathrm{N}(\boldmu,I_n)$. The only properties of $\Pi$ we will use below is that $\mathbb E_{G_{\Pi}}[\Norm{\boldmu}^2] < \infty$ (and that $G_{\Pi}$ is a proper prior).

Now suppose that $\widehat{\boldmu}(\Pi)$ is not admissible. Then there exists another estimator $\tilde{\boldmu}=\tilde{\boldmu}(\boldZ)$ such that
$R(\tilde{\boldmu}, \boldmu) \leq R(\widehat{\boldmu}(\Pi), \boldmu)$ for all $\boldmu \in \RR^n$.
By integrating over $G_{\Pi}$ (which is supported on $\RR^n$, this implies that
$$
\EE[G_{\Pi}]{ \Norm{ \tilde{\boldmu} -  \widehat{\boldmu}(\Pi)}^2} = \EE[G_{\Pi}]{ \Norm{ \tilde{\boldmu} -  \boldmu  }^2}  - \EE[G_{\Pi}]{ \Norm{ \widehat{\boldmu}(\Pi)- \boldmu  }^2} \leq 0.
$$
Thus $\widetilde{\boldmu}(\boldz) =\widehat{\boldmu}(\Pi)(\boldz)$ for almost all $\boldz \in \RR^n$ (with respect to Lebesgue measure). But this implies that
$R(\tilde{\boldmu}, \boldmu) = R(\widehat{\boldmu}(\Pi), \boldmu)$ for all $\boldmu \in \RR^n$,
and so $\tilde{\boldmu}$ does not dominate $\widehat{\boldmu}(\Pi)$.

\end{proof}

\subsection{Proof for inadmissibility of NPMLE ( Proposition~\ref{prop:npmle-inadmissible-cd})}
\label{subsec:proof_npmle_inadmissible}

We start with the high-level idea of the proof before furnishing details.

\begin{proof}[Proof sketch for Proposition~\ref{prop:npmle-inadmissible-cd}.]
The key idea is to observe the following. Fix $n \geq 2$. For $\boldz\in \RR^n$, let $\bar{z} := n^{-1}\sum_{i=1}^n z_i$ and define the set
\begin{equation}
\mathcal U := \left\{\boldz\in\RR^n:\ \max_{1\le i\le n}|z_i-\bar z| < 1 \right\}.
\label{eq:u_set}
\end{equation}
By the first order optimality condition of the NPMLE in~\eqref{eq:npmle_and_marginal} over $\mathcal{P}(\RR)$, we have that $\widehat{G}(\boldz) = \delta_{\bar z}$ for
all $\boldz \in \mathcal{U}$. Thus,
$$
\widehat{\boldmu}^{\mathrm{EB}}(\boldz) = \bar{z} \cdot \mathbf{1}_n = (\bar{z},\ldots,\bar{z}) \quad \text{for all } \boldz \in \mathcal{U}.
$$
Now suppose that \smash{$\widehat{\boldmu}^{\mathrm{EB}}$} is admissible. Then by Theorem 3.1.1. of~\citet{brown1971admissible}, \smash{$\widehat{\boldmu}^{\mathrm{EB}}$} is a generalized Bayes estimator and so it is analytic as a function of $\boldz$. Hence, \smash{$\widehat{\boldmu}^{\mathrm{EB}}(\boldz) = \bar{z} \cdot \mathbf{1}_n$} for almost all $\boldz \in \RR^n$ by the identity theorem. However, this can be shown to be a contradiction. Thus, $\widehat{\boldmu}^{\mathrm{EB}}$ is inadmissible.
\end{proof}
 \begin{rema}  The proof technique above has been used before to show inadmissibility of  estimators in the multivariate normal mean estimation problem. Recall the James-Stein~\citeyearpar{james1961estimation} estimator and its positive-part,
$$
\widehat{\boldmu}^{\mathrm{JS}} := \left(1 - \frac{(n-2)}{\Norm{\boldZ}^2_2}\right) \boldZ,\;\;\;\;\;\; \widehat{\boldmu}^{\mathrm{JS+}} := \left(1 - \frac{(n-2)}{\Norm{\boldZ}^2_2}\right)_+ \boldZ,
$$
where $(x)_+ = \max\{x,0\}$. For $n\geq3$, the James-Stein estimator dominates the MLE $\boldZ$, and thus the latter is inadmissible. However, \smash{$\widehat{\boldmu}^{\mathrm{JS}}$} is itself inadmissible since it is dominated by \smash{$\widehat{\boldmu}^{\mathrm{JS+}}$}~\citep{efron1973stein}, and \smash{$\widehat{\boldmu}^{\mathrm{JS+}}$} is inadmissible as well. To show inadmissibility of the latter, \citet[Example 2.9 and Theorem 4.16]{brown1986fundamentals} observes that \smash{$\widehat{\boldmu}^{\mathrm{JS+}}(\boldz) =0$} for \smash{$\Norm{\boldz}^2_2 \leq n-2$}. By the same continuation argument as above, if \smash{$\widehat{\boldmu}^{\mathrm{JS+}}$} were admissible, then \smash{$\widehat{\boldmu}^{\mathrm{JS+}}(\boldz) = 0$} for almost all $\boldz \in \RR^n$, which is a contradiction. Thus, \smash{$\widehat{\boldmu}^{\mathrm{JS+}}$} is inadmissible.
\end{rema}

We are ready to furnish the full proof of the proposition.

\begin{proof}

We only have to fill in two arguments that were not justified in the proof sketch
: (i) to verify the explicit solution of the NPMLE over the open set~\eqref{eq:u_set}, and (ii) to show that there exists another open set on which the NPMLE solution takes a different form.\\

\noindent\textbf{(i).}  Throughout let $\ell(G;\boldz):=\sum_{i=1}^n \log f_G(z_i)$.
Since $G\mapsto f_G(z_i)$ is linear and $\log$ is concave, $G\mapsto \ell(G;\boldz)$ is concave over the
convex set of probability measures on $\RR$.

Let $\bar z := n^{-1}\sum_{i=1}^n z_i$ and recall the open set defined in~\eqref{eq:u_set}:
$
\mathcal U := \{\boldz\in\RR^n:\ \max_{1\le i\le n}|z_i-\bar z|<1\}.
$
Fix $\boldz\in\mathcal{U}$ and consider the point mass $G_0:=\delta_{\bar z}$.
According to a standard first-order optimality condition for the NPMLE~\citep[Theorem 4.1.]{lindsay1983geometrya}, $G_0$ maximizes
$\ell(\cdot;\boldz)$ if and only if for every $\mu\in\RR$,
\begin{equation}\label{eq:KKT}
\sum_{i=1}^n \frac{\varphi(z_i-\mu)}{f_{G_0}(z_i)} \le n.
\end{equation}
(To see this, consider the directional derivative along $(1-\lambda)G_0+\lambda\delta_\mu$ at $\lambda=0^+$.)
Now $f_{G_0}(z_i)=\varphi(z_i-\bar z)$. Let $s_i:=z_i-\bar z$, so that $\sum_{i=1}^n s_i=0$ and
$|s_i|<1$ for all $i$. Take any $\mu \in \RR$ and write $\mu=\bar z+u$, then,
$$
\frac{\varphi(z_i-\mu)}{\varphi(z_i-\bar z)}
=\exp\p{-\frac{(s_i-u)^2-s_i^2}{2}}
=\exp\p{-\frac{u^2}{2}+u\,s_i}
$$
Thus \eqref{eq:KKT} is equivalent to
$
\exp(-u^2/2)\sum_{i=1}^n \exp(u s_i)\le n
\;\Longleftrightarrow\;
n^{-1}\sum_{i=1}^n \exp(u s_i)\le \exp(u^2/2)
\;\text{for all }u\in\RR.
$
But $(s_1,\ldots,s_n)$ has empirical mean $0$ and is supported in $[-1,1]$, so Hoeffding's lemma
implies
$
n^{-1}\sum_{i=1}^n \exp(u s_i) \le \exp(4u^2/8) = \exp\p{u^2/2}
\;\text{ for all }u\in\RR,
$
verifying \eqref{eq:KKT}. Hence $G_0=\delta_{\bar z}$ is an NPMLE for every $\boldz\in\mathcal U$ with uniqueness due to~\citet{lindsay1993uniqueness} and
consequently,
\refstepcounter{equation}\label{eq:onU}$\hat{\boldmu}^{\mathrm{EB}}(\boldz)=\bar z\,\mathbf 1_n \;\text{for all }\boldz\in\mathcal U.$\\

\noindent\textbf{(ii).} We now exhibit an open set on which the NPMLE is not a point mass.
Fix $\varepsilon\in(0,1)$ and let $a>0$ be large. Consider the three-point mixing distribution
$G_\varepsilon := (1-\varepsilon)\delta_0+(\varepsilon/2)\delta_{-a}+(\varepsilon/2)\delta_a$,
and the data vector $\boldz^\star:=(-a,a,0,\ldots,0)\in\RR^n$, for which $\bar z^\star=0$.

Let $L(G;\boldz):=\prod_{i=1}^n f_G(z_i)$ denote the marginal likelihood. We bound:
$f_{G_\varepsilon}(0)=(1-\varepsilon)\varphi(0)+\varepsilon\varphi(a)\ge (1-\varepsilon)\varphi(0),\; f_{G_\varepsilon}(\pm a)=(1-\varepsilon)\varphi(a)+\frac{\varepsilon}{2}\varphi(0)+\frac{\varepsilon}{2}\varphi(2a) \ge \frac{\varepsilon}{2}\varphi(0).$
Therefore
$$
L(G_\varepsilon;\boldz^\star)\ge \Bigl(\frac{\varepsilon}{2}\varphi(0)\Bigr)^2\Bigl((1-\varepsilon)\varphi(0)\Bigr)^{n-2}
= \varphi(0)^n\Bigl(\frac{\varepsilon}{2}\Bigr)^2(1-\varepsilon)^{n-2}.
$$
On the other hand, among point masses $G=\delta_\mu$, the likelihood is maximized at $\mu=\bar z^\star=0$,
so
$$
\sup_{\mu\in\RR}L(\delta_\mu;\boldz^\star)=L(\delta_0;\boldz^\star)
=\varphi(a)^2\varphi(0)^{n-2}=\varphi(0)^n e^{-a^2}.
$$
Hence
$
L(G_\varepsilon;\boldz^\star)/\sup_{\mu}L(\delta_\mu;\boldz^\star)
\ge (\varepsilon/2)^2(1-\varepsilon)^{n-2}e^{a^2},
$
which is $>1$ for $a$ sufficiently large. Thus for such a choice of $a$,
$L(G_\varepsilon;\boldz^\star)>\sup_{\mu}L(\delta_\mu;\boldz^\star)$,
so no point mass can be an NPMLE at $\boldz^\star$.

Define the continuous function
$H(\boldz):=\log L(G_\varepsilon;\boldz)-\log\p{\sup_{\mu}L(\delta_\mu;\boldz)}$.
Since $H(\boldz^\star)>0$ and $H$ is continuous, there exists an open neighborhood
$\mathcal V$ of $\boldz^\star$ such that $H(\boldz)>0$ for all $\boldz\in\mathcal V$.
In particular, for every $\boldz\in\mathcal V$, the NPMLE \smash{$\widehat G(\boldz)$} cannot be
a point mass.

For any (nondegenerate) mixing distribution $G$, the map $z\mapsto\delta_G(z)$ is strictly increasing:
writing $m_k(z):=\int \mu^k\varphi(z-\mu)\,\dd G(\mu)$ so that $\delta_G=m_1/m_0$, one computes
$\delta_G'(z)=\cb{m_2(z)m_0(z)-m_1(z)^2}/{m_0(z)^2}=\Var[G]{\mu\mid Z=z}>0.$
Therefore, for every $\boldz\in\mathcal V$ (where $\widehat G(\boldz)$ is nondegenerate)
the coordinates $\hat\mu_i^{\mathrm{EB}}(\boldz)=\delta_{\widehat G(\boldz)}(z_i)$
cannot all be equal unless the $z_i$ are all equal, which does not occur on the neighborhood
$\mathcal V$. In particular,
\refstepcounter{equation}\label{eq:onV}$\hat{\boldmu}^{\mathrm{NPMLE}}(\boldz)\neq \bar z\,\mathbf 1_n \;\text{for all }\boldz\in\mathcal V.$
\end{proof}

\section{Contraction and regret proofs for homoscedastic Gaussian decision problem}

\subsection{Notation}

We recall the relevant notions of metric entropy.
An $\varepsilon$-bracket $[l,u]$ with respect to $\Dhel$ is a pair of functions $l\le u$ such that $\Dhel(l,u)\le\varepsilon$.
A function $f$ belongs to the bracket if $l\le f\le u$ pointwise. Given a class  $\mathcal{F}$, the $\varepsilon$-bracketing number $N_{[]}(\varepsilon,\mathcal{F},\Dhel)$ is the minimum number of $\varepsilon$-brackets needed to cover $\mathcal{F}$, and the bracketing entropy is its logarithm.
Note that in~\eqref{eq:hellinger}, the Hellinger distance was defined for probability densities.
In the bracketing context the bounding functions $l$ and $u$ need not integrate to one, however, we note that~\eqref{eq:hellinger} also makes sense for any non-negative, square-root-integrable functions and coincides with the usual Hellinger distance when both arguments are densities.

In what follows, we will use the following notation for Gaussian location mixtures,
\begin{equation}
f_{G,v}(\cdot):=\int \varphi(\,\cdot-\mu; v)\,\dd G(\mu),\;\;\; f_{G}(\cdot):=\int \varphi(\,\cdot-\mu)\,\dd G(\mu),
\label{eq:gaussian_loc_mixtures}
\end{equation}
where  $\varphi(\cdot;v)$ is the density of the $\mathrm{N}(0,v)$ distribution. We write $\varphi(\cdot)$ for $\varphi(\cdot;1)$ and $f_G$ for $f_{G,1}$ when the variance is $1$. We also introduce the following notation for the class of Gaussian location mixtures with variance $v \in [\ubar{v},\bar{v}]$ and mixing distributions supported on $[-M,M]$:
\begin{equation}
\mathcal F_{M,\ubar{v},\bar v}:=\cb{
f_{G,v}(\,\cdot\,)\;:
\; G\in\mathcal{P}([-M,M]),\; v \in [\ubar{v}, \bar{v}]
},
\label{eq:gaussian_mixture_class}
\end{equation}
For two densities $f$ and $h$ on $\RR$, we also define:
\begin{equation}
\VKL{f}{h} := \int f(z) \p{\log\frac{f(z)}{h(z)} - \DKL{f}{h}}^2 \dd z.
\label{eq:vkl}
\end{equation}
The KL neighborhood around $f_{\boldmu}$ is given by
\begin{equation}
B(f_{\boldmu}, \varepsilon) := \cb{ G\,:\,\DKL{f_{\boldmu}}{f_G} < \varepsilon^2,\;\; \VKL{f_{\boldmu}}{f_G} < \varepsilon^2 }.
\label{eq:dkl_vkl_ball}
\end{equation}
Finally, we define for any nonnegative function $h$,
\begin{equation}
L_n(h):=\prod_{i=1}^n \frac{h(Z_i)}{f_{\boldmu}(Z_i)},
\label{eq:Ln}
\end{equation}
keeping the data-generating vector $\boldmu$ implicit in the notation.

\subsection{Preliminary lemmata for posterior contraction proofs}

\begin{lemm}\label{lem:denom-bernstein}
Assume $\max_{1\le i\le n}|\mu_i|\le M$ and that Specification~\ref{assu:DPGM} holds for $\Pi$ (with $R \geq M$ if $H$ is Base-Compact).
Let
$\varepsilon_n \equiv \varepsilon_n(\Pi)$ be as in~\eqref{eq:epsilon_n}. Let
 $\mathcal K_n = \mathcal{K}(\varepsilon_n)$ be the local set of Lemma~\ref{lem:G-small-k} (with $\varepsilon=\varepsilon_n$, $k=2$, $m_{\star}\equiv0$, $s_{\star} \equiv 1$). Then there exists a constant $D>0$, depending only on the parameters
in Specification~\ref{assu:DPGM} and on $M$, such that for all sufficiently large $n$ the following holds with probability
at least $1-3/n^2$:
\begin{equation}\label{eq:denom-lb-bern}
\int \prod_{i=1}^n \frac{f_G(Z_i)}{f_{\boldmu}(Z_i)}\,\dd\Pi(G)
\;\geq\;
\exp\!\big\{- (1+D)\,n\varepsilon_n^2 \big\}\cdot \Pi(\mathcal K_n).
\end{equation}
\end{lemm}

\begin{proof}
By definition, $n\varepsilon_n^2\gtrsim \log^2 n$.
Let $\Pi_0$ denote $\Pi$ restricted to $\mathcal K_n$ and renormalized to a probability measure, and set
$$
\psi(z):=\int \log\p{\frac{f_G(z)}{f_{\boldmu}(z)}}\,\dd\Pi_0(G),
\quad
U:=\sum_{i=1}^n \psi(Z_i).
$$
We first record three facts about $\mathcal K_n$. First,
$\mathcal K_n\subseteq B(f_{\boldmu},\varepsilon_n)$ (where the latter is defined in~\eqref{eq:dkl_vkl_ball}). This is immediate by noting that $\mathcal{G}^{(2)}(\varepsilon) \equiv B(f_{\boldmu},\varepsilon_n)$.
Second, there exist a fixed radius $M_0>0$ and a fixed
constant $c_0\in(0,1]$, neither depending on $n$, such that every $(Q,A)\in\mathcal K_n$ satisfies
$Q([-M_0,M_0])\ge c_0$. Finally, $A\le\bar A$ holds $\Pi$-almost surely by Specification~\ref{assu:DPGM}.

Restricting the integral to $\mathcal K_n$ and applying Jensen's inequality,
$$
\begin{aligned}
\log\p{\int \prod_{i=1}^n \frac{f_G(Z_i)}{f_{\boldmu}(Z_i)}\,\dd\Pi(G)}
&\geq
\log\p{\int \prod_{i=1}^n \frac{f_G(Z_i)}{f_{\boldmu}(Z_i)}\,\dd\Pi_0(G)} + \log\Pi(\mathcal K_n)
&\geq U + \log\Pi(\mathcal K_n).
\end{aligned}
$$
Hence \eqref{eq:denom-lb-bern} follows once we show that, for a suitable $D>0$,
\begin{equation}\label{eq:denom-U-goal}
\PP[\boldmu]{ U \leq -(1+D)\,n\varepsilon_n^2 } \leq \frac{3}{n^2}
\end{equation}
for all sufficiently large $n$.

\medskip
\noindent\textbf{Mean and variance of $U$.}
Recall the argument presented in the main text:
$$
\EE[\boldmu]{ \sum_{i=1}^n \log\p{\frac{f_G(Z_i)}{f_{\boldmu}(Z_i)}} }
= n\int \log\p{\frac{f_G(z)}{f_{\boldmu}(z)}}\,f_{\boldmu}(z)\,\dd z
= -n\,\DKL{f_{\boldmu}}{f_G}.
$$
Since $\Pi_0$ is supported on $\mathcal K_n\subseteq B(f_{\boldmu},\varepsilon_n)$, Fubini's theorem gives
\[
\EE[\boldmu]{U} = -n\int \DKL{f_{\boldmu}}{f_G}\,\dd\Pi_0(G) \geq -n\varepsilon_n^2.
\]
Moreover,
$$
\begin{aligned}
\Var[\boldmu]{U} & =\sum_{i=1}^n \Var[\mu_i]{ \int \log\p{\frac{f_{\boldmu}(Z_i)}{f_G(Z_i)}} \dd\Pi_0(G) }
 \\ &\leq \sum_{i=1}^n \EE[\mu_i]{ \cb{ \int \p{\log\p{\frac{f_{\boldmu}(Z_i)}{f_G(Z_i)}} - \DKL{f_{\boldmu}}{f_G}} \dd\Pi_0(G) }^2 } \\
& \leq \int \sum_{i=1}^n \EE[\mu_i]{ \p{\log\p{\frac{f_{\boldmu}(Z_i)}{f_G(Z_i)}} - \DKL{f_{\boldmu}}{f_G}}^2 } \dd\Pi_0(G) \\
& = \int \sum_{i=1}^n \int  \p{\log\p{\frac{f_{\boldmu}(z)}{f_G(z)}} - \DKL{f_{\boldmu}}{f_G}}^2 \varphi(z-\mu_i) \dd z \,\dd\Pi_0(G)\\
&= n \int \int  \p{\log\p{\frac{f_{\boldmu}(z)}{f_G(z)}} - \DKL{f_{\boldmu}}{f_G}}^2 f_{\boldmu}(z) \dd z \,\dd\Pi_0(G)
 \\ &= n \int \VKL{f_{\boldmu}}{f_G} \dd\Pi_0(G) \leq n \varepsilon_n^2,
\end{aligned}
$$
where the second inequality follows from Jensen's inequality and Fubini.

\medskip
\noindent\textbf{Envelope.}
Fix $G=Q*\mathrm N(0,A)\in\mathcal K_n$, where we recall that $A \in [0,\bar{A}]$. Then, $f_G(z)\leq 1/{\sqrt{2\pi(1+A)}}\leq 1/{\sqrt{2\pi}}$ and
$$
f_G(z)\geq Q([-M_0,M_0])\inf_{|x|\le M_0}\varphi(z-x;1+A)
\geq \frac{c_0}{\sqrt{2\pi(1+\bar A)}}\,\exp\cb{-(|z|+M_0)^2/2}.
$$
The variance-one location mixture $f_{\boldmu}$ obeys
\smash{$
\exp\{-(|z|+M)^2/2\}/\sqrt{2\pi} \leq f_{\boldmu}(z)\leq 1/\sqrt{2\pi}.
$}
Combining these bounds, there is a constant $C_0>0$, depending only on $(M,M_0,c_0,\bar A)$, such that
$\sup_{G\in\mathcal K_n}\abs{\log\cb{f_G(z)/f_{\boldmu}(z)}}\leq C_0(1+z^2),\; z\in\RR.$
Averaging over $\Pi_0$ yields the envelope
\begin{equation}\label{eq:g-envelope}
|\psi(z)|\leq C_0(1+z^2),\quad z\in\RR.
\end{equation}

\medskip
\noindent\textbf{Truncation and Bernstein.}
Let $T_n:=M+\sqrt{8\log n}$ and $A_n:=\cb{\max_{1\le i\le n}|Z_i|\le T_n}$. Since $Z_i\sim\mathrm N(\mu_i,1)$ with
$|\mu_i|\le M$,
$$
\PP[\boldmu]{A_n^c}\leq \sum_{i=1}^n \PP[\mu_i]{|Z_i|>T_n}\leq 2n\,e^{-(T_n-M)^2/2}=2n^{-3}.
$$
By \eqref{eq:g-envelope}, on $A_n$ we have $|\psi(z)|\le C_0(1+T_n^2)\le C_1\log n$ for $|z|\le T_n$ and all large $n$.
Define the truncated, centered summands
$$
Y_{i,n}:=\big(\psi(Z_i)-\EE[\mu_i]{\psi(Z_i)}\big)\,\ind\{|Z_i|\le T_n\},\quad i=1,\ldots,n.
$$
On $\{|Z_i|\le T_n\}$ we have $|\psi(Z_i)|\le C_1\log n$, while \eqref{eq:g-envelope} gives
$|\EE[\mu_i]{\psi(Z_i)}|\le C_0\,\EE[\mu_i]{1+Z_i^2}=C_0(2+\mu_i^2)\le C_0(2+M^2)$. Hence for $n$ large,
$|Y_{i,n}|\leq C_1\log n + C_0(2+M^2)\leq 2C_1\log n =: B_n.$
Moreover $\Var[\mu_i]{Y_{i,n}}\le \Var[\mu_i]{\psi(Z_i)}$, so
$\sum_{i=1}^n \Var[\mu_i]{Y_{i,n}}\le \Var[\boldmu]{U}\le n\varepsilon_n^2$. Next,
$\abs{\EE[\mu_i]{Y_{i,n}}} \leq \EE[\mu_i]{|\psi(Z_i)|\,\ind\{|Z_i|>T_n\}} + \EE[\mu_i]{|\psi(Z_i)|}\,\PP[\mu_i]{|Z_i|>T_n}.$
Using \eqref{eq:g-envelope}, the Cauchy--Schwarz inequality, and the Gaussian tail bound with
$T_n-M=\sqrt{8\log n}$, both terms are $\lesssim_M n^{-2}$, whence
$\sum_{i=1}^n |\EE[\mu_i]{Y_{i,n}}| \lesssim_M n^{-1}=o(n\varepsilon_n^2)$.

Write $\widetilde Y_{i,n}:=Y_{i,n}-\EE[\mu_i]{Y_{i,n}}$. On $A_n$ we have $U-\EE[\boldmu]{U}=\sum_{i=1}^n Y_{i,n}$, so for
all large $n$,
$$
\cb{U-\EE[\boldmu]{U}\le -D n\varepsilon_n^2}\cap A_n
\;\subseteq\;
\cb{\sum_{i=1}^n Y_{i,n}\le -D n\varepsilon_n^2}
\;\subseteq\;
\cb{\sum_{i=1}^n \widetilde Y_{i,n}\le -\tfrac{D}{2} n\varepsilon_n^2}.
$$
The \smash{$\widetilde Y_{i,n}$} are independent, mean-zero, satisfy \smash{$|\widetilde Y_{i,n}|\le 2B_n$}, and
$\sum_{i=1}^n \VarInline[\mu_i]{\widetilde Y_{i,n}}\le n\varepsilon_n^2.$
Bernstein's inequality~\citep[Theorem 3.1.7]{gine2021mathematical} therefore gives, for all large $n$,
$$
\PP[\boldmu]{\sum_{i=1}^n \widetilde Y_{i,n}\le -\tfrac{D}{2} n\varepsilon_n^2}
\leq
\exp\p{-\frac{(D n\varepsilon_n^2/2)^2/2}{\,n\varepsilon_n^2+\tfrac{2}{3}B_n\cdot \tfrac{D}{2}n\varepsilon_n^2\,}}
\leq
\exp\p{-c_1 D\,\frac{n\varepsilon_n^2}{B_n}},
$$
where $c_1>0$ depends only on $M$ and we used $B_n\to\infty$. Since $B_n\le 2C_1\log n$ and
$n\varepsilon_n^2\gtrsim \log^2 n$, we have $n\varepsilon_n^2/B_n\gtrsim \log n$, so the right-hand side is at
most $n^{-c_2 D}$ for a constant $c_2>0$. Fix $D$ large enough that $c_2 D\ge 2$.

Combining the last display with $\PP[\boldmu]{A_n^c}\le 2n^{-3}$ yields
$\PP[\boldmu]{U-\EE[\boldmu]{U}\le -D n\varepsilon_n^2}\le 3n^{-2}$. Finally, since
$\EE[\boldmu]{U}\ge -n\varepsilon_n^2$, the event $\{U\le -(1+D)n\varepsilon_n^2\}$ is contained in
$\{U-\EE[\boldmu]{U}\le -D n\varepsilon_n^2\}$, establishing \eqref{eq:denom-U-goal} and hence
\eqref{eq:denom-lb-bern}.
\end{proof}

\begin{lemm}\label{lemm:upper_bound_small}
Suppose that
$\max_{1\le i\le n}|\mu_i|\le M$. and that Specification~\ref{assu:DPGM} holds for $\Pi$ (with $R \geq M$ if $H$ is Base-Compact). Let $\varepsilon_n = \varepsilon_n(\Pi)$ be as
in~\eqref{eq:epsilon_n}. Then for any $D>0$ there exists $C>0$ such that, with
$\mathcal U_n:=\cb{G:\Dhel(f_{\boldmu},f_G)\ge C\varepsilon_n}$, the following holds for all sufficiently large $n$ with probability at least $1-1/n^2$:
$\int_{\mathcal U_n} L_n(f_G)\,\dd\Pi(G)\le \exp\!\big(-D\,n\varepsilon_n^2\big).$
\end{lemm}

\begin{proof}
Write $r_n:=C\varepsilon_n$, $\delta_n:=n^{-2}$,
$
T_n:=M+\sqrt{8\log n}$ and
$
\mathcal A_n:=\{\max_{1\le i\le n}|Z_i|\le T_n\}.
$
Since
$Z_i\sim\mathrm N(\mu_i,1)$ with $|\mu_i|\le M$, a union bound gives
$\PP[\boldmu]{\mathcal A_n^c}\le 2n\exp\{-(T_n-M)^2/2\}=2n^{-3}$.

Recall the class $\mathcal F_{\infty,1,1}$ of all standard normal location
mixtures defined after~\eqref{eq:gaussian_mixture_class}, and define the supremum
norm on the data window by
$\Norm{h}_{\infty,T_n}:=\sup_{|z|\le T_n}\abs{h(z)}$.
With probability $1$, any $G$ in the support of $\Pi$ is of the form
$G=Q*\mathrm N(0,A)\in\mathcal P(\RR)$, and so $f_G\in\mathcal F_{\infty,1,1}$
$\Pi$-almost surely, under either base measure and without any restriction on the
support of $Q$. By \citet[Lemma 2]{zhang2009generalized}, there exists a
$\delta_n$-covering $\{h_1,\ldots,h_N\}$ of $\mathcal F_{\infty,1,1}$ with respect to
$\Norm{\cdot}_{\infty,T_n}$ satisfying
\begin{equation}
\log N
\le
a_1\log^2 (n)\cdot\max\cb{\frac{T_n}{\sqrt{2\log n}},\,1}
\le
a_2\log^2 n
\le
a_2\,n\varepsilon_n^2,
\label{eq:zhang_entropy}
\end{equation}
where $a_1$ is a universal constant, $a_2$ depends only on $M$, and we used
$n\varepsilon_n^2\ge\log^2 n$.

Let $J$ be the set of indices $j\le N$ for which there exists
$f_{0,j}\in\mathcal F_{\infty,1,1}$ with
\begin{equation}
\Norm{f_{0,j}-h_j}_{\infty,T_n}\le\delta_n,
\;\;\;\;
\Dhel(f_{0,j},f_{\boldmu})\ge r_n,
\label{eq:pivot}
\end{equation}
and for each $j\in J$ fix one such $f_{0,j}$. We note that the covering, the set $J$,
and the choices $f_{0,j}$ depend only on $(n, M, \boldmu)$ and not on the data.

If $\Pi(\mathcal U_n)=0$ there is nothing to prove, so assume otherwise. Take any
$G\in\mathcal U_n$ in the support of $\Pi$ and let $h_j$ be a point in the cover with
$\Norm{f_G-h_j}_{\infty,T_n}\le\delta_n$. Since $f_G$ itself
satisfies~\eqref{eq:pivot} for this $j$, we have $j\in J$, and moreover, for all
$|z|\le T_n$,
$
f_G(z)\le h_j(z)+\delta_n\le f_{0,j}(z)+2\delta_n.
$
Define $u_j:=f_{0,j}+2\delta_n\ind_{[-T_n,T_n]}$ for $j\in J$. Since $L_n$
in~\eqref{eq:Ln} evaluates its argument only at $Z_1,\ldots,Z_n$, it follows that
on the event $\mathcal A_n$,
\begin{equation}
\int_{\mathcal U_n} L_n(f_G)\,\dd\Pi(G)\le\max_{j\in J}L_n(u_j).
\label{eq:net_domination}
\end{equation}
We next bound the $(1/2)$-moment of $L_n(u_j)$ under $\PP[\boldmu]{\cdot}$. By
independence and the AM--GM inequality, for any nonnegative function $h$,
\begin{equation}\label{eq:sqrt_moment_basic}
\EE[\boldmu]{\sqrt{L_n(h)}}
=
\prod_{i=1}^n \int \sqrt{\frac{h(z)}{f_{\boldmu}(z)}}\,\varphi(z-\mu_i)\,\dd z
\le
\p{\int \sqrt{h(z)f_{\boldmu}(z)}\,\dd z}^n.
\end{equation}
Next note that
$u_j \geq f_{0,j}$ on all of $\RR$, and hence, using $(\sqrt{a}-\sqrt{b})^2\le a-b$ for $a\ge b\ge0$,
$\Dhel^2(f_{0,j},u_j) \le \frac{1}{2}\int \p{u_j(z)-f_{0,j}(z)}\dd z= 2\delta_n T_n.$
By Cauchy--Schwarz and the definition of Hellinger distance,
\begin{align}
&\int \sqrt{u_j(z) f_{\boldmu}(z)} \dd z  \notag \\ &\quad = \int \sqrt{f_{0,j}(z) f_{\boldmu}(z)}\dd z + \int (\sqrt{u_j(z)} -\sqrt{f_{0,j}(z)}) \sqrt{f_{\boldmu}(z)}\dd z  \notag \\
&\quad \leq \int \sqrt{f_{0,j}(z) f_{\boldmu}(z)}\dd z \,+\, \p{\int \cb{\sqrt{u_j(z)} -\sqrt{f_{0,j}(z)}}^2 \dd z}^{1/2} \p{\int f_{\boldmu}(z) \dd z}^{1/2} \notag \\
&\quad
=
1-\Dhel^2(f_{\boldmu},f_{0,j})+\sqrt{2}\,\Dhel(f_{0,j},u_j)
\le
1-r_n^2+2\sqrt{\delta_n T_n}. \label{eq:affinity_uj}
\end{align}
Since $\delta_n=n^{-2}$, we have $\sqrt{\delta_n T_n}\lesssim_M n^{-1}\log^{1/4}n$,
while $nr_n^2\ge \log^2 n$, and so for all sufficiently large $n$,
$$
\int \sqrt{u_j(z) f_{\boldmu}(z)}\dd z
\le
1-\frac{r_n^2}{2}.
$$
Combining with \eqref{eq:sqrt_moment_basic} yields
\begin{equation}\label{eq:sqrt_moment_uj}
\EE[\boldmu]{\sqrt{L_n(u_j)}}
\le
\p{1-\frac{r_n^2}{2}}^n
\le
\exp\p{-\frac{n r_n^2}{2}}.
\end{equation}
Define the event
$
\mathcal E_n
:=
\{
\max_{j\in J} L_n(u_j)\le \exp(-n r_n^2/4)
\}.
$
By Markov's inequality applied to $\sqrt{L_n(u_j)}$ and \eqref{eq:sqrt_moment_uj}, for each fixed $j\in J$,
$$
\begin{aligned}
\PP[\boldmu]{L_n(u_j)>\exp\p{-\frac{n r_n^2}{4}}}
&=
\PP[\boldmu]{\sqrt{L_n(u_j)}>\exp\p{-\frac{n r_n^2}{8}}} \\ &\le
\exp\p{\frac{n r_n^2}{8}}\EE[\boldmu]{\sqrt{L_n(u_j)}}
\le
\exp\p{-\frac{3n r_n^2}{8}}.
\end{aligned}
$$
The union bound over $j\in J$ (with $|J|\le N$) gives
\begin{equation}\label{eq:En_fail_compound}
\PP[\boldmu]{\mathcal E_n^c}
\le
N\exp\p{-\frac{3n r_n^2}{8}} \leq \exp\p{ \log N -\frac{3n r_n^2}{8}}.
\end{equation}
Since $r_n^2=C^2\varepsilon_n^2$ and $\log N\le a_2 n\varepsilon_n^2$ by \eqref{eq:zhang_entropy},
choosing $C$ sufficiently large implies that the above term can be made $\leq 1/(2n^2)$.

Finally, on $\mathcal E_n \cap \mathcal A_n$, by~\eqref{eq:net_domination},
\begin{equation}\label{eq:numer_ub_compound}
\int_{\mathcal{U}_n} L_n(f_G)\,\dd\Pi(G)
\le
\max_{j\in J} L_n(u_j)
\le
\exp\p{-\frac{n r_n^2}{4}},
\end{equation}
which also can be made $\leq \exp(-Dn\varepsilon_n^2)$ by large enough choice of $C$.
Since $\PP[\boldmu]{\mathcal A_n^c} \le 2n^{-3}$, the event
$\mathcal E_n \cap \mathcal A_n$ has probability at least $1-1/n^2$, which
concludes the proof.
\end{proof}

\subsection{Proof of Theorem~\ref{thm:compound-contraction}}
\label{sec:proof_compound_contraction}

\begin{proof}
Write $\varepsilon_n:=\varepsilon_n(\Pi)$.
Let $
\mathcal U_n
:=
\{
G:\Dhel(f_{\boldmu},f_G)\ge C\varepsilon_n
\},
$, so that
$$
\Pi(\mathcal U_n\mid \boldZ)
=
\int_{\mathcal U_n} L_n(f_G)\,\dd\Pi(G)\,\Big /
\int L_n(f_G)\,\dd\Pi(G).
$$

\noindent\textbf{Step 1:} We first lower bound the denominator. Let $D$ be the constant in
Lemma~\ref{lem:denom-bernstein}, and let $\mathcal K_n$ be the set appearing in the statement of the same lemma. Define
$\mathcal D_n := \cb{ \int L_n(f_G)\,\dd\Pi(G) \ge \exp\!\big\{-(1+D)n\varepsilon_n^2\big\} \Pi(\mathcal K_n) }.$
By Lemma~\ref{lem:denom-bernstein},
$
\PP[\boldmu]{\mathcal D_n^c}\le 3/n^2
$
for all sufficiently large $n$.

We claim that there exists a constant $C_{\mathrm{KL}}>0$ such that
\begin{equation}\label{eq:Kn-prior-mass-compound}
\Pi(\mathcal K_n)
\ge
\exp\!\big\{-C_{\mathrm{KL}} n\varepsilon_n^2\big\}
\end{equation}
Indeed, by Lemma~\ref{lem:G-small-k} (which supplied the set $\mathcal{K}_n$),
the form
\[
\Pi(\mathcal K_n) \ge c_1 \exp\!\big(-c_2\log^2(\varepsilon_n^{-1})\big) \varepsilon_n^r \log^{-r}(\varepsilon_n^{-1}) \exp\!\big(-c_3\varepsilon_n^{-\kappa} \log^\kappa(\varepsilon_n^{-1})\big).
\]
Since $n\varepsilon_n^2\gtrsim \log^2 n$ and
$\log(\varepsilon_n^{-1})\lesssim \log n$, the logarithmic terms are bounded by a constant
multiple of $n\varepsilon_n^2$. Moreover, if $\kappa>0$, then using
$\varepsilon_n
\ge
n^{-1/(2+\kappa)}\log^{\kappa/(2+\kappa)} n
$
gives
$\varepsilon_n^{-\kappa}\log^\kappa(\varepsilon_n^{-1}) \lesssim n^{\kappa/(2+\kappa)}\log^{2\kappa/(2+\kappa)} n \le n\varepsilon_n^2,$
while the case $\kappa=0$ is immediate. This proves~\eqref{eq:Kn-prior-mass-compound}.

Combining the definition of $\mathcal D_n$ with~\eqref{eq:Kn-prior-mass-compound}, we obtain that on
$\mathcal D_n$,
\begin{equation}\label{eq:denom-lb-compound}
\int L_n(f_G)\,\dd\Pi(G)
\ge
\exp\!\big\{-C_{\mathrm{den}}n\varepsilon_n^2\big\},
\end{equation}
for a constant $C_{\mathrm{den}}>0$ depending only on $M$ and the parameters in
Specification~\ref{assu:DPGM}.\\

\noindent\textbf{Step 2:} We next upper bound the numerator. Choose $D_1>C_{\mathrm{den}}$. By
Lemma~\ref{lemm:upper_bound_small}, there exists a constant $C>0$ such that,
the event
$
\mathcal E_n
:=
\{
\int_{\mathcal U_n} L_n(f_G)\,\dd\Pi(G)
\le
\exp\!\big\{-D_1 n\varepsilon_n^2\big\}
\}
$
satisfies
$
\PP[\boldmu]{\mathcal E_n^c}\le 1/n^2
$
for all sufficiently large $n$.\\

\noindent\textbf{Step 3:} On $\mathcal D_n\cap\mathcal E_n$, equations~\eqref{eq:denom-lb-compound} and the definition of
$\mathcal E_n$ give
$
\Pi(\mathcal U_n\mid \boldZ)
\le
\exp\!\big\{-(D_1-C_{\mathrm{den}})n\varepsilon_n^2\big\}.
$
Setting $c:=D_1-C_{\mathrm{den}}>0$, we conclude that
$\PP[\boldmu]{ \Pi(\mathcal U_n\mid \boldZ)> \exp\!\big\{-c\,n\varepsilon_n^2\big\} } \le \PP[\boldmu]{\mathcal D_n^c} + \PP[\boldmu]{\mathcal E_n^c} \le \frac{4}{n^2}.$
This is precisely the desired claim.
\end{proof}

\subsection{Preliminary lemmata for regret proofs}

We will need the following function: $\tilde{L}(\rho) := \sqrt{-\log(2\pi \rho^2)}$.
\begin{lemm}[Lemma A.1. in~\citet{jiang2009general}]
\label{lemm:proposition_1_jiang_zhang}
For any $z \in \mathbb R$, we have that:
$$
\begin{aligned}
&\abs{\frac{f_G'(z)}{f_G(z)}} \leq \tilde{L}(f_G(z)), \;\;\; -1 \leq \frac{f_G''(z)}{f_G(z)} \leq   \tilde{L}^2(f_G(z)).
\end{aligned}
$$
\end{lemm}

\begin{lemm}[Theorem E.1.\ in~\citet{saha2020nonparametric}]
\label{lemm:saha_e1}
Suppose that $0< \rho < (2\pi e)^{-1/2}$. Let $G, H$ be two distributions on $\RR$. Then:

$$\int \abs{\frac{f_G'(z)}{f_G(z)\lor \rho} - \frac{f_H'(z)}{f_H(z)\lor \rho}}^2f_G(z)\dd z \lesssim \Dhel^2(f_G, f_H) \cdot \max\cb{ (-\log(2\pi \rho^2))^3,\; \abs{ \log \Dhel(f_G,f_H)}}.$$
\end{lemm}

\begin{lemm}[Posterior moments upper bounded]
\label{lemm:posterior_moments_upper_bounded}
Suppose Specification~\ref{assu:DPGM} holds and $n\geq2$. Then there exists a constant $C$
that depends on the specification parameters such that
$\abs{\widehat{\mu}_i(\Pi)} \leq C (|Z_i| + \sqrt{\log n}), \; \EE[\Pi]{\mu_i^2 \mid \boldZ} \leq C^2(Z_i^2 + \log n).$
\end{lemm}

\begin{proof}
We start with the representation in Proposition~\ref{prop:datta_loo}. We first bound
$
f_{\widebar{G}(\boldZ_{-i})}(Z_i).
$
Write $G=Q*\mathrm N(0,A)$, and let
$
\mu_j\mid\tau_j,A\sim\mathrm N(\tau_j,A)
$
with $\tau_j\sim Q$ (all independent). Conditional on $(A,\boldtau_{-i},\boldZ_{-i})$, the observations
$\boldZ_{-i}$ provide no additional information about $Q$, and hence Dirichlet process conjugacy gives, for any measurable set $S$,
$$
\EE[\Pi]{Q(S)\mid A,\boldtau_{-i},\boldZ_{-i}}
=
\frac{\alpha}{\alpha+n-1}H(S)
+
\frac{1}{\alpha+n-1}\sum_{j\ne i}\delta_{\tau_j}(S).
$$
It follows that
$\EE[\Pi]{G(S)\mid A,\boldtau_{-i},\boldZ_{-i}} \geq \frac{\alpha}{\alpha+n-1} \{H*\mathrm N(0,A)\}(S).$
Taking conditional expectations given $\boldZ_{-i}$ therefore yields
$$
\widebar{G}(\boldZ_{-i})(S)
=
\EE[\Pi]{G(S)\mid\boldZ_{-i}}
\geq
\frac{\alpha}{\alpha+n-1}\eta(S), \;\;\text{ where }\;\eta(S)
:=
\inf_{A\in[0,\bar A]}
\{H*\mathrm N(0,A)\}(S).
$$
Under Base-Compact, let $R$ be the radius of the support of $H$; under
Base-Gauss, fix any $R>0$. In either case,
$
\eta_0:=\eta([-R,R])>0.
$
Thus,
$$
\begin{aligned}
f_{\widebar{G}(\boldZ_{-i})}(Z_i)
\geq
\frac{1}{\sqrt{2\pi}}
\int_{-R}^R
\exp\p{-\frac{(Z_i-\mu)^2}{2}}
\dd\widebar{G}(\boldZ_{-i})(\mu)
\geq
\frac{\alpha\eta_0}{\alpha+n-1}
\frac{1}{\sqrt{2\pi}}
\exp\p{-Z_i^2-R^2}.
\end{aligned}
$$
Hence
$
-\log f_{\widebar{G}(\boldZ_{-i})}(Z_i)
\lesssim
\log n+Z_i^2.
$
From Lemma~\ref{lemm:proposition_1_jiang_zhang} and the Eddington--Tweedie formula, we thus get
$$
\abs{\widehat{\mu}_i(\Pi)}
\lesssim
\abs{Z_i}
+
\sqrt{\log n+Z_i^2}
\lesssim
|Z_i|+\sqrt{\log n}.
$$
We now turn to the second moment. Also using the LOO representation and
Lemma~\ref{lemm:proposition_1_jiang_zhang}, we have
$$
\begin{aligned}
\EE[\Pi]{\mu_i^2 \mid \boldZ}
&=
\EE[\widebar{G}(\boldZ_{-i})]{\mu_i^2 \mid Z_i}
\,\leq\,
2Z_i^2
+
2\EE[\widebar{G}(\boldZ_{-i})]{(\mu_i-Z_i)^2 \mid Z_i}\\
&=
2Z_i^2
+
2
+
2\frac{
f''_{\widebar{G}(\boldZ_{-i})}(Z_i)
}{
f_{\widebar{G}(\boldZ_{-i})}(Z_i)
}
\,\leq\,
2Z_i^2
+
2
+
2\tilde{L}^2\p{
f_{\widebar{G}(\boldZ_{-i})}(Z_i)
}
\,\leq\,
C^2(\log n+Z_i^2),
\end{aligned}
$$
for some constant $C$ depending only on the specification parameters.
\end{proof}

The following lemma is stated using the heteroscedastic notation.

\begin{lemm}\label{lem:hellinger-compactness}
For every $M<\infty$, there exist constants $\rho_0>0$ and $M_0<\infty$, depending only on $M$ and $\bar\sigma^2$, such that, for all
$G_0\in\mathcal P([-M,M])$ and $G\in\mathcal P(\RR)$,
$\Dhel\p{\boldf_{G_0},\boldf_G}\le \rho_0 \;\Longrightarrow\; G([-M_0,\,M_0]) \ge \frac{1}{2}.$
\end{lemm}

\begin{proof}
Take $R$ large enough that $\PP{\abs{\mathrm N(0,\bar\sigma^2)}>R}\le1/16$, and set $M_0:=M+2R$ and $B:=[-M-R,M+R]$. If $G_0$ is supported on $[-M,M]$, then $\int_B f_{G_0,\sigma_i^2}(z)\,\dd z\ge 15/16$ for every $i$. If $G([-M_0,M_0])<1/2$, then at least half of the mixing mass lies outside $[-M_0,M_0]$, and for such means the probability of falling in $B$ is at most $1/16$, uniformly over $\sigma_i^2\le\bar\sigma^2$; the remaining half of the mass contributes at most $1/2$. Hence $\int_B f_{G,\sigma_i^2}(z)\,\dd z\le 9/16$, and so
the total variation distance is at least $3/8$ for every $i$.
Using $\TV(f,h)\le\sqrt2\,\Dhel(f,h)$, every coordinate Hellinger distance, and therefore also the average $\Dhel\p{\boldf_{G_0},\boldf_G}$, is bounded below by $3/(8\sqrt2)$. The claim follows by taking $\rho_0=3/(16\sqrt2)$.
\end{proof}

\subsection{Proof of Theorem~\ref{theo:main_risk_bound_eq}}
\label{subsec:main_cd_regret_proof}
\begin{proof}

Our proof follows six steps.\\

\noindent{\textbf{Step 1:}} Fix $\boldmu\in[-M,M]^n$ and recall the notation $\compoundG$ in \eqref{eq:compound_prior}. Also note the formal equality $f_{\boldmu}=f_{\compoundG}$. Let
$\widehat{\boldmu}^{\star}:=\widehat{\boldmu}^{\mathrm B}(\compoundG) =\bigl(\delta_{\compoundG}(Z_i)\bigr)_{i=1}^n,$
be the oracle separable rule,
so that, by Theorem~\ref{thm:fundamental-compound},
\refstepcounter{equation}\label{eq:opt-separable-is-bayes}$\inf_{\boldt\in\mathcal T^{\mathrm S}} R\bigl(\boldt(\boldZ),\boldmu\bigr) =R\bigl(\widehat{\boldmu}^{\star},\boldmu\bigr).$
By Theorem~\ref{theo:equiv_and_simple}, for a constant $C_M$ depending only on $M$,
\begin{equation}\label{eq:reduce-PE-to-S}
\inf_{\boldt\in\mathcal T^{\mathrm S}} R\bigl(\boldt(\boldZ),\boldmu\bigr)
\le
\inf_{\boldt\in\mathcal T^{\mathrm{PE}}} R\bigl(\boldt(\boldZ),\boldmu\bigr) + \frac{C_M}{\sqrt n}.
\end{equation}
Consequently, it suffices to prove that for some $C>0$ depending only on $M$ and the prior specification,
\begin{equation}\label{eq:goal-separable}
\sup_{\boldmu\in[-M,M]^n}\Bigl\{R\bigl(\widehat{\boldmu}(\Pi),\boldmu\bigr)
-R\bigl(\widehat{\boldmu}^{\star},\boldmu\bigr)\Bigr\}
\le C \log^{3/2}n\cdot \varepsilon_n(\Pi),
\end{equation}
since combining \eqref{eq:goal-separable} with \eqref{eq:reduce-PE-to-S} yields the stated theorem
(after absorbing $C_M/\sqrt n$ into the larger right-hand side).\\

\noindent\textbf{Step 2:}
By the triangle inequality,
\begin{equation}\label{eq:minkowski-BB-vs-star}
R\bigl(\widehat{\boldmu}(\Pi),\boldmu\bigr)-R\bigl(\widehat{\boldmu}^{\star},\boldmu\bigr)
\le
\Biggl\{\frac1n\sum_{i=1}^n
\EE[\boldmu]{\Bigl(\widehat\mu_i(\Pi)-\delta_{\compoundG}(Z_i)\Bigr)^2}\Biggr\}^{1/2}.
\end{equation}
Thus it remains to show that the average inside the braces is $\lesssim \log^{3}n \cdot \varepsilon_n^2$
uniformly in $\boldmu\in[-M,M]^n$.\\

\noindent \textbf{Step 3:}
Let $T_n:=M+\sqrt{6\log n}$ and define the truncation event
$
\mathcal A_n:=\{\max_{1\le i\le n}|Z_i|\le T_n\}.
$
Since $Z_i\sim \mathrm N(\mu_i,1)$ and $|\mu_i|\le M$, a union bound yields
$\PP[\boldmu]{\mathcal A_n^c}\le 2n\exp\{-(T_n-M)^2/2\}\le 2/n^2$.

Let $r_n:=C_0\varepsilon_n$, where $C_0$ is the constant from
Theorem~\ref{thm:compound-contraction}, and define
$$
\mathcal H_n:=\cb{G\,:\, \Dhel\bigl(f_{\boldmu},f_G\bigr)\le r_n}\cap \cb{G\,:\, G[-M_0, M_0] \geq 1/2},
$$
where $M_0$ is given by Lemma~\ref{lem:hellinger-compactness}. By the same lemma, for large enough $n$, the second condition in the definition of $\mathcal H_n$ is automatically satisfied for all $G$ such that $\Dhel(f_{\boldmu},f_G)\le r_n$. Moreover, by Theorem~\ref{thm:compound-contraction}, there exist constants $c>0$ and (for all $n\ge2$)
an event $\mathcal C_n$ with $\PP[\boldmu]{\mathcal C_n^c}\le 4/n^2$ such that on $\mathcal C_n$,
\refstepcounter{equation}\label{eq:posterior-mass-outside}$\Pi(\mathcal H_n^c\mid \boldZ)\le \exp\p{-c n \varepsilon_n^2}.$
In this step, we control the regret in the bad event $\mathcal{A}_n^c \cup \mathcal{C}_n^c$.
Now recall that $\delta_{\compoundG}(Z_i)\in [-M,M]$ almost surely and by Lemma~\ref{lemm:posterior_moments_upper_bounded}, $\widehat\mu_i(\Pi) \lesssim |Z_i| + \sqrt{\log n}$. Thus:
$$
\begin{aligned}
&\EE[\boldmu]{\Bigl(\widehat\mu_i(\Pi)-\delta_{\compoundG}(Z_i)\Bigr)^2 \ind\{\mathcal{A}_n^c \cup \mathcal{C}_n^c\}} \\ &\quad\lesssim \EE[\boldmu]{(Z_i^2 + \log n) \ind\{\mathcal{A}_n^c \cup \mathcal{C}_n^c\}} \\ &\quad\lesssim \EE[\boldmu]{ Z_i^4 + \log^2 n}^{1/2} \PP[\boldmu]{\mathcal{A}_n^c \cup \mathcal{C}_n^c}^{1/2}.
\end{aligned}
$$

\noindent \textbf{Step 4:} We now control the regret in the ``good'' event $\mathcal{A}_n \cap \mathcal{C}_n$.
For every $i$, by~\eqref{eq:bnp-post-mean},
$
\widehat\mu_i(\Pi)
=\int \delta_G(Z_i)\,\dd\Pi(G\mid \boldZ).
$
Moreover let $\pi_n(\boldZ) = \Pi( \mathcal{H}_n^c \mid \boldZ)$ and let $\Pi_{\mathcal{H}_n^c}(\cdot \mid \boldZ)$, resp.  $\Pi_{\mathcal{H}_n}(\cdot \mid \boldZ)$ denote the normalized restriction of  $\Pi(\cdot \mid \boldZ)$ to $\mathcal{H}_n^c$, resp.\ $\mathcal{H}_n$. We get,
$\widehat\mu_i(\Pi) \,=\, (1- \pi_n(\boldZ))\int \delta_G(Z_i)\,\dd\Pi_{\mathcal{H}_n}(G\mid \boldZ) \,+\pi_n(\boldZ) \int \delta_G(Z_i)\,\dd\Pi_{\mathcal{H}_n^c}(G\mid \boldZ).$
Then,
\begin{equation}
\begin{aligned}
\p{\widehat\mu_i(\Pi)-\delta_{\compoundG}(Z_i)}^2  &\leq 2(1-\pi_n(\boldZ))^2 \p{\delta_{\compoundG}(Z_i) - \int \delta_G(Z_i)\,\dd\Pi_{\mathcal{H}_n}(G\mid \boldZ)}^2  \\
& \;\;\;\; + \,\, 2\pi_n(\boldZ)^2\p{\delta_{\compoundG}(Z_i) - \int \delta_G(Z_i)\,\dd\Pi_{\mathcal{H}_n^c}(G\mid \boldZ)}^2.
\end{aligned}
\label{eq:posterior_split}
\end{equation}
We control the above two terms separately. \\

\noindent \textbf{Step 5:} Now recalling that $\delta_{\compoundG}(Z_i) \in [-M,M]$ and Jensen's inequality
$$
\begin{aligned}
\pi_n(\boldZ)^2\p{\delta_{\compoundG}(Z_i) - \int \delta_G(Z_i)\,\dd\Pi_{\mathcal{H}_n^c}(G\mid \boldZ)}^2 &\lesssim \pi_n(\boldZ)^2 \p{1\,+\,\int \delta_G^2(Z_i)\,\dd\Pi_{\mathcal{H}_n^c}(G\mid \boldZ)}.
\end{aligned}
$$
Meanwhile,
$$
\pi_n(\boldZ) \int \delta_G^2(Z_i)\,\dd\Pi_{\mathcal{H}_n^c}(G\mid \boldZ) = \int_{\mathcal{H}_n^c} \delta_G^2(Z_i)\,\dd\Pi(G\mid \boldZ)  \leq \int \delta_G^2(Z_i)\,\dd\Pi(G\mid \boldZ) \leq \EE[\Pi]{\mu_i^2 \mid \boldZ}.
$$
The right-hand side can be upper bounded using Lemma~\ref{lemm:posterior_moments_upper_bounded} by $\log n + Z_i^2$ (up to a constant). Moreover recall that on the event $\mathcal{C}_n$, $\pi_n(\boldZ) \leq \exp(-c n\varepsilon_n^2)$. Thus on $\mathcal{A}_n \cap \mathcal{C}_n$,
$$
\pi_n(\boldZ)^2\p{\delta_{\compoundG}(Z_i) - \int \delta_G(Z_i)\,\dd\Pi_{\mathcal{H}_n^c}(G\mid \boldZ)}^2 \lesssim \exp(-c  n \varepsilon_n^2) \p{\log n + T_n^2}
$$
Thus the contribution of this part to the regret is negligible.\\

\noindent \textbf{Step 6:} We continue by controlling the first term in~\eqref{eq:posterior_split}, first upper bounding \smash{$(1-\pi_n(\boldZ))$} by $1$. Then, Jensen's inequality gives, almost surely,
\begin{equation}\label{eq:jensen-pointwise}
\p{\int \delta_G(Z_i)\,\dd\Pi_{\mathcal{H}_n}(G\mid \boldZ) - \delta_{\compoundG}(Z_i)}^2
\le
\int \p{\delta_G(Z_i)-\delta_{\compoundG}(Z_i)}^2\,\dd\Pi_{\mathcal{H}_n}(G\mid \boldZ).
\end{equation}
Let,
\begin{equation}
\widehat{\Delta}_n(G):=\frac1n\sum_{i=1}^n\Bigl(\delta_G(Z_i)-\delta_{\compoundG}(Z_i)\Bigr)^2.
\end{equation}
Averaging \eqref{eq:jensen-pointwise} over $i$ yields the pointwise bound
\begin{equation}\label{eq:avg-jensen}
\frac{1}{n}\sum_{i=1}^n \p{\int \delta_G(Z_i)\,\dd\Pi_{\mathcal{H}_n}(G\mid \boldZ) - \delta_{\compoundG}(Z_i)}^2
\le \int \widehat{\Delta}_n(G)\,\dd\Pi_{\mathcal{H}_n}(G\mid \boldZ)
\le \sup_{G\in\mathcal H_n}\widehat{\Delta}_n(G).
\end{equation}
Therefore, it suffices to show that
\begin{equation}\label{eq:supDelta-goal}
\EE[\boldmu]{\sup_{G\in\mathcal H_n}\widehat{\Delta}_n(G) \ind\{\mathcal{A}_n\}}
\lesssim \log^3n \cdot \varepsilon_n^2,
\qquad\text{uniformly in }\boldmu\in[-M,M]^n.
\end{equation}
On $\mathcal A_n$, all $Z_i\in[-T_n,T_n]$.
For any $G$ with $G([-M_0,M_0]) \geq 1/2$ and any $|z|\le T_n$,
$f_G(z)=\int \varphi(z-u)\,\dd G(u) \ge \inf_{u\in[-M_0,M_0]}\varphi(z-u)/2=\varphi(|z|+M_0)/2\ge \varphi(T_n+M_0)/2=:\rho_n.$
Following~\citet{jiang2009general}, for any $\rho>0$, define the regularized Bayes rule $$\delta_G^\rho(z):=z+\frac{f'_G(z)}{f_G(z)\lor\rho}.$$
The regularized Bayes rule at $\rho_n$ and any $|z| \leq T_n$
satisfies
$\delta_G^{\rho_n}(z)=\delta_G(z)$.
In particular, on $\mathcal A_n$,
$$
\widehat{\Delta}_n(G) = \widehat{\Delta}^{\rho_n}(G),\;\;\;  \text{ where }\;   \widehat{\Delta}^{\rho_n}(G)  :=\frac{1}{n}\sum_{i=1}^n\Bigl(\delta_G^{\rho_n}(Z_i)-\delta_{\compoundG}^{\rho_n}(Z_i)\Bigr)^2.
$$
By Proposition~3 of \citet{jiang2009general} (applied with $M$ there replaced by $T_n$ and $\rho=\rho_n$),
there exists a finite set of mixing distributions $\cb{H_1,\ldots,H_{N_n}}\subseteq\mathcal H_n$ with $\log N_n \lesssim \log^2 n$ such that for all $G\in\mathcal H_n$
\begin{equation}\label{eq:supnorm-net}
\min_{1\le j\le N_n}\sup_{|z|\le T_n}\abs{\delta_G^{\rho_n}(z)-\delta_{H_j}^{\rho_n}(z)}
\leq n^{-2}.
\end{equation}
For $G\in\mathcal H_n$, pick $j(G)\in\argmin_{j}\sup_{|z|\le T_n}|\delta_G^{\rho_n}(z)-\delta_{H_j}^{\rho_n}(z)|$.
On $\mathcal A_n$, we have for each $i$,
$$
\Bigl(\delta_G^{\rho_n}(Z_i)-\delta_{\compoundG}^{\rho_n}(Z_i)\Bigr)^2
\le
2\Bigl(\delta_{H_{j(G)}}^{\rho_n}(Z_i)-\delta_{\compoundG}^{\rho_n}(Z_i)\Bigr)^2 + 8\cdot n^{-2},
$$
and hence
\begin{equation}\label{eq:supDelta-by-net}
\sup_{G\in\mathcal H_n}\widehat{\Delta}^{\rho_n}_n(G)
\le
2\max_{1\le j\le N_n}\widehat{\Delta}^{\rho_n}_n(H_j) \;+\; 8 n^{-2}
\qquad\text{on }\mathcal A_n.
\end{equation}
Now fix $j\le N_n$ and define
$$
X_{ij}:=\bigl(\delta_{H_j}^{\rho_n}(Z_i)-\delta_{\compoundG}^{\rho_n}(Z_i)\bigr)^2 \lesssim |\log \rho_n|  \lesssim \log n,
$$
where the first inequality follows from Lemma~\ref{lemm:proposition_1_jiang_zhang}.
With this notation, we have that $\widehat{\Delta}_n^{\rho_n}(H_j)=n^{-1}\sum_{i=1}^n X_{ij}$.
Since the $Z_i$ are independent under $\PP[\boldmu]{\cdot}$, the $X_{ij}$ are independent (for fixed $j$).
Moreover,
$$
\begin{aligned}
\EE[\boldmu]{\widehat{\Delta}^{\rho_n}_n(H_j)}
&=
\int \Bigl(\delta^{\rho_n}_{H_j}(z)-\delta^{\rho_n}_{\compoundG}(z)\Bigr)^2 f_{\boldmu}(z)\,\dd z
\\ &\lesssim  \Dhel^2(f_{\boldmu}, f_{H_j}) \cdot  \{(-\log(2\pi \rho_n^2))^3\lor \abs{ \log \Dhel(f_{\boldmu}, f_{H_j})}\},
\end{aligned}
$$
by using Lemma~\ref{lemm:saha_e1}. Hence, we get,
$%
\EEInline[\boldmu]{\widehat{\Delta}_n^{\rho_n}(H_j)}
\lesssim \log^3 n \cdot \varepsilon_n^2.
$
Moreover, note that $0\le X_{ij}\lesssim \log n$ and
$
\Var[\mu_i]{X_{ij}} \leq \EE[\mu_i]{X_{ij}^2} \lesssim \log n \EE[\mu_i]{X_{ij}}.
$
Thus,
$$
\sum_{i=1}^n \Var[\mu_i]{X_{ij}} \lesssim \log n \cdot n \cdot \EE[\boldmu]{\widehat{\Delta}_n^{\rho_n}(H_j)} \lesssim \log^4 n \cdot n \cdot \varepsilon_n^2.
$$
Applying Bernstein’s inequality to $\widehat{\Delta}_n(H_j)$ yields,
for a sufficiently large constant $K>0$ (depending only on $M$),
$$
\PP[\boldmu]{\widehat{\Delta}^{\rho_n}_n(H_j) > K \log^3n \cdot \varepsilon_n^2}
\le \exp\p{ -c_1 \log n \cdot n \varepsilon_n^2} \leq
\exp\p{-c_1' \log^2 n},
$$
for some $c_1,c_1'$ that can be made arbitrarily large by choosing $K$ large enough.
By the union bound over $j\le N_n$ and \eqref{eq:supnorm-net} (so $\log N_n\lesssim \log^2 n$),
we conclude that for all large $n$,
\begin{equation*}
\begin{aligned}
&\PP[\boldmu]{\max_{1\le j\le N_n}\widehat{\Delta}^{\rho_n}_n(H_j) > K \log^3n \cdot \varepsilon_n^2}\\
&\quad\le
N_n\,\exp\p{-c_1\log^2 n}
\le
\exp\p{C\log^2 n - c_1\log^2 n}
\le
\frac{1}{n^2},
\end{aligned}
\end{equation*}
for $c_1$ large enough.
Combining the above with \eqref{eq:supDelta-by-net}, and using $\PP[\boldmu]{\mathcal A_n^c}\le 2/n^2$,
we obtain for all $n\ge2$ (after adjusting constants),
\[
\PP[\boldmu]{\sup_{G\in\mathcal H_n}\widehat{\Delta}_n^{\rho_n}(G) \gtrsim \log^3 n \cdot \varepsilon_n^2}
\le
\PP[\boldmu]{\mathcal A_n^c}
+\PP[\boldmu]{\max_{j\le N_n}\widehat{\Delta}_n^{\rho}(H_j) \gtrsim \log^3 n \cdot \varepsilon_n^2}
\lesssim \frac{1}{n^2}.
\]
This computation implies
\eqref{eq:supDelta-goal} and hence concludes the proof.
\end{proof}

\section{Proofs for CLOSE}

\subsection{Notation}
Throughout this section $\boldsigma^2 = (\sigma_1^2,\dots,\sigma_n^2)$ is treated as fixed and satisfies $0 < \underline\sigma^2 \leq \sigma_i^2 \leq \overline\sigma^2 < \infty $ for all $i$. All probabilities and expectations are conditional on these variances.  For any two vectors of densities $\boldf=(f_1,\ldots,f_n)$ and $\boldh=(h_1,\ldots,h_n)$ recall the average squared Hellinger distance $\Dhel^2(\boldf,\boldh)$ defined in~\eqref{eq:avg_hellinger_het} and also define,
\begin{equation}
\DKL{\boldf}{\boldh}
:=
\frac1n\sum_{i=1}^n
\DKL{f_i}{h_i},\;\;
\VKL{\boldf}{\boldh} := \frac{1}{n}\sum_{i=1}^n \VKL{f_i}{h_i},
\label{eq:vectorized_metrics}
\end{equation}
where $\VKL{f_i}{h_i}$ is defined in~\eqref{eq:vkl}.
It will be convenient to extend the latter to higher moments, i.e., for $k\geq 2$:
$$
\VKLk{\boldf}{\boldh}
=
\frac1n\sum_{i=1}^n
\int f_i(z)
\left|
\log\frac{f_i(z)}{h_i(z)}
-
\DKL{f_i}{h_i}
\right|^k \dd z.
$$
Notice that for $k=2$, $\VKLk{\boldf}{\boldh}=\VKL{\boldf}{\boldh}$. For $k\ge2$, define
$$
B_k(\boldf_{\star},\varepsilon)
=
\left\{
\theta:
\DKL{\boldf_{\star}}{\boldf_\theta}\le \varepsilon^2,\;
\VKLk{\boldf_\star}{\boldf_\theta}\le \varepsilon^k
\right\}.
$$
Let $f_{\star,\sigma_i^2}$, $f_{G,\star,\sigma_i^2}$, and $f_{\theta,\sigma_i^2}$ denote, respectively, the true marginal density, the marginal density obtained by varying $G$ while fixing the nuisance components at their true values $(\eta_{1,\star},s_{\star})$, and the full candidate marginal density. The vectorized counterparts are denoted by $ (\boldf_{\star}, \boldf_{G,\star}, \boldf_{\theta}  )$.

\subsection{Lemmata and proof of contraction}

\begin{lemm}[Softplus rescaling]
\label{lem:softplus-rescaling}
Let $\mathcal{X} = [\underline{\sigma},\overline{\sigma}]$ and $
\Lambda(t)=\log(1+e^t)$. Define $
s_\eta=\Lambda^{1/2}\circ\eta,
$
and suppose that $\alpha>1/2$ and
$s_\star:\mathcal X\to[0,\infty)$ satisfies
$
s_\star\in H^\alpha\cap\Sigma^\alpha $ and $
\max\left\{
\Norm{s_\star}_{H^\alpha},
\Norm{s_\star}_{\Sigma^\alpha}
\right\}
\leq B.
$
Suppose moreover that, for all sufficiently large $n$, there exists
$\eta_n\in H^{\alpha+1/2}$ such that
$\Norm{s_\star-s_{\eta_n}}_\infty \leq B\varepsilon_n \; , \; \Norm{\eta_n}_{H^{\alpha+1/2}} \leq B\sqrt n\,\varepsilon_n,$
where $\varepsilon_n\asymp n^{-\alpha/(2\alpha+1)}$.
Then, for every fixed $c>0$, there exist a constant $B_c<\infty$
and functions $\widetilde\eta_n\in H^{\alpha+1/2}$ such that
$\Norm{c s_\star-s_{\widetilde\eta_n}}_\infty \leq B_c\varepsilon_n \; , \; \Norm{\widetilde\eta_n}_{H^{\alpha+1/2}} \leq B_c\sqrt n\,\varepsilon_n$
for all sufficiently large $n$.
\end{lemm}

\begin{proof}
Throughout the proof, $C_c$ denotes a constant that may change from
line to line. Let
$q(x)=\Lambda^{-1}(x)=\log(e^x-1) \; , \; d_c(x)=q(c^2x)-q(x), \; x>0.$
Define $h(x) := (e^x-1)/x$ for $x\neq 0$ and $h(0):=1$.
The function $h$ is $C^\infty$ and strictly positive on
$[0,\infty)$, and
$d_c(x) = 2\log c+\log h(c^2x)-\log h(x).$
Thus, $d_c$ extends to a $C^\infty$ function on $[0,\infty)$,
with $d_c(0)=2\log c$. Since $\alpha>1/2$, the spaces $H^\alpha$ and $\Sigma^\alpha$ are
closed under pointwise multiplication. Hence
$s_\star^2\in H^\alpha\cap\Sigma^\alpha$. By smooth composition, we also have
$
g_c:=d_c\circ s_\star^2
\in H^\alpha\cap\Sigma^\alpha $ and $
\Norm{g_c}_{H^\alpha}
+
\Norm{g_c}_{\Sigma^\alpha}
\leq C_c.
$
Lemmas~11.36--11.37 in~\citet{ghosal2017fundamentals} therefore
give functions $g_{c,n}\in H^{\alpha+1/2}$ satisfying
\begin{equation}
\label{eq:gc-approx}
\Norm{g_{c,n}-g_c}_\infty
\leq C_c\varepsilon_n \; , \;
\Norm{g_{c,n}}_{H^{\alpha+1/2}}
\leq C_c\sqrt n\,\varepsilon_n.
\end{equation}
Set $
s_n=s_{\eta_n},
v_n=s_n^2=\Lambda(\eta_n)$.
Since $s_n$ converges uniformly to the bounded function $s_\star$,
the functions $s_n$ are uniformly bounded. Consequently,
\[
\Norm{v_n-s_\star^2}_\infty \leq \Norm{s_n-s_\star}_\infty \left( \Norm{s_n}_\infty+\Norm{s_\star}_\infty \right) \leq C\varepsilon_n.
\]
The ranges of $v_n$ and $s_\star^2$ therefore lie in a common compact
interval. Since $d_c$ is Lipschitz on this interval,
$
\Norm{d_c(v_n)-g_c}_\infty
=
\Norm{d_c(v_n)-d_c(s_\star^2)}_\infty
\leq C_c\varepsilon_n.
$
Together with~\eqref{eq:gc-approx}, this implies
\begin{equation}
\label{eq:correction-close}
\Norm{g_{c,n}-d_c(v_n)}_\infty
\leq C_c\varepsilon_n.
\end{equation}

Now define $
\widetilde\eta_n:=\eta_n+g_{c,n}.
$
Since $q(v_n)=\eta_n$, we have
$
q(c^2v_n)=\eta_n+d_c(v_n).
$
Moreover,
$
s_{q(c^2v_n)}
=
\sqrt{\Lambda(q(c^2v_n))}
=
c\sqrt{v_n}
=
c s_n.
$
Since the map $\Lambda^{1/2}$ is globally Lipschitz, it follows from~\eqref{eq:correction-close} that
$
\Norm{s_{\widetilde\eta_n}-c s_n}_\infty
\leq C_c\varepsilon_n.$
Therefore,
$\Norm{s_{\widetilde\eta_n}-c s_\star}_\infty \leq \Norm{s_{\widetilde\eta_n}-c s_n}_\infty + c\Norm{s_n-s_\star}_\infty \leq C_c\varepsilon_n.$
Finally, by the triangle inequality and~\eqref{eq:gc-approx}, we obtain
$\Norm{\widetilde\eta_n}_{H^{\alpha+1/2}} \leq \Norm{\eta_n}_{H^{\alpha+1/2}} + \Norm{g_{c,n}}_{H^{\alpha+1/2}} \leq C_c\sqrt n\,\varepsilon_n.$
\end{proof}

\begin{lemm}[Small ball for CLOSE] \label{lem-small-close}
Suppose the hypothesis of Theorem \ref{thm:close_contraction} holds. Let $\varepsilon_n \asymp n^{-\alpha/(2 \alpha+1)} $. Then, for any $k \geq 2$, there exists constants $c,c' > 0$ such that $\Pi \big( B_k(\boldf_{\star},\varepsilon_n) \big) \geq c \exp(- c' n \varepsilon_n^2).$

\end{lemm}

\begin{proof}
Under the working model for $\mu_i$, we have
$Z_i = m_{\star}(\sigma_i) + s_\star(\sigma_i)\xi + \varepsilon_i, \; \varepsilon_i\sim \mathrm{N}(0,\sigma_i^2), \; \xi\sim G_{\star}.$
Note that the marginal distribution is unchanged under the rescaling
$
(s_\star,\xi)\mapsto (c s_\star,c^{-1}\xi), \; c>0.
$
By Lemma~\ref{lem:softplus-rescaling}, the regularity conditions are preserved under any fixed rescaling of this form, up to a change in constants. Moreover, since the small-ball set is defined in terms of the induced marginal density, prior mass near any observationally equivalent representation contributes to the same small-ball probability. We may therefore choose $c$ sufficiently large so that the corresponding $G_\star$ is supported in $[-R,R]$, the support of the base measure $H$. We fix such a representation for the remainder of the proof and suppress the dependence on $c$ from the notation. Define
\[
\mathcal G^{(k)}(\varepsilon_n)
=
\left\{
G:
\DKL{\boldf_{\star}}{\boldf_{G,\star}}
\leq \varepsilon_n^2,\;
\VKLk{\boldf_{\star}}{\boldf_{G,\star}}
\leq \varepsilon_n^k
\right\}.
\]
The proof proceeds through the following steps. \\

\noindent \textbf{Step 1: Represent $f_{G,\star,\sigma_i^2}$ and $f_{\theta,\sigma_i^2}$ using the prior parameters.}
Since
$
G=Q * \mathrm{N}(0,A)$
we may represent a latent draw $U\sim G$ as
$
U=X+\xi,
$
where
$
X\sim Q, \;
\xi\sim \mathrm{N}(0,A),
$
and $X$ and $\xi$ are independent. The partial-true marginal $f_{G,\star,\sigma_i^2} $ is the law of
\[
Z_i
=
m_{\star}(\sigma_i)
+
s_\star(\sigma_i)X
+
s_\star(\sigma_i)\xi
+
\varepsilon_i, \;
\varepsilon_i\sim \mathrm{N}(0,\sigma_i^2),
\]
with $\varepsilon_i$ independent of $(X,\xi)$. In particular,
$
Z_i\mid X=x
\sim
\mathrm{N}(a_{0i}(x),\omega_{0i}^2)$,
where
$
a_{0i}(x)
=
m_{\star}(\sigma_i)
+
s_\star(\sigma_i)x,
$
and
$
\omega_{0i}^2
=
\sigma_i^2+s_\star^2(\sigma_i)A.
$
It follows that
\[
f_{G,\star,\sigma_i^2} (z)
=
\int
\varphi(z-a_{0i}(x),\omega_{0i}^2)\,dQ(x).
\]
Similarly, under the full candidate $\theta=(G,\eta_1,\eta_2)$, we have \[
f_{\theta,\sigma_i^2}(z)
=
\int
\varphi(z-a_i(x),\omega_i^2)\,dQ(x).
\]
where $
a_i(x)
=
\eta_1(\sigma_i)
+
s_{\eta_2}(\sigma_i)x,
$
and
$
\omega_i^2
=
\sigma_i^2+s_{\eta_2}^2(\sigma_i)A.
$\\

\noindent \textbf{Step 2: Control the perturbation from $f_{G,\star,\sigma_i^2}$ to $f_{\theta,\sigma_i^2}$.}
Suppose
$
\|\eta_1-m_{\star}\|_\infty\le c\varepsilon_n $ and $
\|\eta_2-\eta_{2\star,n}\|_\infty\le c\varepsilon_n$. Then,
by the Lipschitz property of $\sqrt{\Lambda}$, we have
$\|s_{\eta_2}-s_\star\|_\infty \le \|s_{\eta_2}-s_{\eta_{2\star,n}}\|_\infty + \|s_{\eta_{2\star,n}}-s_\star\|_\infty \le C\varepsilon_n.$
We use the following elementary Gaussian comparison bound. Suppose we have constants
$
0<\underline\omega\le \omega,\omega_0\le \overline\omega<\infty
,\;
|a|+|a_0|\le C,
$. If we define
$
\delta=|a-a_0|+|\omega^2-\omega_0^2|
$, then
\[
\DKL{\mathrm{N}(a_0,\omega_0^2)}{\mathrm{N}(a,\omega^2)}
\le
C\delta^2, \; \; \left|
\log
\frac{\varphi(z-a_0,\omega_0^2)}
{\varphi(z-a,\omega^2)}
\right|
\le
C\delta(1+z^2).
\]
Since \(Q\) is supported on \([-R,R]\), we obtain , uniformly over \(i\) and \(x\),
\[
|a_i(x)-a_{0i}(x)|
\le
|\eta_1(\sigma_i)-\eta_{1\star}(\sigma_i)|
+
|s_{\eta_2}(\sigma_i)-s_\star(\sigma_i)|\,|x|
\le
C\varepsilon_n.
\]
Since \(A\le \bar A < \infty\), and \(s_{\eta_2},s_\star\) are uniformly bounded on the local
event $\{  s_{\eta_2} : \| s_{\eta_2} - s_{\star}  \|_{\infty} \leq C \varepsilon_n  \} $, we have
$
|\omega_i^2-\omega_{0i}^2|
=
A\left|s_{\eta_2}^2(\sigma_i)-s_\star^2(\sigma_i)\right|
\le
C\varepsilon_n.
$
Moreover $
\omega_i^2,\omega_{0i}^2\ge \sigma_i^2\ge \underline\sigma^2 $,
so the component variances are uniformly bounded away from zero. Note that they are also uniformly
bounded above on the local event. Therefore, on this event, uniformly over \(i\) and \(x\),
$\DKLInline{\mathrm{N}(a_{0i}(x),\omega_{0i}^2)}{\mathrm{N}(a_i(x),\omega_i^2)} \le C\varepsilon_n^2.$
Since $f_{G,\star,\sigma_i^2}$ and $f_{\theta,\sigma_i^2}$ are Gaussian mixtures with the same mixing distribution \(Q\), the preceding bound and the log-sum
inequality gives
$
\DKLInline{f_{G,\star,\sigma_i^2}}{f_{\theta,\sigma_i^2}}
\le
C\varepsilon_n^2
$. Define $L_{\theta,i}(z):=\log (  f_{G,\star,\sigma_i^2}(z) / f_{\theta,\sigma_i^2}(z))$. After integrating over $Q$, the Gaussian log-ratio bound given above yields $
|L_{\theta,i}(z)|
\le
C\varepsilon_n(1+z^2)$. Since $Q$ is compactly supported and \(A\le\bar A\), it follows that

\[
\int |L_{\theta,i}(z)|^k f_{\star,\sigma_i^2}(z)\,dz
\le
C\varepsilon_n^k, \;
\int |L_{\theta,i}(z)|^k f_{G,\star,\sigma_i^2}(z)\,\dd z
\le
C\varepsilon_n^k.
\]

\noindent \textbf{Step 3: Implied KL bounds.} Define the quantity $D_{\theta,i}(z):=\log(f_{\star,\sigma_i^2} (z) / f_{G,\star,\sigma_i^2} (z))$. Since
$ \log (f_{\star,\sigma_i^2}(z) / f_{\theta,\sigma_i^2}(z))
=
D_{\theta,i}(z)+L_{\theta,i}(z),
$
we can write
\[
\begin{aligned}
\DKL{f_{\star,\sigma_i^2}}{  f_{\theta,\sigma_i^2} }
&=
\DKL{f_{\star,\sigma_i^2}}{ f_{G,\star,\sigma_i^2}  }
+
\int L_{\theta,i}(z)f_{G,\star,\sigma_i^2}(z) \dd z \\ &\quad +\int L_{\theta,i}(z)(f_{\star,\sigma_i^2}(z)-f_{G,\star,\sigma_i^2} (z)) \dd z.
\end{aligned}
\]
The second term is $ \int L_{\theta,i} f_{G,\star,\sigma_i^2}  =      \DKLInline{  f_{G,\star,\sigma_i^2}    }{ f_{\theta,\sigma_i^2}    } \leq C \varepsilon_n^2 $. For the third term, Cauchy-Schwarz yields
\[
\begin{aligned}
&\left|
\int L_{\theta,i} ( f_{\star,\sigma_i^2}  -f_{G,\star,\sigma_i^2})
\right|\\
&\quad\le
\left\{
\int L_{\theta,i}^2\p{\sqrt{f_{\star,\sigma_i^2}}+\sqrt{f_{G,\star,\sigma_i^2}}}^2
\right\}^{1/2}
\left\{
\int\p{\sqrt{f_{\star,\sigma_i^2}}-\sqrt{f_{G,\star,\sigma_i^2}}}^2
\right\}^{1/2}.
\end{aligned}
\]
Note that
$
(\surd{f_{\star,\sigma_i^2}}+\surd{f_{G,\star,\sigma_i^2}})^2\le 2(f_{\star,\sigma_i^2}+f_{G,\star,\sigma_i^2})
$. The bounds in Step 2 also show
$
\int L_{\theta,i}^2  f_{\star,\sigma_i^2} \le C\varepsilon_n^2, \;
\int L_{\theta,i}^2 f_{G,\star,\sigma_i^2}
  \le C\varepsilon_n^2$. It follows that
\[
\left|
\int L_{\theta,i} (f_{\star,\sigma_i^2}-f_{G,\star,\sigma_i^2})
\right|
\le
C\varepsilon_n \Dhel(f_{\star,\sigma_i^2},f_{G,\star,\sigma_i^2}).
\]
Averaging over $i$ and applying Cauchy--Schwarz,
\[
\frac1n\sum_{i=1}^n
\left|
\int L_{\theta,i}(f_{\star,\sigma_i^2}-f_{G,\star,\sigma_i^2})
\right|
\le
C\varepsilon_n
\left\{
\frac1n\sum_{i=1}^n \Dhel^2(f_{\star,\sigma_i^2}, f_{G,\star,\sigma_i^2})
\right\}^{1/2}.
\]
Since $\Dhel^2(f_{\star,\sigma_i^2},f_{G,\star,\sigma_i^2})\le \DKL {f_{\star,\sigma_i^2}}{f_{G,\star,\sigma_i^2}} $ and \(G\in\mathcal G_{}^{(k)}(c\varepsilon_n) \), we obtain
\[
 n^{-1} \sum_{i=1}^n \Dhel^2(f_{\star,\sigma_i^2},f_{G,\star,\sigma_i^2}
)
\le \DKL{f_{\star,\sigma_i^2}}{f_{G,\star,\sigma_i^2}}
\le
c^2\varepsilon_n^2.
\]
Thus,
$
 n^{-1} \sum_{i=1}^n
\left|
\int L_{\theta,i}(f_{\star,\sigma_i^2}-f_{G,\star,\sigma_i^2})
\right|
\le
C\varepsilon_n^2.$
Combining these bounds gives
$\DKL{\boldf_{\star}}{\boldf_{\theta}} \leq \DKL{\boldf_{\star}}{\boldf_{G,\star}} + C\varepsilon_n^2 \le C\varepsilon_n^2.$\\

\noindent \textbf{Step 4: Higher-order log-variation control.}
With $ L_{\theta,i} $ and $ D_{\theta,i}(z) $ as defined in the previous steps, observe that
\[
\log\frac{f_{\star,\sigma_i^2}}{f_{\theta,\sigma_i^2}}-\DKL{f_{\star,\sigma_i^2}}{f_{\theta,\sigma_i^2}}
=
\left(D_{\theta,i}-\DKL{f_{\star,\sigma_i^2}}{f_{G,\star,\sigma_i^2}}\right)
+
\left(L_{\theta,i}-\int L_{\theta,i}  f_{\star,\sigma_i^2} \right).
\]
Using $
|x+y|^k\le 2^{k-1}(|x|^k+|y|^k),
$
we get
\begin{align*}
& \int f_{\star,\sigma_i^2}
\left|
\log\frac{f_{\star,\sigma_i^2}}{ f_{\theta,\sigma_i^2}}-\DKL{f_{\star,\sigma_i^2}}{ f_{\theta,\sigma_i^2}}
\right|^k \\
&\quad \le
C
\int f_{\star,\sigma_i^2}|D_{\theta,i}-\DKL{f_{\star,\sigma_i^2}}{f_{G,\star,\sigma_i^2}}|^k\\
&\qquad+
C
\int f_{\star,\sigma_i^2}
\left|
L_{\theta,i}-\int L_{\theta,i} f_{\star,\sigma_i^2}
\right|^k.
\end{align*}
For the second term, Jensen's inequality gives
\[
\int f_{\star,\sigma_i^2}
\left|
L_{\theta,i}-\int L_{\theta,i} f_{\star,\sigma_i^2}
\right|^k
\le
C\int |L_{\theta,i}|^k f_{\star,\sigma_i^2}
\le
C\varepsilon_n^k,
\]
where the last inequality is from Step 2. Hence
\[
\int f_{\star,\sigma_i^2}
\left|
\log\frac{f_{\star,\sigma_i^2}}{f_{\theta,\sigma_i^2}}-\DKL{f_{\star,\sigma_i^2}}{f_{\theta,\sigma_i^2}}
\right|^k
\le
C
\int f_{\star,\sigma_i^2}|D_{\theta,i}-\DKL{f_{\star,\sigma_i^2}}{f_{G,\star,\sigma_i^2}}|^k
+
C\varepsilon_n^k.
\]
Averaging over \(i\) gives
$ \VKLk{\boldf_{\star}}{\boldf_{\theta}} \leq C\VKLk{ \boldf_{\star} }{\boldf_{G,\star}} + C \varepsilon_n^k.
$
Since \(G\in\mathcal G_{}^{(k)}(c\varepsilon_n) \), we obtain the bound $ \VKLk{\boldf_{\star}}{\boldf_{\theta}} \leq C \varepsilon_n^k .  $ \\

\noindent \textbf{Step 5: Small ball probability:} The preceding steps show that for a sufficiently small constant $c > 0$, we have the inclusion
$$
B_k(\boldf_{\star},    \varepsilon_n)
\supseteq
\left\{
(G,\eta_1,\eta_2):
G\in\mathcal G^{(k)}(c\varepsilon_n),\;
\|\eta_1-m_{\star}\|_\infty\le c\varepsilon_n,\;
\|\eta_2-\eta_{2\star,n}\|_\infty\le c\varepsilon_n
\right\}.
$$
In particular, it suffices to lower bound the prior probability of the event on the right. As the priors are independent and the event has a product structure, we can focus on the individual events for $(G,\eta_1,\eta_2)$.
By Lemma \ref{lem:G-small-k} and our Prior Specification \ref{loc-scale-prior}, there exists $c_1,c_2 > 0$ such that $\Pi \big( G : G \in \mathcal{G}^k(c \epsilon_n) \big) \geq c_1 \exp(- c_2 \log^2n) \geq c_1\exp( - c_2n \varepsilon_n^2) .$
 Since Prior Specification \ref{loc-scale-prior} imposes independent Matérn GP priors $\Pi_{\eta}$ with regularity $\alpha>1/2$ on $\eta_i$, $i=1,2$, their small-ball probabilities can be determined by their concentration function: $$ \phi_{g}(\varepsilon) = \inf_{h \in H^{\alpha+1/2}: \| h-g  \|_{\infty} \leq \varepsilon} \frac 12 \| h \|_{H^{\alpha+1/2}}^2 - \log  \Pi_{\eta} \big(  \| \eta \|_{\infty} < \varepsilon   \big) . $$
Note that $H^{\alpha+1/2}$ is the RKHS of this Gaussian process. For $m_{\star} \in \Sigma^{\alpha} \cap H^{\alpha}$. It follows from Lemma 11.36-11.37 in \cite{ghosal2017fundamentals} that $\phi_{m_{\star}} (\epsilon) \lessapprox \varepsilon^{-1/\alpha}$ for all $\epsilon > 0$ sufficiently small and  $$ \Pi_{\eta} \big( \eta : \|  \eta - m_{\star}  \|_{\infty} \leq c \varepsilon_n   \big) \geq c_1 \exp(- c_2 \varepsilon_n^{-1/\alpha}) $$
 for universal constants $c_1,c_2 > 0$. Since $\varepsilon_n^{-1/\alpha} \asymp n \varepsilon_n^2$, this lower bound is of the right order. For the other nuisance, we have by assumption, $\eta_{2 \star,n} \in H^{\alpha+1/2} $ and $\| \eta_{2 \star,n} \|_{H^{\alpha+1/2}} \lessapprox \sqrt{n} \varepsilon_n$. By choosing $h = \eta_{2 \star,n}$ in the infimum, Lemma 11.36-11.37 in \cite{ghosal2017fundamentals} implies that $\Pi_{\eta} \big( \eta : \| \eta - \eta_{2 \star,n} \|_{\infty} \leq c \varepsilon_n \big) \geq c_1 \exp(- c_2 n \varepsilon_n^2)$
 for universal constants $c_1,c_2 > 0$. The claim follows from combining the preceding bounds.

\end{proof}

\begin{proof}[Proof of Theorem~\ref{thm:close_contraction}]
 From the results in Chapter 11 of \cite{ghosal2017fundamentals}, the RKHS of the Matérn prior with regularity $\alpha$  coincides (with equivalent norms) with the Sobolev space of order $ \alpha + 1/2$. Moreover, the prior can be viewed as a Gaussian random
element in the Sobolev space $H^{\beta}$ for any order $\beta < \alpha$. Fix any  $1/2 <\beta < \alpha$ and let $\delta_n \asymp (n \varepsilon_n^2)^{-(\alpha - \beta)}$. For a constant $C > 0$ (to be specified below), and $H_{1}^{\beta}$ the unit ball of $H^{\beta}$, we define the sieve set
\begin{equation}
\mathcal{H}_n = \big \{   \delta_n H_1^{\beta} + C  \sqrt{n} \varepsilon_n  H_1^{\alpha+1/2}  \big \}.
\label{eq:eta_sieve}
\end{equation}
Let $\Pi_{\mathrm{GP}}$ denote the law of the Matérn Gaussian process prior, and let
$W\sim \Pi_{\mathrm{GP}}$. By the Borell-Sudakov inequality \citep[Proposition 11.17]{ghosal2017fundamentals} applied to the Gaussian random element
$W$ in the Banach space $H^\beta$, we have
$$
\begin{aligned}
    \Pi(\mathcal H_n^c)
    &\leq
    1-\Phi\left[
        \Phi^{-1}\left\{
            \Pi_{\mathrm{GP}}\big(\|W\|_{H^\beta}\leq \delta_n\big)
        \right\}
        +
        C\sqrt n\,\varepsilon_n
    \right].
\end{aligned}
$$
From the results in  Section 2.7.2 of \cite{van1996weak}, we have
$$
    \log N\big(
        H_1^{\alpha+1/2},\,
        \|\cdot\|_{H^\beta},\,
        \eta
    \big)\,
    \lesssim\,
    \eta^{-1/(\alpha+1/2-\beta)}
$$
for all sufficiently small $\eta > 0$. By Theorem 1.2 of \cite{li1999approximation}, $- \log \Pi_{\mathrm{GP}}(\|W\|_{H^\beta}\leq \delta) $ is bounded above by a constant multiple of \smash{$\delta^{-1/(\alpha - \beta)}$} for all sufficiently small $\delta$. Using the Gaussian quantile bound
$
    \Phi^{-1}(y)\geq -\sqrt{2\log(1/y)}
$,
we get, for another constant \(C'>0\),
$$
    \Pi(\mathcal H_n^c)
    \leq
    1-\Phi\left[
        C\sqrt n\,\varepsilon_n
        -
        C'\delta_n^{-1/\{2(\alpha-\beta)\}}
    \right].
$$
Since $
    \delta_n \asymp (n\varepsilon_n^2)^{-(\alpha-\beta)} $,
we have $C'\delta_n^{-1/\{2(\alpha-\beta)\}}\leq K\sqrt n\,\varepsilon_n$
for some constant \(K>0\).
Substituting this bound into the previous display gives
$
    \Pi(\mathcal H_n^c)
    \leq
    1-\Phi\left[
        (C-K)\sqrt n\,\varepsilon_n
    \right].
$
For \(C>K\), the argument of the Gaussian tail is positive. Hence, using
\(1-\Phi(x)\leq e^{-x^2/2}\) for \(x>0\), we obtain
$\Pi(\mathcal H_n^c)\leq
\exp\{-(C-K)^2n\varepsilon_n^2/2\}$.
Since $ L = (C-K)^2/2 $ can be made as large as desired by picking $C$ sufficiently large, we obtain
$
    \Pi(\mathcal H_n^c)
    \leq
    \exp\{-L n\varepsilon_n^2\} $
with \(L\) as large as desired. Define

\begin{equation}
\mathcal G
=
\left\{
G=Q * \mathrm{N}(0,A)\;:\;
Q\in\mathcal P([-R,R]),
\;
0\le A\le \overline A
\right\}
\label{eq:mathcal_G_sieve}
\end{equation}
where $[-R,R]$ is the support of the base measure $H$. Viewing $\theta = (G,\eta_1,\eta_2)  $ as the parameter, we define the produce sieve
\begin{equation}
\Theta_n = \mathcal{G} \otimes \mathcal{H}_n \otimes \mathcal{H}_n.
\label{eq:close-sieve}
\end{equation}
Note that, since $\beta > 1/2$, the Sobolev embedding theorem implies that  $ \| \eta \|_{\infty} \leq C \sqrt{n} \epsilon_n $ for every $\eta \in \mathcal{H}_n$. To prove the theorem, we verify the conditions of Theorem 8.23 in \cite{ghosal2017fundamentals} using $\Theta_n$ as the candidate sieve. The proof proceeds through several steps which we outline below.  \\

\noindent \textbf{Step 1: Lipschitz dependence on the nuisance functions.} We fix $G\in\mathcal G$ and compare the parameters $
\theta=(G,\eta_1,\eta_2)
$ and $
\theta'=(G,\eta_1',\eta_2') $. By convexity of squared Hellinger distance under mixtures,
$$
\Dhel^2(f_{\theta,\sigma^2},f_{\theta',\sigma^2})
\le
\int
\Dhel^2
\left(
\mathrm{N}(\eta_1(\sigma)+s_{\eta_2}(\sigma)u,\,\sigma^2),
\mathrm{N}(\eta_1'(\sigma)+s_{\eta_2'}(\sigma)u,\,\sigma^2)
\right)
\,\dd G(u).
$$
For equal-variance normals, note that
$
\Dhel^2\{\mathrm{N}(a,\sigma^2),\mathrm{N}(b,\sigma^2)\}
\le
\frac{(a-b)^2}{8\sigma^2}.
$
Thus
\[
\Dhel^2(f_{\theta,\sigma_i^2},f_{\theta',\sigma_i^2})
\le
C
\int
\left[
\{\eta_1(\sigma_i) - \eta'_1(\sigma_i)\} + \{s_{\eta_2}(\sigma_i) - s_{\eta'_2}(\sigma_i)\}u
\right]^2
\dd G(u).
\]
Because every $G \in \mathcal{G}$ satisfies $
\int u^2\,\dd G(u)\le C$ for another constant $C >0$, we obtain \[ \sup_{i}
\Dhel^2(f_{\theta,\sigma_i^2},f_{\theta',\sigma_i^2})
\le
C
\left[
\|\eta_1-\eta_1'\|_\infty^2
+
\|s_{\eta_2}-s_{\eta_2'}\|_\infty^2
\right].
\]
Since $s$ is Lipschitz, averaging over $i$ gives
$$
\Dhel(\boldf_{G,\eta_1,\eta_2},\;\boldf_{G,\eta_1',\eta_2'})
\le
C\|\eta_1-\eta_1'\|_\infty
+
C\|\eta_2-\eta_2'\|_\infty.
$$

\noindent \textbf{Step 2: Entropy of the DPGM.} For fixed \(\eta_2\), remove set the common location $\eta_1 \equiv 0$, yielding the marginal density,
$$
f_{(G,0,\eta_2),\sigma_i^2}(z)
=
\int
\varphi\left(
z-s_{\eta_2}(\sigma_i)u;\sigma_i^2
\right)
\,\dd G(u).
$$
If \(G=Q*\mathrm{N}(0,A)\), then we may write \(U=X+\xi\), where $
X\sim Q,
\xi\sim \mathrm{N}(0,A)$,
and \(X\) and \(\xi\) are independent. Integrating out \(\xi\), we get
$$
f_{(G,0,\eta_2),\sigma_i^2}(z)
=
\int
\varphi\left(
z-s_{\eta_2}(\sigma_i)x;\,
\sigma_i^2+s_{\eta_2}^2(\sigma_i)A
\right)
\,\dd Q(x).
$$
For fixed \(\eta_2\), define the metric
$$
d_{\eta_2}(G,G')
:=
\left\{
n^{-1} \sum_{i=1}^n
\Dhel^2(
f_{(G,0,\eta_2),\sigma_i^2},
f_{(G',0,\eta_2),\sigma_i^2})
\right\}^{1/2}.
$$

Let
$
S_n
=
1\vee
\sup_{\eta\in\mathcal H_n}
\sup_{\sigma}
s_\eta(\sigma)
$. Since $\sqrt{\Lambda(\cdot)}$ is globally Lipschitz, we have the bound $ S_n  \leq D \sqrt{n} \varepsilon_n $.  For \(\eta_2\in\mathcal H_n\), write
$
s_i=s_{\eta_2}(\sigma_i).
$ If \(s_i=0\), then $f_{(G,0,\eta_2),\sigma_i^2}$ does not depend on \(G\), and hence the
corresponding Hellinger distance is zero. If $s_i>0$, then by the change of variables
\(y=z/s_i\), we have
$$
\Dhel(
f_{(G,0,\eta_2),\sigma_i^2},
f_{(G',0,\eta_2),\sigma_i^2})
=
\Dhel\left(
f_{G,\tau_i^2},
f_{G',\tau_i^2}
\right),$$
where $\tau_i = \sigma_i / s_i$ and
$
f_{G,\tau^2}(y)
=
\int \varphi(y-u;\tau^2)\,\dd G(u).
$
Since \(\sigma_i\ge\underline\sigma\) and \(s_i\le S_n\), we have
$
\tau_i\ge \tau_n := \underline\sigma / S_n   $. For \(\tau_i\ge \tau_n\), note that we can write
$f_{G,\tau_i^2}=f_{G,\tau_n^2}* \mathrm{N}(0,\tau_i^2-\tau_n^2)$.
As the Hellinger contracts under convolution, we have
$
\Dhel(
f_{G,\tau_i^2},
f_{G',\tau_i^2})
\leq
\Dhel(
f_{G,\tau_n^2},
f_{G',\tau_n^2}).$
It follows that, uniformly over \(\eta_2\in\mathcal H_n\),
$
d_{\eta_2}(G,G')
\le
\Dhel(
f_{G,\tau_n^2},
f_{G',\tau_n^2}).$ In particular, it suffices to bound the entropy of the class
\[ \mathcal{P} =
\left\{
f_{Q,\sigma^2} :  \;
Q\in\mathcal P([-R,R]),\;
\tau_n^2\le \sigma^2\le \overline A + \tau_n^2
\right\}.
\]
By Lemma 3 of \cite{ghosal2007posterior},
$\log N(\varepsilon_n,\mathcal P,\Dhel)\le CR \tau_n^{-1}\log^2n$.
Since $\tau_n = \underline\sigma/ S_n$ and noting the bound on $S_n$ above, we obtain
$\log N(\varepsilon_n,\mathcal P,\Dhel)\le C\sqrt n\varepsilon_n\log^2 n$.\\

\noindent \textbf{Step 3: Entropy of the GP.}
For $\mathcal H_n
=
\delta_nH_1^\beta
+
M\sqrt n\,\varepsilon_n H_1^{\alpha+1/2}$, we can apply separate coverings for the two pieces of the set. By Section 2.7.2 of \cite{van1996weak}, we have \[
\log N(\tau,\mathcal H_n,\|\cdot\|_\infty) \leq C \bigg[
\left(
\frac{\delta_n}{\tau}
\right)^{1/\beta}
+
\left(
\frac{\sqrt n\,\varepsilon_n}{\tau}
\right)^{1/(\alpha+1/2)} \bigg]
\]
for all sufficiently small $\tau > 0$. At $\tau = c \varepsilon_n$ for some $c > 0$ and $\delta_n \asymp (n \varepsilon_n^2)^{-(\alpha - \beta)} \asymp \varepsilon_n^{(\alpha-\beta)/\alpha}$, we have  the bound \begin{align*}
    \left(
\frac{\delta_n}{\tau}
\right)^{1/\beta}
+
\left(
\frac{\sqrt n\,\varepsilon_n}{\tau}
\right)^{1/(\alpha+1/2)} \leq C \big[ \varepsilon_n^{-1/\alpha}   + n^{\frac{1}{2 \alpha+1}}   \big].
\end{align*}
Since $\varepsilon_n^{-1/\alpha} \asymp n \varepsilon_n^2 \asymp n^{\frac{1}{2 \alpha+1}} $, we obtain $
\log N(c\varepsilon_n,\mathcal H_n,\Norm{\cdot}_\infty) \leq C n \varepsilon_n^2. $ \\

\noindent \textbf{Step 4: Overall entropy.}
We now construct a product cover of  $\Theta_n = \mathcal{G} \otimes \mathcal{H}_n \otimes \mathcal{H}_n$.  Let $
\{\eta_{1,k}:1\le k\le N_1\} $
be a \(c\epsilon_n\)-net of \(\mathcal H_n\) in sup norm, and let $
\{\eta_{2,\ell}:1\le \ell\le N_2\} $
be another \(c\epsilon_n\)-net of \(\mathcal H_n\) in sup norm. Here \(c>0\) is a sufficiently
small fixed constant. For each fixed \(\eta_{2,\ell}\), let
$
\{G_{\ell,m}:1\le m\le N_G\}
$
be a \(c\epsilon_n\)-net of \(\mathcal G\) under the metric \(d_{\eta_{2,\ell}}\), defined in Step 2. The bounds in Step 2 show that these nets may be chosen so that $
\log N_G \le C n\epsilon_n^2 $
uniformly in \(\ell\).

We claim that the product collection
\[
\left\{ (G_{\ell,m},\eta_{1,k},\eta_{2,\ell}) : 1\le k\le N_1,\; 1\le \ell\le N_2,\; 1\le m\le N_G \right\}
\]
is a \(C\epsilon_n\)-net of \(\Theta_n\) under the average marginal Hellinger distance \(\Dhel\). To see this, fix any
$
\theta=(G,\eta_1,\eta_2)\in\Theta_n.
$
Choose \(\eta_{1,k}\) and \(\eta_{2,\ell}\) such that
$
\|\eta_1-\eta_{1,k}\|_\infty\le c\epsilon_n $ and $
\|\eta_2-\eta_{2,\ell}\|_\infty\le c\epsilon_n$.
By Step 1, after absorbing the Lipschitz constant of the link into the generic constant, we obtain
$
\Dhel(\boldf_{G,\eta_1,\eta_2},\;\boldf_{G,\eta_{1,k},\eta_{2,l}})
\le
C\varepsilon_n.
$
Now choose \(G_{\ell,m}\) such that
$
d_{\eta_{2,\ell}}(G,G_{\ell,m})\le c\varepsilon_n.
$
Since \(\eta_{1,k}\) is a common translation in this comparison, we obtain
$ \Dhel(\boldf_{G,\eta_{1,k},\eta_{2,\ell}},\;\boldf_{G_{\ell,m},\eta_{1,k},\eta_{2,l}})
\le
c\varepsilon_n.$
Therefore, by the triangle inequality,
$\Dhel(\boldf_{G,\eta_{1},\eta_{2}},\;\boldf_{G_{\ell,m},\eta_{1,k},\eta_{2,l}}) \le C\varepsilon_n.$
Taking logarithms of the product cardinality yields
$\log N(C\varepsilon_n,\Theta_n,\Dhel) \le \sup_{\eta_2\in\mathcal H_n} \log N(c\varepsilon_n,\mathcal G,d_{\eta_2}) + 2\log N(c\varepsilon_n,\mathcal H_n,\|\cdot\|_\infty).$
By Step 3, $
\log N(c\varepsilon_n,\mathcal H_n,\|\cdot\|_\infty)
\le
C n\varepsilon_n^2.
 $
Together with the DPGM entropy bound above, this gives
$
\log N(C\varepsilon_n,\Theta_n,\Dhel)
\le
C n\varepsilon_n^2.$
Thus, the entropy condition of Theorem 8.23 in \cite{ghosal2017fundamentals} is satisfied.

For the sieve complement condition, since
the prior for \(G\) is supported on \(\mathcal G\), the only contribution to the complement
comes from the two nuisance-function sieves. Therefore
$\Pi(\Theta_n^c) \le \Pi(\eta_1\notin\mathcal H_n)+\Pi(\eta_2\notin\mathcal H_n).$
From the bound derived in the beginning of the proof, for \(M\) sufficiently large (in the definition of $\mathcal{H}_n)$, we have
$
\Pi(\eta_j\notin\mathcal H_n)
\le
\exp\{-L n\varepsilon_n^2\}$ for $j=1,2$,
where \(L\) can be made arbitrarily large. Hence
$
\Pi(\Theta_n^c)
\le
2\exp\{-L n\varepsilon_n^2\}
\le
\exp\{-L' n\varepsilon_n^2\}
$
for any desired fixed \(L'>0\), after increasing \(M\) if necessary. The small ball requirement with rate $\varepsilon_n$ is verified in Lemma \ref{lem-small-close}. Together with the preceding estimates, this verifies all the conditions of Theorem 8.23 in \cite{ghosal2017fundamentals} and posterior contraction follows.

The lower bound on the probability of contraction follows by noting that Lemma \ref{lem-small-close} verifies the small ball probability for any $k^{th}$ order ball with $k \geq 2$. By Lemma 8.21 in \cite{ghosal2017fundamentals}, the corresponding marginal-likelihood lower bound fails with probability of order $(n\varepsilon_n^2)^{-k/2}$, while the testing and sieve-complement errors are exponentially small in $n\varepsilon_n^2$. Since $n\varepsilon_n^2 = n^{1/(2\alpha+1)}$, choosing any $k > 2q(2\alpha+1)$ ensures that $(n\varepsilon_n^2)^{-k/2} \leq n^{-q}$.

\end{proof}

\begin{lemm}[Small-ball with fixed nuisances]
\label{lem:G-small-k}
Consider the marginal density $\boldf_{\star}$ that arises from $(m_{\star},s_{\star},G_{\star})$. The marginal density obtained by varying $G$ while fixing $(m_{\star},s_{\star})$ is denoted by $\boldf_{G,\star}$. Suppose the prior on $
G=Q* \mathrm{N}(0,A)$ is $  \Pi =\mathrm{DP}(\alpha,H) \otimes \Gamma$ for some $\alpha > 0 $, base measure $H$ and variance prior $\Gamma$. Assume the following: \begin{enumerate}
    \item  $G_{\star}$ has support contained in $[-M,M]$ for some $M > 0$ and $H$ has density that is bounded below on  $[-M,M]$.
    \item For all sufficiently small $\delta > 0$,  $
\Gamma(A\le \delta)
\ge
c_1\delta^r \exp(-c_2 \delta^{- \kappa}) $ for some constants $c_1,c_2 > 0$ and $\kappa,r \geq 0$.
\item  The variances satisfy $0 < \underline\sigma \leq \sigma_i \leq \overline\sigma < \infty$ for all $i$. Moreover, $s_{\star}(\sigma)$ and $m_{\star}(\sigma)$ are uniformly bounded as functions of  $ (\sigma_1, \dots , \sigma_n)$.

\end{enumerate}
For any integer $k \geq 2$, define
$
\mathcal G^{(k)}(\varepsilon)
:=
\{
G\,:\,
\DKL{\boldf_{\star}}{\boldf_{G,\star}}\le \varepsilon^2,\;\;
\VKLk{\boldf_{\star}}{\boldf_{G,\star}}\leq \varepsilon^k\}.
$
Then, for all sufficiently small $\varepsilon > 0 $, there exists a set $\mathcal{K}(\varepsilon)$ in the $(Q,A)$ parameter space that satisfies: \begin{enumerate}
    \item  $ (Q,A) \in \mathcal{K}(\varepsilon)  $ implies $G = Q * \mathrm{N}(0,A) \in \mathcal{G}^{(k)}(\varepsilon)$
\item  $Q([-M-1,M+1]^c) \leq \varepsilon^{L_0} $ \item  $\Pi (\mathcal{K}(\varepsilon)) \geq a_3 \exp\!\big(- a_4 \log^2 (\varepsilon^{-1})\big)\,
\varepsilon^r \log^{-r}(\varepsilon^{-1})
\exp\!\big(- a_5 \varepsilon^{- \kappa} \log^{\kappa}(\varepsilon^{-1})\big)$
    \end{enumerate}
where $a_3,a_4,a_5, L_0 > 0$ are universal constants.

\end{lemm}

\begin{proof}
Define
$
m_i=m_{\star}(\sigma_i),s_i=s_\star(\sigma_i), \lambda_i=\frac{s_i}{\sigma_i}.
$
Since \(s_\star\) is bounded above and \(\sigma_i\ge \underline\sigma>0\), there exists
\(\Lambda<\infty\) such that
$
0\le \lambda_i\le \Lambda
$
for all \(i\). Denote the true  marginal by
$f_{\star, i}(z) = \int \varphi(z-m_i-s_i u,\sigma_i^2)\,\dd G_\star(u).$
Under the change of variables
$ y = \sigma_i^{-1} (z- m_i),
$
we have
$\sigma_i f_{\star, i}(m_i+\sigma_i y) = \int \varphi(y-\lambda_i u)\, \dd G_\star(u),$
where \(\varphi\) is the standard normal density. For \(\lambda\in[0,\Lambda]\), define
$h_{\lambda,G}(y) = \int \varphi(y-\lambda u)\, \dd G(u).$
We denote the transformed true marginal by \(h_{\lambda_i,G_\star}\). Let
$
\bar m=\sup_i |m_i|<\infty.
$
Choose \(B_z>0\), to be fixed sufficiently large below, and define
$
S_{\varepsilon}=B_z\sqrt{\log(\varepsilon^{-1})}.
$
Define
$
T_{\varepsilon}=\underline\sigma^{-1}(S_{\varepsilon}+\bar m).
$
Then \(|z|\le S_{\varepsilon}\) implies
$ \sigma_i^{-1} \left|   z- m_i  \right|
\le T_{\varepsilon}
$
for every \(i\). Also, for some fixed constant \(B_y<\infty\), we have that
$T_{\varepsilon}\le B_y\sqrt{\log(\varepsilon^{-1})}.$\\

\noindent \textbf{Step 1: Uniform finite approximation of \(G_\star\).}
Let \(q_{\varepsilon}\ge1\) be an integer to be chosen below. Since \(G_\star\) is supported on
\([-M,M]\), Lemma A.1 of \cite{ghosal2001entropies} implies that there exists a discrete
measure \(\widetilde G\), supported on \([-M,M]\), with at most \(2q_{\varepsilon}+2\) support points
such that
\[
\int z^\ell\,\dd \widetilde G(z)
=
\int z^\ell\,\dd G_\star(z),
\qquad
\ell=0,1,\ldots,2q_{\varepsilon}.
\]
For fixed \(\lambda\in[0,\Lambda]\) and \(|y|\le T_{\varepsilon}\), observe that
\[
\varphi(y-\lambda u)
=
\varphi(y)
\exp\left\{
\lambda yu-\frac{\lambda^2u^2}{2}
\right\}.
\]
Define
$
t_{\lambda,y}(u)
=
\lambda yu-\frac{\lambda^2u^2}{2}.
$
For \(u\in[-M,M]\), \(\lambda\in[0,\Lambda]\), and \(|y|\le T_{\varepsilon}\), we have
$
|t_{\lambda,y}(u)|\le C(1+T_{\varepsilon}).
$
Taylor expanding the exponential gives
\[
e^{t_{\lambda,y}(u)}
=
\sum_{\ell=0}^{q_{\varepsilon}}
\frac{t_{\lambda,y}(u)^\ell}{\ell!}
+
R_{q_{\varepsilon}+1}\{t_{\lambda,y}(u)\}.
\]
Each \(t_{\lambda,y}(u)^\ell\) is a polynomial in \(u\) of degree at most \(2\ell\).
For \(\ell\le q_{\varepsilon}\), this degree is at most \(2q_{\varepsilon}\). Hence the polynomial part integrates
identically under \(G_\star\) and \(\widetilde G\). Thus the approximation error is only the
Taylor remainder. Since
$
|R_{q_{\varepsilon}+1}(t)|
\le
e^{|t|}
|t|^{q_{\varepsilon}+1}/(q_{\varepsilon}+1)!,
$
we obtain
\[
\sup_{\lambda\in[0,\Lambda]}
\sup_{|y|\le T_{\varepsilon}}
\left|
h_{\lambda,G_\star}(y)-h_{\lambda,\widetilde G}(y)
\right|
\le
C e^{ C(1+T_{\varepsilon})}
\frac{\{C(1+T_{\varepsilon})\}^{q_{\varepsilon}+1}}{(q_{\varepsilon}+1)!}.
\]
Fix \(L>1\) sufficiently large, to be chosen below. Taking $
q_{\varepsilon}=\lceil R\log(\varepsilon^{-1})\rceil $
with $R>0$ sufficiently large gives a discrete measure
\(\widetilde G\) satisfying
\[
\sup_{\lambda\in[0,\Lambda]}
\sup_{|y|\le T_{\varepsilon}}
\left|
h_{\lambda,G_\star}(y)-h_{\lambda,\widetilde G}(y)
\right|
\le
\varepsilon^L.
\]
Next, for \(u\in[-M,M]\), \(\lambda\in[0,\Lambda]\), and \(|y|\le T_{\varepsilon}\),
$
|y-\lambda u|\le T_{\varepsilon}+\Lambda M.
$
Therefore, for constants \(c,C'>0\), we have that
$h_{\lambda,G_\star}(y) = \int\varphi(y-\lambda u)\,\dd G_\star(u) \ge c\exp\{-C'T_{\varepsilon}^2\}.$
Since \(T_{\varepsilon}^2\le B_y^2\log(\varepsilon^{-1})\), this implies
$
h_{\lambda,G_\star}(y)
\ge
c\varepsilon^{C'B_y^2}.
$
With any $
L>C'B_y^2+1$, it follows that
$$
\sup_{\lambda\in[0,\Lambda]}
\sup_{|y|\le T_{\varepsilon}}
\left|
\frac{h_{\lambda,\widetilde G}(y)}
{h_{\lambda,G_\star}(y)}
-1
\right|
\le
C\varepsilon.
$$
Returning to the \(z\)-scale, define
\[
 f_{\widetilde{G},i}(z)
=
\int
\varphi(z-m_i-s_i u,\sigma_i^2) \dd \widetilde G(u).
\]
Since \(|z|\le S_{\varepsilon}\) implies \(|(z-m_i)/\sigma_i|\le T_{\varepsilon}\), we obtain
\[
\sup_i
\sup_{|z|\le S_{\varepsilon}}
\left|
\frac{f_{\widetilde{G},i}(z)}
{f_{\star, i}(z)}
-1
\right|
\le
C\varepsilon.
\]

\noindent \textbf{Step 2: Stability under small perturbations of the mixing measure.}
Write the discrete measure as
$
\widetilde G
=
\sum_{j=1}^{J_{\varepsilon}}p_j\delta_{u_j}$ for support points $ u_j\in[-M,M]$ and  $
J_{\varepsilon}\le 2q_{\varepsilon}+2.
$
By slightly  perturbing the support points and using fewer points if necessary, we may assume that the support points are separated by
at least \(2\varepsilon^{L_0}\) for some \(L_0>L\) and that $U_j = \{ t : \left| t - u_j\right| \leq \varepsilon^{L_0}  \}  \subset [-M,M]$ for $j=1,\dots,J_{\varepsilon}$. In particular, by defining
$
U_0
=
\mathbb R\setminus \bigcup_{j=1}^{J_{\varepsilon}}U_j
$, the collection \(\{U_0,U_1,\ldots,U_{J_{\varepsilon}}\}\) is a disjoint partition of \(\mathbb R\). For any probability measure \(Q\), define
$h_{\lambda,Q}(y) = \int \varphi(y-\lambda t) \dd Q(t).$
We can write the difference in terms of:
\[
\begin{aligned}
&h_{\lambda,Q}(y)-h_{\lambda,\widetilde G}(y)\\
&\quad=
\int_{U_0}\varphi(y-\lambda t)\dd Q(t)
+
\sum_{j=1}^{J_{\varepsilon}}
\int_{U_j}
\left\{
\varphi(y-\lambda t)-\varphi(y-\lambda u_j)
\right\}
\dd Q(t)\\
&\qquad+
\sum_{j=1}^{J_{\varepsilon}}
\{Q(U_j)-p_j\}\varphi(y-\lambda u_j).
\end{aligned}
\]
Since \(\sum_{j=1}^{J_{\varepsilon}}p_j=1\), we have
$
Q(U_0)
=
1-\sum_{j=1}^{J_{\varepsilon}}Q(U_j)
\le
\sum_{j=1}^{J_{\varepsilon}}|Q(U_j)-p_j|.
$
Also,
\[
\sup_{\lambda\in[0,\Lambda]}\sup_{y,t\in\mathbb R}
\left|
\frac{\partial}{\partial t}\varphi(y-\lambda t)
\right|
\le
\Lambda\|\varphi'\|_\infty
<\infty.
\]
Therefore, we have the bound
\[
\sup_{\lambda\in[0,\Lambda]}
\sup_{y\in\mathbb R}
\left|
h_{\lambda,Q}(y)-h_{\lambda,\widetilde G}(y)
\right|
\le
C
\left[
\varepsilon^{L_0}
+
\sum_{j=1}^{J_{\varepsilon}}|Q(U_j)-p_j|
\right].
\]
In particular, if
$
\sum_{j=1}^{J_{\varepsilon}}|Q(U_j)-p_j|
\le
\varepsilon^{L_0},
$
then
\[
\sup_{\lambda\in[0,\Lambda]}
\sup_{y\in\mathbb R}
\left|
h_{\lambda,Q}(y)-h_{\lambda,\widetilde G}(y)
\right|
\le
C\varepsilon^{L_0}.
\]
Combining this with the preceding lower bound for \(h_{\lambda,G_\star}\), for
\(|y|\le T_{\varepsilon}\), since $L_0 > L$, we obtain
\[ \sup_{\lambda\in[0,\Lambda]}
\sup_{|y|\le T_{\varepsilon}}
\left|
\frac{h_{\lambda,Q}(y)}
{h_{\lambda,G_\star}(y)}
-1
\right|
\le
C\varepsilon
+
C\varepsilon^{L_0-C'B_y^2} \leq C \varepsilon.
\]
Define the same-variance marginal induced by \(Q\) as
$f_{Q,i}(z) = \int \varphi(z-m_i-s_i t,\sigma_i^2)\,\dd Q(t).$
Since $
\sigma_i f_{Q,i}(m_i+\sigma_i y)=h_{\lambda_i,Q}(y)$
and \(|z|\le S_{\varepsilon}\) implies \(|(z-m_i)/\sigma_i|\le T_{\varepsilon}\), we obtain
\[
\sup_i
\sup_{|z|\le S_{\varepsilon}}
\left|
\frac{f_{Q,i}(z)}
{f_{\star,i}(z)}
-1
\right|
\le
C\varepsilon.
\]
\noindent \textbf{Step 3: Incorporate the small Gaussian convolution \(A\).} Let $
G=Q* N(0,A) $
and suppose
$
A\le a_{\mathrm{var}}\varepsilon\log^{-1}(\varepsilon^{-1}),$
where \(a_{\mathrm{var}}>0\) is sufficiently small. Define the partial true marginal
$f_{G,i}(z) = \int \varphi(z-m_i-s_i t,\sigma_i^2+s_i^2A)\dd Q(t).$
Change of variables yields
$\sigma_i f_{G,i}(m_i+\sigma_i y) = \int \varphi(y-\lambda_i t,1+\lambda_i^2A) \dd Q(t).$
On the set \(U = \cup_{j=1}^{J_{\varepsilon}} U_j \), all \(t\)'s lie in
a fixed compact interval, say \([-M-1,M+1]\), for all large \(n\). Hence, for
\(|y|\le T_{\varepsilon}\), \(t\in U\), and \(\lambda_i\in[0,\Lambda]\),
\[
\left|
\log
\frac{\varphi(y-\lambda_i t,1+\lambda_i^2A)}
{\varphi(y-\lambda_i t,1)}
\right|
\le
C A(1+T_{\varepsilon}^2)
\le
C a_{\mathrm{var}}\varepsilon.
\]
Taking \(a_{\mathrm{var}}>0\) sufficiently small, the contribution from \(U\) changes by a
relative \(1+O(\varepsilon)\). The contribution from \(U_0\) is bounded in absolute value by
\( Q(U_0)\le C\varepsilon^{L_0}\). Dividing by the lower bound
\(h_{\lambda_i,G_\star}(y)\ge c\varepsilon^{C'B_y^2}\), and using
\(L_0>C'B_y^2+1\), this outside-cell contribution is also \(O(\varepsilon)\). Hence
\[
\sup_i
\sup_{|y|\le T_{\varepsilon}}
\left|
\frac{\sigma_i f_{G,i}(m_i+\sigma_i y)}
{h_{\lambda_i,G_\star}(y)}
-1
\right|
\le
C\varepsilon.
\]
Returning to the \(z\)-scale,
\[
\sup_i
\sup_{|z|\le S_{\varepsilon}}
\left|
\frac{f_{G,i}(z)}
{f_{\star,i}(z)}
-1
\right|
\le
C\varepsilon.
\]

\noindent \textbf{Step 4 KL and \(V_k\) control.}
Let
$
\ell_i(z)
:=
\log(f_{\star,i}(z)/f_{G,i}(z)).
$
By the preceding approximations, we obtain
$
\sup_i
\sup_{|z|\le S_{\varepsilon}}
|\ell_i(z)|
\le
C\varepsilon.
$ We first record a global envelope for \(\ell_i\). Let
$
U=\bigcup_{j=1}^{J_{\varepsilon}}U_j.
$
Since
$
Q(U_0)
\le
\varepsilon^{L_0}$, we have, for all sufficiently large \(n\),
$
Q(U)\ge 1-\varepsilon^{L_0}\ge \frac12.
$
Moreover, since the atoms \(u_j\) belong to \([-M,M]\) and the radii of the \(U_j\)'s are
\(\varepsilon^{L_0}\), we have
$
U\subset[-M-1,M+1]
$
for all sufficiently small $\varepsilon$.

Note that the true density \(f_{\star,i}\) is a Gaussian mixture with mixing support contained in
\([-M,M]\). In particular, the component means \(m_i+s_i u\) are contained in a compact interval, and the variances
\(\sigma_i^2\) are, by assumption, bounded away from $0$ and $\infty$. Therefore there exist constants
\(c_1,C_1>0\) such that
$f_{\star,i}(z)\ge c_1\exp\{-C_1(1+z^2)\}$
uniformly in \(i\) and \(z\). Similarly, on the event under consideration, the density \(f_{G,i}\) contains at least
\(Q(U)\ge 1/2\) mass over \(t\in U\subset[-M-1,M+1]\). For such \(t\), the component means
\(m_i+s_i t\) are contained in a common compact interval, and the component variances
$
\sigma_i^2+s_i^2A
$
are bounded away from $0$ and $\infty$. Hence there exist constants \(c_2,C_2>0\) such that
$f_{G,i}(z)\ge c_2\exp\{-C_2(1+z^2)\}$
uniformly in \(i\) and \(z\). Note that the functions \(f_{\star,i}\) and \(f_{G,i}\) are also uniformly bounded above. Combining these upper and lower
bounds gives
$
|\ell_i(z)|
\le
C(1+z^2),
$
uniformly in \(i\). Next we record the tail bounds. Since \(G_\star\) is compactly supported and
\(m_i,s_i,\sigma_i\) are uniformly bounded with \(\sigma_i\ge\underline\sigma>0\), the true
marginals \(f_{\star,i}\) have uniformly Gaussian tails. Hence, by choosing \(B_z>0\) sufficiently
large in
$
S_{\varepsilon}=B_z\sqrt{\log(\varepsilon^{-1})},
$
we have
\[
\sup_i
\int_{|z|>S_{\varepsilon}}
(1+z^2)^k f_{\star,i}(z) \dd z
\le
C\varepsilon^k \; , \;
\sup_i
\int_{|z|>S_{\varepsilon}}
(1+z^2)^2 f_{\star,i}(z)\dd z
\le
C\varepsilon^2.
\]
For the candidate density \(f_{G,i}\), split its mixing distribution into \(U = \cup_{j=1}^{J_{\varepsilon}} U_j \) and \(U_0\).
The contribution from \(U\) has uniformly Gaussian tails because \(U\subset[-M-1,M+1]\).
The contribution from \(U_0\) has total mass at most \(Q(U_0)\le\varepsilon^{L_0}\). Taking
\(L_0\ge 2\) and \(B_z\) sufficiently large gives
\[
\sup_{i} \int_{|z| > S_{\varepsilon}} f_{G,i} \dd z \leq C \varepsilon^2.
\]
We now control the KL divergence. On \(|z|\le S_{\varepsilon}\), write
$
\frac{f_{G,i}(z)}{f_{\star,i}(z)}
=
1+\Delta_i(z),
|\Delta_i(z)|\le C\varepsilon.
$
Then
$
\ell_i(z)
=
-\log(1+\Delta_i(z)).
$
Using the elementary inequality
$
-\log(1+x)\le -x+Cx^2
$
for \(|x|\) sufficiently small, we get
$
\ell_i(z)
\le
-\Delta_i(z)+C\Delta_i^2(z).
$ Therefore,
\[
\int_{|z|\le S_\varepsilon}\ell_i(z)f_{\star,i}(z) \dd z
\le
-\int_{|z|\le S_\varepsilon}\{f_{G,i}(z)-f_{\star,i}(z)\} \dd z
+
C\varepsilon^2.
\]
Since both \(f_{\star,i}\) and \(f_{G,i}\) integrate to one,
\[
-\int_{|z|\le S_{\varepsilon}}\{f_{G,i}(z)-f_{\star,i}(z)\} \dd z
=
\int_{|z|>S_{\varepsilon}}\{f_{G,i}(z)-f_{\star,i}(z)\} \dd z.
\]
Thus
\[
-\int_{|z|\le S_{\varepsilon}}\{f_{G,i}(z)-f_{\star,i}(z)\} \dd z
\le
\int_{|z|>S_{\varepsilon}}f_{G,i}(z) \dd z
+
\int_{|z|>S_{\varepsilon}}f_{\star,i}(z) \dd z
\le
C\varepsilon^2.
\]
The tail contribution to the KL integral is also \(O(\varepsilon^2)\):
\[
\int_{|z|>S_{\varepsilon}}|\ell_i(z)|f_{\star,i}(z) \dd z
\le
C\int_{|z|>S_{\varepsilon}}(1+z^2)f_{\star,i}(z) \dd z
\le
C\varepsilon^2.
\]
Combining the bounds yields
$
\DKL{f_{\star,i}}{f_{G,i}}
\le
C\varepsilon^2 $
uniformly in \(i\). Averaging over \(i\), we obtain
$
\DKL{\boldf_{\theta_\star}}{\boldf_{\theta_{G,\star}}}
\le
C\varepsilon^2
$. Next, let
$
K_i=\DKL{f_{\star,i}}{f_{G,i}}.
$ By Jensen's inequality,
\[
\int |\ell_i(z)-K_i|^k f_{\star,i}(z) \dd z
\le
C\int |\ell_i(z)|^k f_{\star,i}(z) \dd z.
\]
 On \(|z|\le S_{\varepsilon}\), we have $
|\ell_i(z)|\le C\varepsilon$
 and so
$
\int_{|z|\le S_{\varepsilon}}|\ell_i(z)|^k f_{\star,i}(z) \dd z
\le
C\varepsilon^k.
$
On the tail,
\[
\int_{|z|>S_{\varepsilon}}|\ell_i(z)|^k f_{\star,i}(z) \dd z
\le
C\int_{|z|>S_{\varepsilon}}(1+z^2)^k f_{\star,i}(z) \dd z
\le
C\varepsilon^k.
\]
Thus
$
\int |\ell_i(z)-K_i|^k f_{\star,i}(z) \dd z
\le
C\varepsilon^k
$
uniformly in \(i\). Averaging over \(i\), we obtain the desired bound
$
\VKLk{\boldf_{\theta_\star}}{\boldf_{\theta_{G,\star}}}
\le
C\varepsilon^k.
$\\

\noindent \textbf{Step 5: Lower bound the prior probability.}
The preceding steps show that if
$
\sum_{j=1}^{J_{\varepsilon}}|Q(U_j)-p_j|
\le
\varepsilon^{L_0} $ , $
A\le a_{\mathrm{var}}\varepsilon\log^{-1}(\varepsilon^{-1})$  for some sufficiently large $L_0 > 0$  and sufficiently small $a_{var} > 0$, we have $G = Q * \mathrm{N}(0,A) \in \mathcal G^{(k)}(\varepsilon)  $. It remains to lower bound the probability of the set $$ \mathcal{K}(\varepsilon) = \bigg \{   (Q,A)\, :\, \sum_{j=1}^{J_{\varepsilon}}|Q(U_j)-p_j|
\le
\varepsilon^{L_0} \;, \;    A\le a_{\mathrm{var}}\varepsilon\log^{-1}(\varepsilon^{-1})    \bigg \}. $$
As the priors on $(Q,A)$ are independent, it suffices to lower bound the events separately. Since the density of the base measure $H$ is bounded below on $[-M,M]$, we have $H(U_j) \geq c_0 \varepsilon^{L_0}$ for each $j$ and some universal constant $c_0 > 0$. Since $J_{\varepsilon} \leq 3R \log(\varepsilon^{-1})$,  Lemma G.13 in \cite{ghosal2017fundamentals} implies that $$ \Pi \big( Q : \sum_{j=1}^{J_{\varepsilon}}|Q(U_j)-p_j|
\le
\varepsilon^{L_0}    \big) \geq c' \exp(-C' \log^2(\varepsilon^{-1})) $$
for some constants $c,C' > 0$. The small  ball probability of $\{ A : A \le a_{\mathrm{var}}\varepsilon\log^{-1}(\varepsilon^{-1})    \} $ follows  from the assumption on the variance prior $\Gamma$. Finally, as noted in Step 2, the $Q$ probability of $[-M-1,M+1]^c  $ follows from $ [-M-1,M+1]^c \subset U_0$ and $Q(U_0) \leq  \sum_{j=1}^{J_{\varepsilon}}|Q(U_j)-p_j|   \leq \varepsilon^{L_0}$.

\end{proof}

\subsection{Lemmata for regret}

Let us use the notation, for $\theta=(G,\eta_1,\eta_2)$ and $\rho>0$, for a regularized Bayes rule
\begin{equation}
\delta_{\theta}^{\rho}(z, \sigma^2) := z \,+\, \sigma^2 \frac{ f'_{\theta,\sigma^2}(z)}{f_{\theta,\sigma^2}(z)\vee (\rho/\sigma)},
\label{eq:reg_bayes_rule_close}
\end{equation}
where $'$ denotes differentiation with respect to $z$.  In this section, for functions
$h: \RR \times [\ubar\sigma^2,\bar\sigma^2] \to \RR$, we write $\Norm{h}_{\infty}$ for the
supremum norm over the whole domain,
$$
\Norm{h}_{\infty} := \sup_{z \in \RR}\; \sup_{\sigma^2\in[\ubar\sigma^2,\bar\sigma^2]} \abs{h(z, \sigma^2)}.
$$
Let us recall the sieve $\Theta_n$ defined in \eqref{eq:close-sieve} and let $\varepsilon_n$ be as in
Theorem~\ref{thm:close_contraction}. The following result establishes a covering number for regularized Bayes rules
induced by parameters in the sieve $\Theta_n$ under the $\Norm{\cdot}_{\infty}$ norm.

\begin{lemm}[Covering regularized Bayes rules for CLOSE]
\label{lemm:covering_number_reg_bayes_close}
Let \(\alpha>1/2\) denote the Mat\'ern regularity in Specification~\ref{loc-scale-prior}, so that
\(\varepsilon_n=n^{-\alpha/(2\alpha+1)}\), and let \(1/2<\beta<\alpha\) be the exponent used in the
definition of the sieve \(\mathcal H_n\). There exist constants \(c,C>0\), depending only on the prior
specification, \(\alpha,\beta,M,\bar A,\ubar\sigma^2,\bar\sigma^2\), the link function, and the
constants in the definition of $\mathcal H_n$, but not on $\rho$ or $n$, such that, for every
$0<\rho<1/10$ and all sufficiently large $n$,
$\log N\p{c\varepsilon_n,\; \cb{\delta_\theta^\rho:\theta\in\Theta_n},\; \Norm{\cdot}_{\infty}} \le C n\varepsilon_n^2|\log(\rho\varepsilon_n)|^{3}.$
In particular, if \(\rho=\rho_n\ge n^{-q}\) for some fixed \(q<\infty\), then there exists
\(C_q<\infty\) such that
$\log N\p{c\varepsilon_n,\; \cb{\delta_\theta^{\rho_n}:\theta\in\Theta_n},\; \Norm{\cdot}_{\infty}} \le C_q n\varepsilon_n^2 \log^3 n$
for all sufficiently large \(n\).
\end{lemm}

\begin{proof}
All constants below may change from line to line but do not depend on \(\rho\) or \(n\).

\medskip
\noindent\textbf{Step 1: Pointwise Gaussian bounds.}
For every polynomial $P_r$ of degree at most $r$ and given a fixed constant $C_0>0$, there exists
another constant $C>0$ such that
$$
\abs{P_r(y)}
\min\cb{1,\frac{C_0}{\rho}\exp(-y^2/2)}
\le
C|\log\rho|^{r/2},
\;\text{ for all }\, \rho \in (0,1/10),\; y \in \RR.
$$
Indeed, if $|y|\le C_1|\log\rho|^{1/2}$ where $C_1$ is a fixed large constant, then
$\abs{P_r(y)}\le C_2 |\log \rho|^{r/2}$ for another constant $C_2$. If
$|y|> C_1 |\log \rho|^{1/2}$, then, taking $C_1$ large enough depending only on $C_0$,
$$
\frac{C_0}{\rho}\exp(-y^2/2)
=
\cb{\frac{C_0}{\rho}\exp(-y^2/4)}\exp(-y^2/4)
\le
\exp(-y^2/4),
$$
and hence
$
\abs{P_r(y)}
\min\cb{1,C_0 \rho^{-1}\exp(-y^2/2)}
\le
\abs{P_r(y)}\exp(-y^2/4)
\le C_2.
$
This proves the displayed polynomial bound.

Now let $v\ge \ubar\sigma^2$, $\sigma\in[\ubar\sigma,\bar\sigma]$, and write $y=(z-m)/\sqrt v$.
Since
$\partial_z^r\varphi(z-m;v) = v^{-r/2}P_r(y)\varphi(z-m;v), \; r=0,1,2,$
for certain polynomials $P_r$ of degree $r$, and since
$$
\frac{\varphi(z-m;v)}{\varphi(z-m;v)\vee \rho/\sigma}
=
\min\cb{1,\frac{\sigma}{\rho\sqrt{2\pi v}}\exp(-y^2/2)}
\le
\min\cb{1,\frac{C}{\rho}\exp(-y^2/2)},
$$
we get
\begin{equation}
\label{eq:close-reg-bayes-gaussian-z-bounds}
\frac{\abs{\partial_z^r\varphi(z-m;v)}}
{\varphi(z-m;v)\vee \rho/\sigma}
\le
C|\log\rho|^{r/2},
\quad r=0,1,2.
\end{equation}
Moreover, if \(v=\sigma^2+s^2A\), \(s\ge0\), $\sigma^2 \geq \ubar{\sigma}^2$ and \(A\in[0,\bar A]\),
then
$$
\frac{sA}{v}\le C,
\;\;\,
\frac{sA}{v^{3/2}}\le C,
$$
uniformly over \(s,A,\sigma\). Since
$\partial_v\varphi(z-m;v) = \frac{1}{2v}(y^2-1)\varphi(z-m;v) \;\;\text{ and }\;\; \partial_z\partial_v\varphi(z-m;v) = v^{-3/2}P_3(y)\varphi(z-m;v)$
for a polynomial \(P_3\) of degree \(3\), the same polynomial bound gives
\begin{equation}
\label{eq:close-reg-bayes-gaussian-v-bounds}
\frac{sA\abs{\partial_v\varphi(z-m;v)}}
{\varphi(z-m;v)\vee \rho/\sigma}
\le
C|\log\rho|,
\;\;\,
\frac{sA\abs{\partial_z\partial_v\varphi(z-m;v)}}
{\varphi(z-m;v)\vee \rho/\sigma}
\le
C|\log\rho|^{3/2}.
\end{equation}

\medskip
\noindent\textbf{Step 2: Lipschitz continuity in the nuisance functions.}
We next control the effect of changing the nuisance functions while keeping $G$ fixed. Write
$G=Q*\mathrm N(0,A)$ where $Q$ is supported on $[-R,R]$. Recall that
$$
f_{\theta,\sigma^2}(z)
=
\int
\varphi\!\p{
z-\eta_1(\sigma)-s_{\eta_2}(\sigma)x;\,\,
\sigma^2+s_{\eta_2}^2(\sigma)A
}
\,\dd Q(x).
$$
Let
$
\theta=(G,\eta_1,\eta_2),
$
$
\widetilde\theta=(G,\widetilde\eta_1,\widetilde\eta_2),
$
and for $t \in [0,1]$,
$$
\theta_t := (G, \eta_{1,t},\eta_{2,t}),\; \text{ where }\;   \eta_{j,t}:=(1-t)\eta_j+t\widetilde\eta_j,
\;\; j=1,2.
$$
Since \(\sqrt\Lambda\) is globally Lipschitz, the map \(t\mapsto s_{\eta_{2,t}}(\sigma)\) is absolutely
continuous and, for almost every \(t\),
$$
\abs{
\frac{\dd}{\dd t}s_{\eta_{2,t}}(\sigma)
}
\le
C\Norm{\eta_2-\widetilde\eta_2}_\infty .
$$
Combining this with \eqref{eq:close-reg-bayes-gaussian-z-bounds} and
\eqref{eq:close-reg-bayes-gaussian-v-bounds}, and using that $Q$ is supported on $[-R,R]$, gives,
for almost every \(t\),
$$
\abs{
\frac{\dd}{\dd t}
f_{\theta_t,\sigma^2}(z)
} \, \le \,
 C|\log\rho|
\p{
\Norm{\eta_1-\widetilde\eta_1}_\infty
+
\Norm{\eta_2-\widetilde\eta_2}_\infty
}
\cb{
f_{\theta_t,\sigma^2}(z)\vee \rho/\sigma
},
$$
and,
$$
\abs{
\frac{\dd}{\dd t}
f_{\theta_t,\sigma^2}'(z)
}
\;\le\;
C|\log\rho|^{3/2}
\p{
\Norm{\eta_1-\widetilde\eta_1}_\infty
+
\Norm{\eta_2-\widetilde\eta_2}_\infty
}
\cb{
f_{\theta_t,\sigma^2}(z)\vee \rho/\sigma
}.
$$
Here we also used the elementary inequality
$$
\int \{g_x(z)\vee \rho/\sigma\}\,\dd Q(x)
\le
\int g_x(z)\,\dd Q(x)+\rho/\sigma
\le
2\cb{\int g_x(z)\,\dd Q(x)\vee \rho/\sigma}.
$$
Similarly, \eqref{eq:close-reg-bayes-gaussian-z-bounds} with $r=1$ gives
$
\abs{
f_{\theta_t,\sigma^2}'(z)
}
\le
C|\log\rho|^{1/2}
\cb{
f_{\theta_t,\sigma^2}(z)\vee \rho/\sigma
}.
$
Therefore, for almost every $t$,
$$
\begin{aligned}
\abs{
\frac{\dd}{\dd t}
\left[
\frac{
f_{\theta_t,\sigma^2}'(z)
}{
f_{\theta_t,\sigma^2}(z)\vee \rho/\sigma
}
\right]
}
& \le
\frac{
\abs{
\frac{\dd}{\dd t}
f_{\theta_t,\sigma^2}'(z)
}
}{
f_{\theta_t,\sigma^2}(z)\vee \rho/\sigma
}
\;+\;
\frac{
\abs{f_{\theta_t,\sigma^2}'(z)}
\abs{
\frac{\dd}{\dd t}
f_{\theta_t,\sigma^2}(z)
}
}{
\{f_{\theta_t,\sigma^2}(z)\vee \rho/\sigma\}^2
}\\
 &\le
C|\log\rho|^{3/2}
\p{
\Norm{\eta_1-\widetilde\eta_1}_\infty
+
\Norm{\eta_2-\widetilde\eta_2}_\infty
}.
\end{aligned}
$$
Integrating over \(t\in[0,1]\), multiplying by \(\sigma^2 \le \bar\sigma^2\), and taking suprema over
\(z \in \RR\) and \(\sigma^2\in[\ubar\sigma^2,\bar\sigma^2]\) gives
\begin{equation}
\label{eq:close-reg-bayes-eta-lipschitz}
\Norm{
\delta_{(G,\eta_1,\eta_2)}^\rho
-
\delta_{(G,\widetilde\eta_1,\widetilde\eta_2)}^\rho
}_{\infty}
\le
C|\log\rho|^{3/2}
\p{
\Norm{\eta_1-\widetilde\eta_1}_\infty
+
\Norm{\eta_2-\widetilde\eta_2}_\infty
}.
\end{equation}

\medskip

\noindent\textbf{Step 3: Covering the $G$-coordinate.}
Let $G,G'\in\mathcal G$.
Fix $(\eta_1,\eta_2)\in\mathcal H_n\times\mathcal H_n$ and $\sigma\in[\ubar\sigma,\bar\sigma]$. Then observe that,
$$
\begin{aligned}
&\abs{
\frac{f'_{(G,\eta_1,\eta_2), \sigma^2}(z)}{f_{(G,\eta_1,\eta_2), \sigma^2}(z)\vee \rho/\sigma}
-
\frac{f'_{(G',\eta_1,\eta_2), \sigma^2}(z)}{f_{(G',\eta_1,\eta_2), \sigma^2}(z)\vee \rho/\sigma}
} \\
& \quad\quad
\le \;
\frac{\abs{f'_{(G,\eta_1,\eta_2), \sigma^2}(z)-f'_{(G',\eta_1,\eta_2),\sigma^2}(z)}}{\rho/\bar\sigma}\\
&\qquad+
\frac{\abs{f'_{(G',\eta_1,\eta_2), \sigma^2}(z)}\,\abs{f_{(G,\eta_1,\eta_2), \sigma^2}(z)-f_{(G',\eta_1,\eta_2), \sigma^2}(z)}}{(\rho/\bar\sigma)^2}.
\end{aligned}
$$
Moreover, since $\sup_{\theta,z,\sigma \geq \ubar{\sigma}}|f'_{\theta,\sigma^2}(z)|\le C$. Multiplying by \(\sigma^2\le\bar\sigma^2\) and using
\(0<\rho<1\), we obtain, uniformly over \((\eta_1,\eta_2)\in\mathcal H_n\times\mathcal H_n\),
\begin{equation}
\label{eq:close-reg-bayes-G-lipschitz}
\Norm{
\delta_{(G,\eta_1,\eta_2)}^\rho
-
\delta_{(G',\eta_1,\eta_2)}^\rho
}_{\infty}
\le
C\rho^{-2}\,d_{\mathcal G}(G,G'),
\end{equation}
where we define:
$$
d_{\mathcal G}(G,G')
:=
\sup_{(\eta_1,\eta_2)\in\mathcal H_n\times\mathcal H_n}\;
\sup_{\sigma^2\in[\ubar\sigma^2,\bar\sigma^2]}\;
\sum_{r=0}^1\,
\sup_{z\in\RR}\,
\abs{
\partial_z^r f_{(G,\eta_1,\eta_2),\sigma^2}(z)
-
\partial_z^r f_{(G',\eta_1,\eta_2),\sigma^2}(z)
}.
$$
Our goal in this step is to upper bound the covering number $\log N\p{u,\,\mathcal G,\,d_{\mathcal G}}$ for $ 0<u\le \varepsilon_n$.
using a moment-matching construction in the spirit of
\citet[proof of Theorem~7]{jiang2020general}.

By the Sobolev embedding bound used in the proof of Theorem~\ref{thm:close_contraction},
$
\sup_{\eta\in\mathcal H_n}\lVert\eta\rVert_\infty
\le
C\sqrt n\,\varepsilon_n,
$
and since \(\sqrt\Lambda\) is globally Lipschitz,
$1\vee \sup_{\eta\in\mathcal H_n} \sup_{\sigma\in[\ubar\sigma,\bar\sigma]}s_\eta(\sigma) \le C(1+\sqrt n\,\varepsilon_n).$
Fix \(G=Q*\mathrm N(0,A)\in\mathcal G\), with \(Q\in\mathcal P([-R,R])\) and \(A\in[0,\bar A]\).
For $(\eta_1,\eta_2)\in\mathcal H_n\times\mathcal H_n$, $\sigma\in[\ubar\sigma,\bar\sigma]$,
$z\in\RR$, $A\in[0,\bar A]$, and \(r=0,1\), put
$$
\psi_r(x)
:=
\partial_z^r \varphi\!\p{
z-\eta_1(\sigma)-s_{\eta_2}(\sigma)x;\,\,
\sigma^2+s_{\eta_2}^2(\sigma)A
},
$$
with the dependence on $\eta_1,\eta_2, A, \sigma, z$ omitted. Notice that
$\partial_z^r f_{(G,\eta_1,\eta_2),\sigma^2}(z) = \int \psi_r(x)\,\dd Q(x).$
Moreover, for any $ k \geq 0$ and $r \in \cb{0,1}$
$\partial_x^k \psi_r(x) = \{-s_{\eta_2}(\sigma)\}^k\,\partial_z^{k+r} \varphi(z-\eta_1(\sigma)-s_{\eta_2}(\sigma)x;\,\sigma^2+s_{\eta_2}^2(\sigma)A).$
Meanwhile, by Cramér's inequality, we have that $\sup_y \abs{\partial_y^m\varphi(y;v)}\le C^{m+1}\sqrt{m!}$, uniformly over $v\ge\ubar\sigma^2$. Thus,
\begin{equation}
\abs{\partial_x^k \psi_r(x)} \leq C^{k+1}(1+\sqrt n\,\varepsilon_n)^k\sqrt{(k+r)!}\;,
\label{eq:close-kernel-x-derivative-bound}
\end{equation}
uniformly over $x$ and the other suppressed arguments.

Next, for $u$ sufficiently small, set
$
a:=|\log u|^{1/2}
$
and partition $[-R,R]$ into intervals $I_j$ of length at most $a/(1+\sqrt n\,\varepsilon_n)$,
with left endpoints $x_j$; the number of intervals is at most
$C\{1+(1+\sqrt n\,\varepsilon_n)/a\}$. Let $K:=\lceil C_0a^2\rceil$, where $C_0$ is a
sufficiently large constant. By Carath\'eodory's theorem, applied separately on each interval
$I_j$, there is a discrete probability measure $Q^\circ$ on $[-R,R]$ such that
$$
\int_{I_j}(x-x_j)^k\,\dd Q(x)
=
\int_{I_j}(x-x_j)^k\,\dd Q^\circ(x),
\qquad
k=0,\ldots,2K+1,
$$
for every $j$, and $Q^\circ$ has at most
$$
m\le
C\cb{1+\frac{1+\sqrt n\,\varepsilon_n}{a}}(2K+2)
\le
C(1+\sqrt n\,\varepsilon_n)a^2
$$
support points. The partition and $Q^\circ$ depend only on $(Q,u,n)$, and not on
$(\eta_1,\eta_2,\sigma,z,A,r)$.
For $x\in I_j$, Taylor's theorem and~\eqref{eq:close-kernel-x-derivative-bound} give, uniformly over the suppressed arguments and $r\in\{0,1\}$,
$$
\left|
\psi_r(x)
-
\sum_{k=0}^{2K+1}
\frac{\partial_x^k\psi_r(x_j)}{k!}(x-x_j)^k
\right|
\le
\frac{
C^{2K+3}(1+\sqrt n\,\varepsilon_n)^{2K+2}
\sqrt{(2K+3)!}
}{
(2K+2)!
}
\left(\frac{a}{1+\sqrt n\,\varepsilon_n}\right)^{2K+2}.
$$
Since the polynomial part integrates identically under $Q$ and $Q^\circ$ on each $I_j$, it follows that
$$
\begin{aligned}
\left|
\int \psi_r(x)\,\dd Q(x)
-
\int \psi_r(x)\,\dd Q^\circ(x)
\right|
\le
C
\frac{(Ca)^{2K+2}\sqrt{(2K+3)!}}{(2K+2)!}
\le
C'\sqrt K
\left(\frac{C^2 e a^2}{K}\right)^{K+1}
\le
e^{-K},
\end{aligned}
$$
where the second inequality follows from Stirling's formula and for the third inequality we took $C_0$ sufficiently large in the definition $K=\lceil C_0a^2\rceil$. Now, recalling that
$a^2=|\log u|$, the last display is bounded by $u/4$ (for $u$ small enough). Summing over $r=0,1$, we obtain
\begin{equation}
\label{eq:close-Q-finite-approx}
d_{\mathcal G}\p{Q*\mathrm N(0,A),\,Q^\circ*\mathrm N(0,A)}
\le
u/2.
\end{equation}
It remains to discretize $Q^\circ$ and $A$. Move each support point of $Q^\circ$ to the nearest point on a grid of $[-R,R]$ with grid spacing
$cu/\{1+\sqrt n\,\varepsilon_n\}$.
By~\eqref{eq:close-kernel-x-derivative-bound} with $k=1$, moving a support point by this amount changes $\psi_r$, uniformly over $r=0,1$ and the suppressed arguments, by at most $Ccu$. Hence the resulting perturbation in $d_{\mathcal G}$ is at most $Ccu$.
Next, approximate the vector of weights of the resulting discrete measure in $\ell_1$ to accuracy $cu$. By~\eqref{eq:close-kernel-x-derivative-bound} with $k=0$, this changes $d_{\mathcal G}$ by at most another $Ccu$. Together, these steps yield a discretized approximation $Q'$ of $Q^{\circ}$.

Now, approximate $A\in[0,\bar A]$ as $A'$ on a grid with grid spacing
$cu/\{(1+\sqrt n\,\varepsilon_n)^2\}$.
Indeed,
$\partial_A\psi_r(x) = s_{\eta_2}^2(\sigma)\, \partial_v\partial_z^r \varphi\p{ z-\eta_1(\sigma)-s_{\eta_2}(\sigma)x;\, v }, \;\; v=\sigma^2+s_{\eta_2}^2(\sigma)A.$
Since $v\ge\ubar\sigma^2$, the Gaussian derivatives above are uniformly bounded for $r=0,1$, and hence
$\sup_{r=0,1}\sup_x |\partial_A\psi_r(x)| \le C(1+\sqrt n\,\varepsilon_n)^2.$
Thus the perturbation of $A$ to $A'$ changes $d_{\mathcal G}$ by at most $Ccu$ (for fixed $Q$).

The overall discretization yields $G'=Q'*\mathrm N(0,A') \in \mathcal{G}$.
Taking $c>0$ sufficiently small, combining the above three bounds with~\eqref{eq:close-Q-finite-approx}, and using the triangle inequality, we find that
$d_{\mathcal G}(G,G')\le u.$
It remains to count the number of possible $G'$. The grid for the support points has at most
$
C\sqrt n\varepsilon_n/u
$
points of which we choose $m$, while the $A$-grid has at most $C n \varepsilon_n^2 / u$
points. Moreover, the simplex of probability vectors on at most $m$ points admits an $\ell_1$-cover of radius $cu$ with at most $(C/u)^m$ elements. Consequently, the logarithm of the number of possible discretized laws is bounded by
$
Cm
\left\{|\log u|+\log(\sqrt n\,\varepsilon_n)
\right\}.
$
Using $m\le C\sqrt n\,\varepsilon_n|\log u|$
and the fact that $u\le\varepsilon_n=n^{-\alpha/(2\alpha+1)}$,
we conclude that
\begin{equation}
\log N\p{u,\mathcal G,d_{\mathcal G}}
\le
C\sqrt n\,\varepsilon_n
\log^2 u.
\label{eq:close-G-density-derivative-cover}
\end{equation}

\medskip
\noindent\textbf{Step 4: The product cover.}
Let
$
\tau_n
:=
c_0\varepsilon_n|\log\rho|^{-3/2},
$
where \(c_0>0\) will be chosen small enough. The entropy bound for the Mat\'ern sieve gives
$$
\log N\p{\tau_n,\mathcal H_n,\Norm{\cdot}_\infty}
\le
C\left[
\p{\frac{\delta_n}{\tau_n}}^{1/\beta}
+
\p{\frac{\sqrt n\,\varepsilon_n}{\tau_n}}^{1/(\alpha+1/2)}
\right]
\le
C n\varepsilon_n^2|\log\rho|^{3}.
$$
Here we used
$$
\delta_n\asymp(n\varepsilon_n^2)^{-(\alpha-\beta)}
\asymp\varepsilon_n^{(\alpha-\beta)/\alpha},
\;\;
n\varepsilon_n^2\asymp\varepsilon_n^{-1/\alpha},
$$
and \(\beta>1/2\).

Let \(\{\eta_{1,k}:1\le k\le N_1\}\) and \(\{\eta_{2,\ell}:1\le \ell\le N_2\}\) be \(\tau_n\)-nets of
\(\mathcal H_n\) in sup norm, and let \(\{G_m:1\le m\le N_G\}\) be a
\(c_1\rho^2\varepsilon_n\)-net of \((\mathcal G, d_{\mathcal G})\), where \(c_1>0\) will be chosen
small enough; note that a single net of \(\mathcal G\) serves all pairs \((k,\ell)\), by the
definition of \(d_{\mathcal G}\). By \eqref{eq:close-G-density-derivative-cover} applied with
\(u=c_1\rho^2\varepsilon_n\), and using \(1+\sqrt n\,\varepsilon_n\le C n\varepsilon_n^2\),
$$
\log N_G
\le
C n\varepsilon_n^2
\cb{1+\log\frac{1}{\rho^2\varepsilon_n}}^{2}.
$$
Now take any \(\theta=(G,\eta_1,\eta_2)\in\Theta_n\). Choose \(k,\ell\) such that
\(\Norm{\eta_1-\eta_{1,k}}_\infty\le\tau_n\) and \(\Norm{\eta_2-\eta_{2,\ell}}_\infty\le\tau_n\), and
then \(m\) such that \(d_{\mathcal G}(G,G_m)\le c_1\rho^2\varepsilon_n\). By the triangle inequality,
\eqref{eq:close-reg-bayes-eta-lipschitz}, and \eqref{eq:close-reg-bayes-G-lipschitz},
$$
\Norm{
\delta_\theta^\rho
-
\delta_{(G_m,\eta_{1,k},\eta_{2,\ell})}^\rho
}_{\infty}
\le
C|\log\rho|^{3/2}\cdot 2\tau_n
+
C\rho^{-2}\cdot c_1\rho^2\varepsilon_n
\le
c\varepsilon_n,
$$
after decreasing \(c_0\) and \(c_1\), if necessary. Hence
$\{\delta_{(G_m,\eta_{1,k},\eta_{2,\ell})}^\rho : k\le N_1,\,\ell\le N_2,\,m\le N_G\}$
is a $c\varepsilon_n$-net of $\{\delta_\theta^\rho:\theta\in\Theta_n\}$ under
\(\Norm{\cdot}_{\infty}\). Since \(|\log\rho|\le\log\{1/(\rho^2\varepsilon_n)\}\),
$
\log N_1+\log N_2+\log N_G
\le
C n\varepsilon_n^2
\{1-\log(\rho^2\varepsilon_n)\}^{3},
$
which proves the first claim. If \(\rho=\rho_n\ge n^{-q}\), then
\(1+\log\{1/(\rho_n^2\varepsilon_n)\}\le C_q\log n\), and the second display follows.
\end{proof}

\begin{lemm}
\label{lem:close-crude-moments}
Under Specification~\ref{loc-scale-prior}, there is a constant \(C<\infty\), depending only on the
specification and on \(\ubar\sigma^2,\bar\sigma^2\), such that
$\widehat\mu_i^2(\Pi) \le \EE[\Pi]{\mu_i^2\mid \boldZ,\boldsigma^2} \le C\{n+\Norm{\boldZ}_2^2\}, \; i=1,\ldots,n .$
\end{lemm}

\begin{proof}
The first inequality follows from Jensen's inequality.  We prove the second.
Use the leave-one-out representation of the full Bayes rule analogous to the one stated in Proposition~\ref{prop:datta_loo}. Let $\overline{G}_{i}$ be the posterior mean of the law of
$\mu_i$ conditional on
$(\boldZ_{-i},\boldsigma^2)$. Then
$
\EE[\Pi]{\mu_i^2\mid\boldZ,\boldsigma^2}
=
\EE[\overline{G}_{i}]{\mu_i^2\mid Z_i,\sigma_i^2}$.
Let
\smash{$
f_{\overline{G}_{i}, \sigma_i^2}(z)
=
\int \varphi(z-\mu;\sigma_i^2)\,\dd\overline{G}_{i}(\mu)
$}
be the corresponding marginal (predictive) density.

We first lower bound $f_{\overline{G}_{i},\sigma_i^2}(Z_i)$. Choose a fixed event $E_0 \subset \Theta$ in the parameter
space on which
$
\|\eta_1\|_\infty\le 1,\,
\|\eta_2\|_\infty\le 1,\,
A\le a_0
$
for some fixed \(a_0>0\) with \(\Gamma(A\le a_0)>0\). The Gaussian process support and
the lower-tail assumption on \(\Gamma\) imply \(\Pi(E_0)>0\). Since the base measure \(H\) is
supported on \([-R,R]\), on \(E_0\) each induced law assigns at least a fixed
positive amount of mass $\delta >0$ to a fixed compact interval \([-B_0,B_0]\), uniformly in \(j\), that is,
$\Pi( \{\mu_i \in [-B_0, B_0]\} \mid \theta) \geq \delta\, \text{ for all }\, i=1,\ldots,n\; \text{ and }\; \theta \in E_0.$
Hence there
are constants $c,C>0$ such that
$$
f_{\theta,\sigma_i^2}(z)\ge c\exp(-Cz^2)\, \text{ for all }\, i=1,\ldots,n\; \text{ and }\; \theta \in E_0.
$$
Also $f_{\theta,\sigma_i^2}(z)\le (2\pi\ubar\sigma^2)^{-1/2}$ for all $\theta,i,z$. Bayes' formula based
on \((\boldZ_{-i},\boldsigma^2)\) gives
$$
\begin{aligned}
\Pi(E_0\mid\boldZ_{-i},\boldsigma^2) & = \frac{\int_{E_0} \prod_{j \neq i} f_{\theta,\sigma_j^2}(Z_j)\dd \Pi(\theta)}{\int \prod_{j \neq i} f_{\theta,\sigma_j^2}(Z_j)\dd \Pi(\theta)} \\ &\geq \frac{ \Pi(E_0) c^{n-1} \exp(-C \Norm{\boldZ_{-i}}_2^2)}{ (2\pi\ubar\sigma^2)^{-(n-1)/2}}
\ge
\exp\{-C'(n+\Norm{\boldZ_{-i}}_2^2)\},
\end{aligned}
$$
for another constant $C'$.
It follows that \(\overline G_{i}\) puts at least
\(\exp\{-C'(n+\Norm{\boldZ_{-i}}_2^2)\}\) mass on \([-B_0,B_0]\), and therefore
$$
f_{\overline{G}_i, \sigma_i^2}(Z_i)
\ge
\exp\{-C''(n+\Norm{\boldZ_{-i}}_2^2+Z_i^2)\} = \exp\{-C''(n+\Norm{\boldZ}_2^2)\},
$$
for yet another constant $C''$. Using Lemma~\ref{lemm:proposition_1_jiang_zhang} (similar to our argument for the proof of Lemma~\ref{lemm:posterior_moments_upper_bounded}), we find that,
$$
\EE[\Pi]{\mu_i^2\mid\boldZ,\boldsigma^2} = \EE[\overline{G}_i]{\mu_i^2 \mid Z_i, \sigma_i^2} \lesssim 1 + Z_i^2 + \abs{\log(f_{\overline{G}_i, \sigma_i^2}(Z_i))} \lesssim n + \Norm{\boldZ}_2^2,
$$
as claimed.
\end{proof}

\begin{lemm}
\label{lem:close-average-to-coordinate}
Suppose Assumptions~\ref{assu:LS} and~\ref{assu:close-fixed-design} hold. Let $\Theta_n=\mathcal G\otimes\mathcal H_n\otimes\mathcal H_n$ be the
sieve defined in~\eqref{eq:close-sieve}.
For every $D<\infty$,
$$
\sup_{\theta\in\Theta_n:
\,\Dhel(\boldf_\star,\boldf_\theta)\le D\varepsilon_n(\Pi)}
\max_{1\le i\le n}
\Dhel(f_{\star,\sigma_i^2},f_{\theta,\sigma_i^2})
=o(1),
$$
uniformly over $(G_\star,m_\star,s_\star)$ satisfying
Assumption~\ref{assu:LS}.
\end{lemm}

\begin{proof}
Write $\varepsilon_n:=\varepsilon_n(\Pi)$ and
$h_n=n^{-1/(2\alpha+1)}$, so that
$\varepsilon_n=h_n^\alpha$ and $\sqrt n\varepsilon_n=h_n^{-1/2}$.
Also fix
$
\tilde\alpha:=\{1/2+(\alpha\wedge1)\}/2 \in (1/2, \alpha \wedge 1).
$
Recall that, for some $1/2<\beta<\alpha$, the sieve for either nuisance
function is
$$
\mathcal H_n
=
\delta_n H_1^\beta
+
C\sqrt n\,\varepsilon_n H_1^{\alpha+1/2},
\qquad
\delta_n\asymp(n\varepsilon_n^2)^{-(\alpha-\beta)}=o(1).
$$
By Sobolev embedding, uniformly over $\eta\in\mathcal H_n$,
$
\abs{\eta(\sigma)-\eta(\widetilde\sigma)}
\lesssim
\delta_n+
\sqrt n\,\varepsilon_n
\abs{\sigma-\widetilde\sigma}^{\tilde\alpha}.
$
Since $\sqrt{\Lambda}$ is globally Lipschitz, the same bound holds with
$\eta$ replaced by $s_\eta$. Moreover, Assumption~\ref{assu:LS} gives
\smash{$
\abs{m_\star(\sigma)-m_\star(\widetilde\sigma)}
\lesssim
\abs{\sigma-\widetilde\sigma}^{\tilde\alpha},
$}
and, using the approximant $\eta_{2,n}$ from
Assumption~\ref{assu:LS}(iii),
$\abs{s_\star(\sigma)-s_\star(\widetilde\sigma)} \lesssim \varepsilon_n+ \sqrt n\,\varepsilon_n \abs{\sigma-\widetilde\sigma}^{\tilde\alpha}.$
For $\theta\in\Theta_n$, define
$
d_\theta(\sigma):=
\Dhel\p{f_{\star,\sigma^2},f_{\theta,\sigma^2}}.
$
We claim that, uniformly over $\theta\in\Theta_n$,
$\abs{d_\theta(\sigma)-d_\theta(\widetilde\sigma)} \lesssim \delta_n+\varepsilon_n+ (1+\sqrt n\,\varepsilon_n) \abs{\sigma-\widetilde\sigma}^{\tilde\alpha}.$
To see this, write $G=Q*\mathrm N(0,A)$, so that
$$
f_{\theta,\sigma^2}(z)
=
\int
\varphi\p{
z-\eta_1(\sigma)-s_{\eta_2}(\sigma)u;\,
\sigma^2+A s_{\eta_2}^2(\sigma)
}
\,\dd Q(u),
$$
where $Q$ is supported on $[-R,R]$. For normal distributions whose
standard deviations are bounded away from zero,
\smash{$
\Dhel(\mathrm N(a,v^2),\mathrm N(b,w^2))
\lesssim
\abs{a-b}+\abs{v-w}.
$}
Moreover,
$$
\abs{
\sqrt{\sigma^2+A s^2}
-
\sqrt{\widetilde\sigma^2+A\widetilde s^2}
}
\le
\abs{\sigma-\widetilde\sigma}
+
\sqrt A\,\abs{s-\widetilde s}.
$$
Convexity of squared Hellinger distance under mixtures therefore gives
the claimed modulus for the candidate marginals. The same argument,
using the bounds on $(m_\star,s_\star)$ above and the compact support of
$G_\star$, gives the corresponding modulus for the true marginals.
The claim then follows from the triangle inequality.

In particular, uniformly over $\theta\in\Theta_n$,
$$
\sup_{\abs{\sigma-\widetilde\sigma}\le h_n}
\abs{d_\theta(\sigma)-d_\theta(\widetilde\sigma)}
\lesssim
\delta_n+\varepsilon_n+h_n^{\tilde\alpha}
+
\sqrt n\,\varepsilon_n h_n^{\tilde\alpha}
=o(1),
$$
because
$
\sqrt n\,\varepsilon_n h_n^{\tilde\alpha}
=
h_n^{\tilde\alpha-1/2}\to0.
$
Denote the left-hand side above by $\omega_n$, so that
$\omega_n=o(1)$.

Now fix $\theta\in\Theta_n$ such that
$\Dhel(\boldf_\star,\boldf_\theta)\le D\varepsilon_n$, and choose
$i_0$ such that
$
d_\theta(\sigma_{i_0})=\max_i d_\theta(\sigma_i).
$
For every $j$ with
$\abs{\sigma_j-\sigma_{i_0}}\le h_n$,
$
d_\theta(\sigma_j)
\ge
\p{d_\theta(\sigma_{i_0})-\omega_n}_+.
$
Hence, by Assumption~\ref{assu:close-fixed-design},
$D^2\varepsilon_n^2 \ge \frac1n\sum_{j=1}^n d_\theta^2(\sigma_j) \ge c h_n \p{d_\theta(\sigma_{i_0})-\omega_n}_+^2.$
It follows that
$$
\max_{1\le i\le n}
\Dhel(f_{\star,\sigma_i^2},f_{\theta,\sigma_i^2})
\le
\omega_n+\frac{D\varepsilon_n}{\sqrt{c h_n}}
=
o(1)+\frac{D}{\sqrt c}h_n^{\alpha-1/2}
=o(1),
$$
since $\alpha>1/2$, with the $o(1)$ term uniform in the relevant quantities.
\end{proof}

\subsection{Proof for regret for CLOSE (Theorem~\ref{theo:close-regret})}
\label{subsec:close-regret}

\begin{proof}
Write $\varepsilon_n:=\varepsilon_n(\Pi)$. Throughout the proof, all probabilities and expectations are conditional on the fixed design $\boldsigma^2$. By Assumption~\ref{assu:LS}, we may take $M_1=B(1+M)$ so that
$
|\mu_i|\le M_1
$
almost surely under the true model, uniformly over $i$. For $\theta=(G,\eta_1,\eta_2)$, write
$
\delta_\theta(z,\sigma^2):=
z+\sigma^2f'_{\theta,\sigma^2}(z)/f_{\theta,\sigma^2}(z)
$
and similarly let
$
\delta_\star(z,\sigma^2):=
z+\sigma^2f'_{\star,\sigma^2}(z)/f_{\star,\sigma^2}(z).
$
Thus
$
\delta_\star(Z_i,\sigma_i^2)=\EE[\star]{\mu_i\mid Z_i,\sigma_i^2}
$
and let
$
\widehat{\boldmu}^{\star}:=(\delta_\star(Z_i,\sigma_i^2))_{i=1}^n.
$\\

\noindent\textbf{Step 1:}
Since the coordinates are conditionally independent under the true hierarchy,
$
\widehat{\boldmu}^{\star}
=
\EE[\star]{\boldmu\mid\boldZ,\boldsigma^2}.
$
Therefore, by orthogonality,
$$
R^2\p{\widehat{\boldmu}(\Pi),(G_\star,m_\star,s_\star)}
=
R^2\p{\widehat{\boldmu}^{\star},(G_\star,m_\star,s_\star)}
+
\frac1n
\EE[\star]{
\Norm{\widehat{\boldmu}(\Pi)-\widehat{\boldmu}^{\star}}_2^2
}.
$$
Moreover,
$
R^2\p{\widehat{\boldmu}^{\star},(G_\star,m_\star,s_\star)}
=
n^{-1}\sum_{i=1}^n
\EE[\star]{\Var[\star]{\mu_i\mid Z_i,\sigma_i^2}}.
$
Thus, using $\sqrt{a+b}-\sqrt a\le\sqrt b$, it suffices to show that
\begin{equation}
\label{eq:close-oracle-discrepancy-goal}
\frac1n\sum_{i=1}^n
\EE[\star]{
\p{
\widehat\mu_i(\Pi)-\delta_\star(Z_i,\sigma_i^2)
}^2
}
\lesssim
\varepsilon_n^2\log^4 n.
\end{equation}

\noindent \textbf{Step 2:}
Let $\Theta_n$ denote the sieve in~\eqref{eq:close-sieve}. By
Theorem~\ref{thm:close_contraction}, together with the sieve-complement
bound in its proof, there exist constants $D,c>0$ such that, for every
$q\geq1$, there is an event $\mathcal C_n$ satisfying
$
\PP[\star]{\mathcal C_n^c}\le C_qn^{-q}
$, and, on $\mathcal C_n$,
$
\Pi\p{\mathcal B_n^c\mid\boldZ,\boldsigma^2}
\le
\exp\{-cn\varepsilon_n^2\},
$
where
$
\mathcal B_n:=\cb{\theta\in\Theta_n:\Dhel(\boldf_\star,\boldf_\theta)\le D\varepsilon_n}.
$
We next use Lemma~\ref{lem:hellinger-compactness}. Recall that the true
conditional law of $\mu_i$ is supported on $[-M_1,M_1]$, uniformly in
$i$. Applying Lemma~\ref{lem:hellinger-compactness} with $n=1$, there
exist constants $\rho_0>0$ and $M_2<\infty$ such that, uniformly over
$\sigma\in[\ubar\sigma,\bar\sigma]$,
$\Dhel\p{f_{\star,\sigma^2},f_{\theta,\sigma^2}}\le\rho_0 \;\Longrightarrow\; \PP[U\sim G]{ \abs{\eta_1(\sigma)+s_{\eta_2}(\sigma)U}\le M_2 } \ge\frac12.$
Here the two mixing distributions in Lemma~\ref{lem:hellinger-compactness}
are the laws of
$
m_\star(\sigma)+s_\star(\sigma)U$, $U\sim G_\star,
$
and
$
\eta_1(\sigma)+s_{\eta_2}(\sigma)U$, $U\sim G,
$
respectively.

By Lemma~\ref{lem:close-average-to-coordinate}, for all sufficiently
large $n$,
$
\max_{1\le i\le n}
\Dhel(f_{\star,\sigma_i^2},f_{\theta,\sigma_i^2})
\le\rho_0
$
for every $\theta\in\mathcal B_n$. Consequently,
\begin{equation}
\label{eq:close-coordinate-compactness}
\PP[U\sim G]{
\abs{\eta_1(\sigma_i)+s_{\eta_2}(\sigma_i)U}\le M_2
}
\ge\frac12
\end{equation}
for every $i=1,\ldots,n$ and every $\theta\in\mathcal B_n$.

Finally, let
$
T_n:=M_1+\bar\sigma\sqrt{16\log n}
$
and
$
\mathcal A_n:=\cb{\max_{1\le i\le n}|Z_i|\le T_n}.
$
Since $|\mu_i|\le M_1$ almost surely under the true model, a Gaussian
tail bound gives
$
\PP[\star]{\mathcal A_n^c}\le2n^{-7}
$.\\

\noindent\textbf{Step 3:}
We first control the contribution from $\mathcal A_n^c\cup\mathcal C_n^c$. Recall from Lemma~\ref{lem:close-crude-moments} that
$
\widehat\mu_i^2(\Pi)
\lesssim
n+\Norm{\boldZ}_2^2,
$
while $|\delta_\star(Z_i,\sigma_i^2)|\le M_1$. Hence
$
(\widehat\mu_i(\Pi)-\delta_\star(Z_i,\sigma_i^2))^2
\lesssim
n+\Norm{\boldZ}_2^2.
$
Moreover, by Assumption~\ref{assu:LS},
$\EEInline[\star]{(n+\Norm{\boldZ}_2^2)^2}\lesssim n^2$.
Taking $q$ sufficiently large above and applying Cauchy--Schwarz gives
$$
\begin{aligned}
\frac1n\sum_{i=1}^n
\EE[\star]{
\p{
\widehat\mu_i(\Pi)-\delta_\star(Z_i,\sigma_i^2)
}^2
\ind{\mathcal A_n^c\cup\mathcal C_n^c}
}
\lesssim
n\,\PP[\star]{\mathcal A_n^c\cup\mathcal C_n^c}^{1/2}
\lesssim n^{-5/2},
\end{aligned}
$$
which is negligible relative to $\varepsilon_n^2\log^4 n$.\\

\noindent\textbf{Step 4:}
For every $i$, by iterated expectation,
$
\widehat\mu_i(\Pi)=\int
\delta_\theta(Z_i,\sigma_i^2)
\,\dd\Pi(\theta\mid\boldZ,\boldsigma^2).
$
On $\mathcal C_n$, let
$
\pi_n:=\Pi(\mathcal B_n^c\mid\boldZ,\boldsigma^2)
$
and let $\Pi_{\mathcal B_n}$ and $\Pi_{\mathcal B_n^c}$ denote the posterior restricted and normalized to $\mathcal B_n$ and $\mathcal B_n^c$, respectively. Then
$$
\begin{aligned}
\p{
\widehat\mu_i(\Pi)-\delta_\star(Z_i,\sigma_i^2)
}^2
&\le
2
\p{
\int_{\mathcal B_n}
\delta_\theta(Z_i,\sigma_i^2)
\,\dd\Pi_{\mathcal B_n}(\theta\mid\boldZ,\boldsigma^2)
-
\delta_\star(Z_i,\sigma_i^2)
}^2\\
&\quad+
2\pi_n^2
\p{
\int_{\mathcal B_n^c}
\delta_\theta(Z_i,\sigma_i^2)
\,\dd\Pi_{\mathcal B_n^c}(\theta\mid\boldZ,\boldsigma^2)
-
\delta_\star(Z_i,\sigma_i^2)
}^2.
\end{aligned}
$$
By Jensen's inequality, the first term is at most
$$
2\int_{\mathcal B_n}
\p{
\delta_\theta(Z_i,\sigma_i^2)
-
\delta_\star(Z_i,\sigma_i^2)
}^2
\,\dd\Pi_{\mathcal B_n}(\theta\mid\boldZ,\boldsigma^2).
$$
For the second term, using Jensen's inequality and $|\delta_\star|\le M_1$,
$$
\begin{aligned}
&\pi_n^2
\p{
\int_{\mathcal B_n^c}
\delta_\theta(Z_i,\sigma_i^2)
\,\dd\Pi_{\mathcal B_n^c}(\theta\mid\boldZ,\boldsigma^2)
-
\delta_\star(Z_i,\sigma_i^2)
}^2 \\
&\quad\lesssim\;
\pi_n^2
+
\pi_n
\int_{\mathcal B_n^c}
\delta_\theta^2(Z_i,\sigma_i^2)
\,\dd\Pi(\theta\mid\boldZ,\boldsigma^2)\,\le\,
\pi_n
\cb{
1+\EE[\Pi]{\mu_i^2\mid\boldZ,\boldsigma^2}
}.
\end{aligned}
$$
On $\mathcal C_n$, $\pi_n\le\exp\{-cn\varepsilon_n^2\}$. By Lemma~\ref{lem:close-crude-moments}, its contribution on $\mathcal A_n$ is therefore bounded by
$
Cn\log n\exp\{-cn\varepsilon_n^2\}
$
and is negligible. Together with Step 3, it remains to show that
\begin{equation}
\label{eq:close-good-posterior-sup-goal}
\EE[\star]{
\sup_{\theta\in\mathcal B_n}
\frac1n\sum_{i=1}^n
\p{
\delta_\theta(Z_i,\sigma_i^2)
-
\delta_\star(Z_i,\sigma_i^2)
}^2
\ind{(\mathcal A_n)}
}
\lesssim
\varepsilon_n^2\log^4 n.
\end{equation}

\noindent\textbf{Step 5:}
We next replace the ordinary Bayes rules $\delta_{\theta}(z, \sigma^2)$ by the regularized Bayes rules $\delta_\theta^\rho(z,\sigma^2)$ defined in~\eqref{eq:reg_bayes_rule_close}. Set $\rho_n:=n^{-K}$ for a sufficiently large fixed $K$. Recalling~\eqref{eq:close-coordinate-compactness}, we find that on $\mathcal{A}_n$, for every $\theta\in\mathcal B_n$ and every $i$,
$$
\begin{aligned}
f_{\theta,\sigma_i^2}(Z_i)
=
\int
\varphi\p{
Z_i-\eta_1(\sigma_i)-s_{\eta_2}(\sigma_i)u;\sigma_i^2
}
\,\dd G(u)
\ge
\frac{1}{2\sqrt{2\pi}\bar\sigma}
\exp\p{
-\frac{(T_n+M_2)^2}{2\ubar\sigma^2}
}
\ge
\frac{\rho_n}{\sigma_i},
\end{aligned}
$$
for $K$ sufficiently large. Since the true conditional law of $\mu_i$ is supported on $[-M_1,M_1]$, the same argument gives
$
f_{\star,\sigma_i^2}(Z_i)\ge\rho_n/\sigma_i.
$
Consequently, on $\mathcal A_n$,
$
\delta_\theta(Z_i,\sigma_i^2)
=
\delta_\theta^{\rho_n}(Z_i,\sigma_i^2)$ and
$
\delta_\star(Z_i,\sigma_i^2)
=
\delta_\star^{\rho_n}(Z_i,\sigma_i^2),
$
simultaneously for every $i$ and every $\theta\in\mathcal B_n$.
Define
$$
\widehat\Delta_n(\theta)
:=
\frac1n\sum_{i=1}^n
\p{
\delta_\theta^{\rho_n}(Z_i,\sigma_i^2)
-
\delta_\star^{\rho_n}(Z_i,\sigma_i^2)
}^2.
$$
It therefore suffices to show
$
\EEInline[\star]{\sup_{\theta\in\mathcal B_n}\widehat\Delta_n(\theta)}
\lesssim
\varepsilon_n^2\log^4 n.
$

\noindent\textbf{Step 6:}
By Lemma~\ref{lemm:covering_number_reg_bayes_close}, there exist $\theta_1,\ldots,\theta_N\in\mathcal B_n$ such that
$
\log N\lesssim n\varepsilon_n^2\log^3 n
$
and, for every $\theta\in\mathcal B_n$, there exists $j=j(\theta)$ satisfying
$
\lVert\delta_\theta^{\rho_n}-\delta_{\theta_j}^{\rho_n}\rVert_\infty
\lesssim
\varepsilon_n.
$
Indeed, we may start with a cover of $\Theta_n$, discard balls that do not intersect $\mathcal B_n$, and choose one point of $\mathcal B_n$ from each remaining ball, at the cost of changing the covering radius by a constant factor. Hence
$\sup_{\theta\in\mathcal B_n} \widehat\Delta_n(\theta) \le 2\max_{1\le j\le N}\widehat\Delta_n(\theta_j)+C\varepsilon_n^2.$
Fix $j\le N$ and define
$$
X_{ij}:=\p{
\delta_{\theta_j}^{\rho_n}(Z_i,\sigma_i^2)
-
\delta_\star^{\rho_n}(Z_i,\sigma_i^2)
}^2.
$$
By Lemma~\ref{lemm:proposition_1_jiang_zhang}, after rescaling by $\sigma_i$,
$
0\le X_{ij}\lesssim|\log\rho_n|\lesssim\log n.
$
Let
$
h_{ij}:=\Dhel(f_{\star,\sigma_i^2},f_{\theta_j,\sigma_i^2}).
$
Lemma~\ref{lemm:saha_e1}, again applied after rescaling by $\sigma_i$, gives
$\EE[\star]{X_{ij}} \lesssim h_{ij}^2 \max\cb{ |\log\rho_n|^3,\, |\log h_{ij}| }.$
Since $\theta_j\in\mathcal B_n$, $n^{-1}\sum_{i=1}^n h_{ij}^2\le D^2\varepsilon_n^2$.
Also, for all sufficiently large $n$, the function $x\mapsto x\log(1/x)$ is increasing over the relevant range. By concavity and Jensen's inequality,
$$
\begin{aligned}
\frac1n\sum_{i=1}^n
h_{ij}^2|\log h_{ij}|
=
\frac{1}{2n}\sum_{i=1}^n
h_{ij}^2\log\p{\frac{1}{h_{ij}^2}}
\le
\frac12
\p{\frac1n\sum_{i=1}^n h_{ij}^2}
\log\p{
\frac{1}{n^{-1}\sum_{i=1}^n h_{ij}^2}
}
\lesssim
\varepsilon_n^2\log n.
\end{aligned}
$$
It follows that
\begin{equation}
\label{eq:close-net-mean}
\frac1n\sum_{i=1}^n
\EE[\star]{X_{ij}}
\lesssim
\varepsilon_n^2\log^3 n.
\end{equation}
Moreover, since $X_{ij}\lesssim\log n$,
$$
\begin{aligned}
\sum_{i=1}^n
\Var[\star]{X_{ij}}
\le
\sum_{i=1}^n
\EE[\star]{X_{ij}^2}
\lesssim
\log n
\sum_{i=1}^n
\EE[\star]{X_{ij}}
\lesssim
n\varepsilon_n^2\log^4 n.
\end{aligned}
$$
Bernstein's inequality therefore gives, for a sufficiently large constant $A$,
$$
\PP[\star]{
\widehat\Delta_n(\theta_j)
>
A\varepsilon_n^2\log^4 n
}
\le
\exp\p{
-cA n\varepsilon_n^2\log^3 n
}.
$$
Taking a union bound over $j=1,\ldots,N$, and increasing $A$ if necessary, yields
$$
\PP[\star]{
\max_{1\le j\le N}
\widehat\Delta_n(\theta_j)
>
A\varepsilon_n^2\log^4 n
}
\le
\exp\p{
-cn\varepsilon_n^2\log^3 n
}.
$$
Finally,
$
\widehat\Delta_n(\theta_j)\lesssim\log n
$
uniformly over $j$, again by the regularized score bound. Thus the preceding tail probability implies
$
\EEInline[\star]{\max_{1\le j\le N}\widehat\Delta_n(\theta_j)}
\lesssim
\varepsilon_n^2\log^4 n.
$
Combining this with the covering bound gives
$
\mathbb E_{\star}[
\sup_{\theta\in\mathcal B_n}
\widehat\Delta_n(\theta)
]
\lesssim
\varepsilon_n^2\log^4 n.
$
By Step 5, this proves~\eqref{eq:close-good-posterior-sup-goal}. Combining Steps 3--6 proves~\eqref{eq:close-oracle-discrepancy-goal}, and the claim now follows from Step 1 by taking square roots.
\end{proof}

\subsection{Proofs for the mean-variance independent heteroscedastic model}

\begin{proof}[Proof of Theorem~\ref{theo:het-ind-contraction}]
Throughout, we condition on the fixed variances $\boldsigma^2$. The proof follows the steps in the proof of Theorem~\ref{thm:compound-contraction}, and we only record substantial changes.

Let
$
L_n(\boldf_G):=\prod_{i=1}^n
f_{G,\sigma_i^2}(Z_i)/f_{G_\star,\sigma_i^2}(Z_i).
$
In contrast to the compound setting, this is a genuine likelihood ratio. Take the local set $\mathcal K_n$ supplied by Lemma~\ref{lem:G-small-k} with $m_\star\equiv0$ and $s_\star\equiv1$. The Bernstein argument as in Lemma~\ref{lem:denom-bernstein}, with $\DKL{\cdot}{\cdot}$ and $\VKL{\cdot}{\cdot}$  replaced by their vectorized counterparts in~\eqref{eq:vectorized_metrics}, gives the required denominator lower bound. Lemma~\ref{lem:G-small-k} gives $\Pi(\mathcal K_n)\ge\exp\{-Cn\varepsilon_n^2\}$.
Next, we can follow Lemma~\ref{lemm:upper_bound_small} almost verbatim, again with some simplifications because we are working with genuine likelihood ratios. We replace \citet[Lemma 2]{zhang2009generalized} for constructing the covering by~\citet[Lemma 4]{jiang2020general}. Combining the denominator and numerator bounds exactly as in Step~3 of the proof of Theorem~\ref{thm:compound-contraction} proves the overall claim.
\end{proof}

\begin{proof}[Proof of Theorem~\ref{theo:regret_bayes_ind_gauss}]
The proof follows the proof of Theorem~\ref{theo:main_risk_bound_eq} (i.e., the same $6$ steps); and so again we only sketch differences.

Step 1, i.e., the reduction to permutation-equivariant rules to the separable oracle, is not needed. We replace Step 2 by the following argument.
Since
$
\widehat{\boldmu}^{\mathrm B}(G_\star)=\EE[G_\star]{\boldmu\mid\boldZ,\boldsigma^2},
$
orthogonality gives

$$
R^2(\widehat{\boldmu}(\Pi),G_\star)
=
R^2(\widehat{\boldmu}^{\mathrm B}(G_\star),G_\star)
+
\frac1n
\EE[G_\star]{
\Norm{\widehat{\boldmu}(\Pi)-\widehat{\boldmu}^{\mathrm B}(G_\star)}_2^2
}.
$$
Using the triangle inequality, it therefore suffices to bound the last term by
$C\varepsilon_n^2\log^3 n$.

Steps 3-6 can also be applied almost verbatim.
The main changes are that we use the heteroscedastic regularized rules from~\eqref{eq:reg_bayes_rule_close}, with threshold $\rho/\sigma$, and cover them using Theorem 7 of~\citet{jiang2020general} (in lieu of using Proposition~3 of \citet{jiang2009general} which we used for the proof of Theorem~\ref{theo:main_risk_bound_eq}).

\end{proof}

\end{document}